\documentclass{amsart}
\usepackage{amsmath}
\usepackage{graphicx}
\usepackage{amsfonts}
\usepackage{amssymb}
\usepackage{mathtools}
\usepackage{pdfsync}
\usepackage{color, soul}
\usepackage{mathrsfs}
\usepackage{epsfig}
\usepackage{url}
\usepackage{mathrsfs,a4wide}
\usepackage{calligra}
\usepackage{bm}
\usepackage{tikz}

\theoremstyle{plain}
\newtheorem{theorem}{Theorem}[section]

\newtheorem{lemma}[theorem]{Lemma}
\newtheorem{proposition}[theorem]{Proposition}
\newtheorem{corollary}[theorem]{Corollary}
\newtheorem{remark}[theorem]{Remark}
\newtheorem{definition}[theorem]{Definition}
\theoremstyle{definition}
\theoremstyle{remark}
\numberwithin{equation}{section}

\newcommand{\average}{{\mathchoice {\kern1ex\vcenter{\hrule
height.4pt width 8pt depth0pt}
\kern-11pt} {\kern1ex\vcenter{\hrule height.4pt width 4.3pt
depth0pt} \kern-7pt} {} {} }}

\newcommand{\I}{\mathcal{I}}

\newcommand{\Airy}{\mathsf{A}}
\newcommand{\Bb}{\mathscr{B}}
\newcommand{\reg}{\mathscr{R}}
\newcommand{\Dive}{\mathrm{Div}\,}
\newcommand{\de}{\bm{\delta}}

\newcommand{\skw}{\mathrm{skew}}

\newcommand{\sym}{\mathrm{sym}}

\newcommand{\ED}{\mathscr{ED}}
\newcommand{\WD}{\mathscr{WD}}
\newcommand{\vk}{\varkappa}
\newcommand{\Dcal}{\mathcal{D}}

\newcommand{\ep}{\varepsilon}

\newcommand{\ffi}{\varphi}

\newcommand{\C}{\mathbb{C}}
\newcommand{\R}{\mathbb{R}}
\newcommand{\N}{\mathbb{N}}
\newcommand{\Z}{\mathbb{Z}}

\newcommand{\E}{\mathcal{E}}
\newcommand{\F}{\mathcal{F}}

\newcommand{\di}{\mathrm{dist}}

\newcommand{\ud}{\mathrm{d}}

\newcommand{\supp}{\mathrm{spt}\,}

\newcommand{\Div}{\mathrm{Div}}

\newcommand{\Huno}{\mathcal H^1}

\newcommand{\weakly}{\rightharpoonup}           
\newcommand{\weakstar}{\stackrel{*}{\weakly}}   

\newcommand{\loc}{\mathrm{loc}}

\newcommand{\dist}{\mathrm{dist}}

\newcommand{\G}{\mathcal{G}}

\newcommand{\ce}{\ep}

\def\red#1{\textcolor{red}{#1}}

\newcommand{\blu}[1]{\textcolor[rgb]{0,0,0}{#1}}

\def\XXint#1#2#3{{\setbox0=\hbox{$#1{#2#3}{\int}$}
     \vcenter{\hbox{$#2#3$}}\kern-.5\wd0}}

\newcommand{\res}{\mathop{\hbox{\vrule height 7pt width .5pt depth 0pt
\vrule height .5pt width 6pt depth 0pt}}\nolimits}
\usepackage{latexsym}
\newcommand{\newatop}{\genfrac{}{}{0pt}{1}}

\makeatletter
\def\@splitop#1#2\@nil{$\mathscr{#1}\!\!$\calligra#2\,\,}
\newcommand*\DeclareCursiveOperator[2]{%

 \newcommand#1{\mathop{\mbox{\@splitop#2\@nil}}\nolimits}}
\makeatother
\DeclareCursiveOperator{\Anew}{A}
\DeclareCursiveOperator{\Bnew}{B}
\DeclareCursiveOperator{\Cnew}{C}
\DeclareCursiveOperator{\Dnew}{D}
\DeclareCursiveOperator{\Enew}{E}
\DeclareCursiveOperator{\Qnew}{Q}
\DeclareCursiveOperator{\Tnew}{T}
\DeclareMathOperator{\ccurl}{CURL}
\DeclareMathOperator{\curl}{curl}
\DeclareMathOperator{\Curl}{Curl}
\DeclareMathOperator{\CURL}{\textsc{curl}}
\DeclareMathOperator{\INC}{INC}

 \title[Semi-discrete modeling of systems of disclinations and dislocations]{Semi-discrete modeling of systems of wedge disclinations and edge dislocations via the Airy stress function method}

\author[P. Cesana]
{Pierluigi Cesana}
\address[Pierluigi Cesana]{Institute of Mathematics for Industry, Kyushu University, 744 Motooka, Fukuoka 819-0395, Japan}
\email[P. Cesana]{cesana@math.kyushu-u.ac.jp}
\author[L. De Luca]
{Lucia De Luca}
\address[Lucia De Luca]{Istituto per le Applicazioni del Calcolo ``M. Picone'', IAC-CNR, via dei Taurini, 19, 00185 Rome, Italy}
\email[L. De Luca]{lucia.deluca@cnr.it}
\author[M. Morandotti]{Marco Morandotti}
\address[Marco Morandotti]{Dipartimento di Scienze Matematiche ``G.~L.~Lagrange'', Politecnico di Torino, Corso Duca degli Abruzzi, 24, 10129 Torino, Italy}
\email[M. Morandotti]{marco.morandotti@polito.it}

\usepackage{hyperref}
\begin{document}
\allowdisplaybreaks
\begin{abstract}
We present a variational theory for lattice defects of rotational and translational type.
We focus on finite systems of planar wedge disclinations, disclination dipoles, and edge dislocations, which we model as the solutions to minimum problems for isotropic elastic energies under the constraint of kinematic incompatibility.
Operating under the assumption of planar linearized kinematics, we formulate the mechanical equilibrium problem 
in terms of the Airy stress function, for which we introduce a rigorous analytical formulation   in the context of incompatible elasticity.
Our main result entails the  analysis of the energetic equivalence of systems of disclination dipoles and edge dislocations 
in the asymptotics of their singular limit regimes.
By adopting the regularization approach via core radius, we  
show that, as the core radius vanishes, the asymptotic energy expansion for disclination dipoles coincides with the   energy of finite systems of edge dislocations.
This proves that Eshelby's kinematic characterization of an edge dislocation in terms of a disclination dipole is exact also from the energetic standpoint.

\end{abstract}
\maketitle
\tableofcontents
\section*{Introduction}
The  modeling of  translational and rotational defects in solids, typically referred to as \emph{dislocations} and \emph{disclinations}, respectively, dates back to the pioneering work of Vito Volterra  on the investigation of the equilibrium configurations of multiply connected bodies \cite{V07}.
Dislocations, possibly the most common lattice {defects}, are regarded as the main {microscopic} mechanism of ductility and plasticity 
of metals and elastic crystals \cite{Orowan1934,Polanyi1934,Taylor1934}. Disclinations  appear at the lattice level in metal alloys \cite{ET67,TE68}, 
graphene \cite{Banhart11,Yang18}, and 
virus shells \cite{harris77,NABARRO71}.
Despite both being line defects, their behavior is different, both geometrically and energetically. Moreover, the mathematical modeling is mostly available in the mechanical assumption of cylindrical geometry, where the curves on which the defects are concentrated are indeed line segments parallel to the cylinder axis.

Dislocations entail a violation of translational symmetry and are characterized by the so-called \emph{Burgers vector}. Here we consider only \emph{edge dislocations}, namely those whose Burgers vector is perpendicular to the dislocation line.
Disclinations arise as a violation of rotational symmetry and are characterized by the so-called \emph{Frank angle}.
Disclinations are defined 
(see
\cite{AEST68,N67}) as the ``closure failure of rotation ... for a closed circuit round the disclination centre''.
Conceptually, a planar wedge disclination can be realized in the following way, see \cite{V07}. 
In an infinite cylinder, remove a triangular wedge of material and restore continuity by glueing together the two surfaces of the cut: this results in a positive wedge disclination; conversely,  open a surface with a vertical cut originating at the axis of the infinite cylinder through the surface, insert an additional wedge of material into the cylinder through the opening, and restore continuity of the material: this results in a negative wedge disclination \cite{romanov92}. 
Because of the cylindrical geometry, we will work in the cross-section of the material, where {both disclination and dislocation} lines are identified by points in the two-dimensional sections.
In this setting,
the energy of an edge dislocation scales, far away from its center, as the logarithm of the size   of the domain,
while the   energy of a single disclination is non-singular and scales quadratically with the size of the domain \cite{L03,ROMANOV09}.
In many observations disclinations  appear   in the form of dipoles \cite{hirth20,KK91,romanov92}, 
which are pairs of wedge disclinations of opposite Frank angle placed at  a close (but finite) distance.
This configuration has the effect of  screening  the  mutual  elastic strains resulting in significantly lower energy than the one of single, isolated disclinations.
 
A continuum theory for disclinations in the framework of linearized elasticity has been developed and systematized, among a number of   authors, by de Wit in \cite{W68} and subsequently in \cite{dW1,dW2,dW3}.
A non-linear theory of disclinations and dislocations has been developed in \cite{Z97}, to which we refer the interested reader for a historical excursus 
and a list of references to classical linearized theories, as well as to other early contributions on the foundation of non-linear theories. 
For more recent modeling approaches, 
in \cite{acharya15}
disclinations are comprised as a special case of   \textit{g.disclinations}, a general concept designed 
to model phase transformations, grain boundaries, and other plastification mechanisms. 
Qualitative and quantitative comparison between the classical linearized elasticity approach and the g.disclination theory is   discussed in details in \cite{ZHANG18}.
 The contributions \cite{FRESSENGEAS11} 
and \cite{Taupin15}
propose a mesoscale theory for crystal plasticity designed for modeling the dynamic interplay of disclinations and dislocations
based on linearized kinematics and written in terms of elastic and plastic curvature tensors.
Variational analysis of a discrete model for planar disclinations is performed in \cite{CVM21}.
Finally, {we point out the papers
\cite{Acharya19} and 
\cite{yavari12,yavari13},
where a differential geometry approach to large non-linear deformations is considered}.

While the body of work on dislocations is vast both in the mathematics \cite{ArizaOrtiz05,ContiGarroniMueller11,ContiGarroniMueller22,GarroniMueller05, GarroniMueller06,GromaGyorgyiIspanovity10,OrtizRepetto99} as well as in the physics and chemistry literature \cite{FleckHutchinson93, Groma97, GurtinAnand05,HirthLothe82,HullBacon01, LimkumnerdVan-der-Giessen08} due to their relevance in metallurgy and crystal plasticity, the interest on disclinations has been much lower.
This disproportion owes to the fact that disclinations are thought to be less predominant in the formation of  plastic microstructure. 
However, a large body of experimental evidence, some of which in recent years, has shown that disclinations, both in single isolated as well as multi-dipole configuration, are in fact a very relevant plastification mechanism, so that understanding their energetics and kinematics is
crucial to understanding crystal micro-plasticity
(see  \cite{ABE02,BALANDRAUD10,BALANDRAUD07,BCH15,CH20,HMHYIOONK16,HAGIHARA16,HAGIHARA10,I19,IHM13, ILTH17,IS20,Kawamura01,KK91,LN15,MA80,Manolikas80-2,hasebe20}, 
 \cite[Section~12.3.3]{GUTKIN2011329}). 
%
With this paper we intend to   lay the foundations of a general and comprehensive variational theory suitable to treat systems of rotational and  translational defects on a lattice.
We focus on three different aspects: we propose a variational model for finite systems of planar wedge disclinations;  we study dipoles of disclinations  and we identify   relevant energy scalings dictated  by geometry and loading parameters;  finally, we prove the 
asymptotic energetic equivalence of a dipole of wedge disclinations with an edge dislocation. 

\medskip

\noindent\textbf{Main contributions and impact of this work.}
We  operate under the assumption of plane strain  elastic displacements
and under the approximation of linearized kinematics so that contributions of individual defects can be added up
via superposition.
As we are mainly concerned with the modeling of experimental configurations of metals and hard crystals, we restrict our analysis to the case of two-dimensional plane strain geometries, leaving to future work the analysis in the configuration of 
buckled membranes.
We model disclinations and dislocations as point sources of  kinematic incompatibility following an approach analogous to \cite{SN88} and \cite{CermelliLeoni06}. Alternative approaches according to the stress-couple theory in linearized kinematics are pursued in \cite{W68,FRESSENGEAS11,Taupin15}.
Despite their intrinsic limitations,  linearized theories have proven useful to describe properties of systems of dislocations both in continuous and discrete models \cite{CermelliGurtin99,CermelliLeoni06, GarroniLeoniPonsiglione10, DeLucaGarroniPonsiglione12, AlicandroDeLucaGarroniPonsiglione14, ContiGarroniOrtiz15, DeLuca16, AlicandroDeLucaGarroniPonsiglione16, AlicandroDeLucaGarroniPonsiglione17, BlassMorandotti17,BlassFonsecaLeoniMorandotti15, Ginster19_1, AlicandroDeLucaLazzaroniPalombaroPonsiglione22}
(see also \cite{ScardiaZeppieri12, MuellerScardiaZeppieri14, Ginster19_2, GarroniMarzianiScala21} for related nonlinear models for (edge) dislocations).
In \cite{PL13, CPL14, CDRZZ19} systems of disclinations have been investigated in linear and finite elasticity models, 
and qualitative as well as quantitative comparisons have been discussed.

By working in plane strain linearized kinematics, it is   convenient to formulate the mechanical equilibrium problem in terms of  a scalar potential, the Airy stress function  of the system, see, \emph{e.g.}, \cite{Meleshko03,Michell}. 
This is a classical method in two-dimensional elasticity based on the introduction of a potential scalar function whose second-order derivatives correspond to the components of the stress tensor (see \cite[Section~5.7]{ciarlet97} and \cite{Schwartz66}).
From the formal point of view, by denoting with $\sigma_{ij}$ the components of the $2\times 2$ mechanical stress tensor, we write
$$
\sigma_{11}=\frac{\partial^2 v}{\partial x_2^2}\,,\qquad
\sigma_{12}=\sigma_{21}=-\frac{\partial^2  v}{\partial x_2\partial x_1}\,,\qquad
\sigma_{22}=\frac{\partial^2 v}{\partial x_1^2}\,,
$$
where $v\colon \R^2\supset\Omega\to \R$ is the Airy stress function.  
Upon introduction of the Airy potential $v$, the equation of mechanical equilibrium $\Dive\sigma=0$ is identically satisfied while the information on kinematic (in-)compatibility  is translated into  a loading source problem for the biharmonic equation for the scalar field $v$\,. 
By indicating  with~$\epsilon$ the $2\times 2$ symmetric strain tensor (related to~$\sigma$ via the linear relation $\sigma=\C\epsilon$, with~$\C$ being the fourth-order elasticity tensor), the mechanical equilibrium problem formulated in terms of strains and stresses  (which we refer to as the \emph{laboratory variables})
and in terms of the Airy potential formally read, respectively,
 %
\begin{equation}\label{2304191245}
\begin{cases}
\curl\Curl\epsilon=-\theta&\text{in $\Omega$}\\
\Dive\sigma=0&\text{in $\Omega$}\\
\sigma\,n=0 &\textrm{on $\partial\Omega$}
\end{cases}
 \qquad\qquad\text{and}\qquad\qquad
 \begin{cases}
\displaystyle \frac{1-\nu^2}{E}\Delta^2v=-\theta&\text{in }\Omega\,\\[2mm]
\displaystyle \nabla^2v\,t=0&\textrm{on }\partial \Omega.\,
\end{cases}
\end{equation}
Here, $E$ and $\nu$ are the classical Young modulus and Poisson ratio, respectively; the unit vectors~$t$ and~$n$ are the tangential and normal {directions} to the boundary of $\Omega$, 
and~$\theta$ denotes a source term {accounting} for kinematic incompatibility.
%
 Existence of the Airy stress function   and  the  variational equivalence of the equilibrium problems formulated in terms of strains and stresses,
 with the single-equation problem for the Airy potential are proved in \cite{ciarlet97} in simply connected domains for perfectly compatible (that is, defect-free, $\theta\equiv 0$) elasticity. 
 
 \smallskip
 
Our first results, Propositions~\ref{prop:airyepsilonA_soloel} and~\ref{prop:airyepsilon_soloel} (and Corollaries~\ref{prop:airyepsilonA} and~\ref{prop:airyepsilon}), entail the investigation of finite systems of  isolated disclinations, modeled by a finite sum of Dirac deltas placed at the disclination centers and modulated by their corresponding Frank angles  (see \cite{SN88}).
Consequently, we take $\theta=\sum_{k=1}^{K}s^k\de_{y^k}$, where 
$\{y^k\}_{k=1}^K\subset\Omega$ are the  fixed (hence the term  \emph{isolated}) centers of the disclinations  
and we clarify the equivalence, in terms of suitable notions of weak solutions (see Definition~\ref{defdeb}),
of the two formulation for mechanical equilibrium appearing in \eqref{2304191245},
thus generalizing the analysis of \cite{ciarlet97}
for non-zero Frank angles~$s^k$.
%
%
In doing so, we construct a rigorous variational setting so that the equilibrium problem formulated in terms of the Airy potential is well posed in terms of existence, uniqueness, and regularity of solutions.
The Airy potentials corresponding to the singular strains and stresses are the classical solutions for planar wedge disclinations computed in \cite{V07} -- and correctly recovered by our model -- corresponding to the Green's function for the bilaplacian operator.
An immediate application of our analysis is in providing a
rigorous framework for numerical calculations
of lattice defects with the Airy potential method (see, \emph{e.g.}, \cite{SN88, ZHANG14}).  
Additionally, we show that the solutions to the mechanical equilibrium problem formulated in the Airy variable (the system on the right in~\eqref{2304191245})
can be characterized as the minimizers of the following functional for the Airy stress function
\begin{equation}\label{2304241448}
{\I^{\theta}(v;\Omega)\coloneqq\frac{1}{2}\frac{1+\nu}{E}\int_{\Omega}\left(|\nabla^2 v|^2-\nu(\Delta v)^2\right)\ud x+\langle\theta,v\rangle\,,}
\end{equation}
defined over a suitable class of Sobolev functions {(see \eqref{defI} and \eqref{minI})}.
Here, the bulk term of the functional coincides with the elastic energy measured in terms of the laboratory variable, while the linear part $v\mapsto\langle\theta,v\rangle$  represents  the work performed by the point singularities (disclinations) and does not enter the  mechanical energy balance.
In the remainder of the paper we exploit extensively the variational  characterization for the problem formulated in the Airy variable in the analysis of singular regimes. 

\smallskip

Secondly, we show the energetic equivalence between finite families of wedge disclinations dipoles and systems of edge dislocations.
{From the point of view of  the applications in Materials Science, these systems are interesting because disclination dipoles are fundamental building blocks to model
kinks as well as
grain boundaries
\cite{Gertsman89,LI72, Nazarov13}, which are important configurations in crystals  and metals. }
{To this end, we first consider}
a dipole of wedge disclinations placed at a distance $h>0$ along the~$x$-axis, that is, we set $\theta_h=s\de_{(\frac{h}{2};0)}-s\de_{(-\frac{h}{2};0)}$ in~\eqref{2304241448}.
Then, {replacing the linear term in \eqref{2304241448} with an average of the variable $v$ at a scale $\ce>h$\,, we define the regularization of the functional $\I^{\theta_h}$ 
as
\begin{equation}\label{2304241515}
\begin{aligned}
\I^{\theta_h}_{h,\ce}(v)\coloneqq\,&\,
\frac{1}{2}\frac{1+\nu}{E}\int_{\Omega_{\varepsilon}} \!\!\left(|\nabla^2 v|^2-\nu(\Delta v)^2\right)\ud x\\
&\,
+\frac{s}{2\pi(\ce-h)}\int_{\partial B_{\ce-h}(0)} \!\bigg[v\Big(x+\frac h 2e_1\Big)-v\Big(x-\frac h 2 e_1\Big)\bigg]\,\ud\Huno(x),
\end{aligned}
\end{equation}
where $\Omega_\ce\coloneqq\Omega\setminus\overline{B}_\ce(0)$\,.

Observe that, when keeping $\ce>0$ fixed and letting $h\to 0$\,, we have that $\theta_h\to 0$ and the functional $\I^{\theta_h}_{h,\ce}$ converges to 
the sole bulk energy term integrated over
$\Omega_{\varepsilon}$\,.
On the other hand, since
\begin{equation}\label{questaquilachiamo}
\I^{\theta_h/h}_{h,\ce}(v/h)=\frac1{h^2}\I^{\theta_h}_{h,\ce}(v),
\end{equation}
we obtain that linear rescalings by $h$ of both the function $v$ and the measure $\theta_h$ induce quadratic rescalings by $h^2$ of corresponding energies.
Therefore, setting 
$$\widetilde\I_{h,\ce}^{\theta_h}(w)\coloneqq \frac{1}{h^2}\I^{\theta_h}_{h,\ce}(hw),$$
we have that 
$$\widetilde\I_{h,\ce}^{\theta_h}(w)= \I^{\theta_h/h}_{h,\ce}(w),$$
that is, the left-hand side of \eqref{questaquilachiamo} with $w=v/h$.

Notice that $\theta_h/h\to -s\partial_{x_1}\de_0$ in the sense of distributions; this leads to the formalization in terms of incompatibility operators of Eshelby's derivation \cite{Eshelby66} (see also \cite{dW3}) of the edge dislocation $\alpha\coloneqq se_2\de_0$ 
as the disclination dipole $\theta_h$ with vanishing length $h$ and Frank angles $\pm s$\,. 
The relationship between the dipole of disclinations and the equivalent edge dislocation is clarified by noting that $-s\partial_{x_1}\de_0=\curl\alpha$ 
(see \eqref{incfinal} and \eqref{incfinal2} for the full details details).
%
%
Therefore, from an energetic point of view, we expect
 that the functionals $\I^{\theta_h/h}_{h,\ce}$
 converge, in a suitable sense, to the functional
 \begin{equation}\label{2304241523}
\I^{\alpha}_{0,\ce}(v) \coloneqq
\frac{1}{2}\frac{1+\nu}{E}\int_{\Omega_{\varepsilon}}\left(|\nabla^2 v|^2-\nu(\Delta v)^2\right)\ud x
+\frac{s}{2\pi\ce}\int_{\partial B_\ce(0)}\partial_{x_1}v(x)\,\ud\Huno(x)\,.
\end{equation}
This is the content of Proposition~\ref{convhtozero}, where we prove that the minima 
and minimizers 
of the functional $\I^{\theta_h}_{h,\ce}$ converge (as $h\to 0$) to those of $\I^{\alpha}_{0,\ce}$\,.

Since the loading term in \eqref{2304241523} is an $\ce$-regularization of the {\it core} energy associated with the edge dislocation $se_2\de_0$\,, the length scale~$\ce$ can be interpreted as the core radius of this edge dislocation. 
In other words, at scales larger than~$\ce$, the material responds to continuum theories of elasticity, whereas discrete descriptions are better suited at scales smaller then~$\ce$\,, thus establishing the \emph{semi-discrete} nature of our model.

In Remark~\ref{2202211829}, we show that the same convergence carries through for a system of isolated dipoles of disclinations, \emph{i.e.}, when $\theta_h$ represents a finite system of disclination dipoles (with length~$h$) approximating a finite system of edge dislocations identified by the corresponding~$\alpha$.

\smallskip

Finally, keeping $\alpha$ fixed, we discuss the asymptotic expansion of the $\ce$-regularized dislocation energy $\I_{0,\ce}^\alpha$ as $\ce\to0$. 
By relying on an additive decomposition between plastic (\emph{i.e.}, determined by the disclinations) and elastic parts of the Airy stress function, in an analogous fashion to \cite{CermelliLeoni06}, we study the limit of the minimal $\I_{0,\ce}^\alpha$ as $\ce\to0$ (Theorem~\ref{2201181928}), we compute the renormalized energy of the system (Theorem~\ref{CLequiv}), and we finally obtain the energetic equivalence, which is the sought-after counterpart of Eshelby's kinematic equivalence. 
{Our asymptotic expansion of the minimal $\I^\alpha_{0,\ce}$\,, obtained via the Airy stress function formulation (see \eqref{20220222_8})\,,  is in agreement with  \cite[Theorem~5.1 and formula (5.2)]{CermelliLeoni06} at all orders;
therefore, as in \cite{CermelliLeoni06}, the minimizers  of $\I_{0,\ce}^\alpha$\,, as $\ep\to 0$\,, converge to the sum of the Green's functions associated with each of the dislocation in~$\alpha$ plus a smooth function matching the traction-free boundary condition.
}
%
%
To conclude, in Theorem~\ref{diago}, we combine in a cascade the convergence results obtained above (sending first $h\to 0$ and then $\ce\to 0$)\,,
computing, via a diagonal argument, the asymptotic expansion of the energy $\I^{\theta_h}_{h,\ce(h)}$ for $h\ll \ce(h)$ as $h\to 0$\,.
This extends the asymptotic analysis in \cite{CermelliLeoni06}  to finite systems of dipoles of wedge disclinations.

\medskip
\noindent
\textbf{Outline of the paper and methods.}
The outline of the paper is as follows.
 Section \ref{sc:model}  is devoted to the presentation of the mechanical equilibrium equations, in terms both of the laboratory variables and of the Airy stress function of the system.
Our results are based on a crucial characterization of traction-free boundary displacements for the problem formulated in terms of the Airy potential. Such a characterization involves a non-standard tangential boundary condition for the Hessian of the Airy stress function which we are able to characterize in terms of classical Dirichlet-type boundary conditions for the biharmonic equation (Proposition \ref{20220421}).

In Section \ref{sc:isolated_disclinations}, we study the mechanical problem
for systems of isolated disclinations 
formulated in terms of the Airy potential. 
With Section \ref{sub:dipole},
we begin our investigation of systems of disclination dipoles which we then conclude in Section \ref{sc:four}. 
Length scales and mutual distances between disclinations are regarded as model parameters, of which we study the asymptotics.
%
%
We operate by directly computing the limits of energy minima and minimizers; a more general approach 
via $\Gamma$-convergence \cite{Braides02,DalMaso93}
 is not explored in this paper.
We stress that the results in Section~\ref{sc:four} are written for finite systems of disclination dipoles and dislocations. In particular, Theorem~\ref{CLequiv} fully characterizes the energy of a finite system of dislocations: the renormalized energy in~\eqref{ren_en_dislo}  contains information on the mutual interaction of the dislocations.
%

While our  focus is on defects and kinematically incompatible systems, our systematization of the Airy stress function method is useful also for the general case of compatible elasticity. We investigate a number of analytical questions, such as the equivalence of boundary data in terms of the laboratory variables and the Airy potential, fine Poincar\'{e} and trace inequalities in perforated domains, and density of Airy potentials under non-standard constraints. 
We gather these original results 
in a series of 
appendices. 

\medskip

\textsc{Acknowledgments:} The authors are members of the Gruppo Nazionale per l'Analisi Matematica, la Probabilit\`a e le loro Applicazioni (GNAMPA) of the Istituto Nazionale di Alta Matema\-tica (INdAM).
PC   holds an honorary appointment at La Trobe University and is supported by  JSPS Innovative Area Grant JP21H00102 and partially  JP19H05131.
MM gratefully acknowledges support from the \emph{Japan meets Italian Scientists} scheme of the Embassy of Italy in Tokyo and from the MIUR grant Dipartimenti di Eccellenza 2018-2022 (CUP: E11G18000350001).
MM acknowledges  the   Institute of Mathematics for Industry, an International Joint Usage and Research Center located in Kyushu University, where part of the work contained in this paper was carried out.

\medskip

\textsc{Supplementary materials or data:} There are no supplementary materials or data associated with this manuscript.

\medskip

\textsc{Conflict of interests:} The authors declare no conflict of interests.

\vskip10pt
\textsc{Notation.}
For $d\in\{2,3\}$\,, $m\in\N$\,, and for every $k\in\Z$\,, let $\reg^k(A;\R^m)$ denote the space of $k$-regular $\R^m$-valued functions defined on an open set $A\subset\R^d$ (we will consider Sobolev spaces like $H^{k}(A;\R^m)$ or spaces of $k$-differentiable functions like $C^{k}(A;\R^m)$, for $k\ge 0$)\,.
Now we introduce different curl operators and show relationships among them.
For $d=3$ and $m=3$ we define $\CURL\colon\reg^k(A;\R^{3})\to \reg^{k-1}(A;\R^3)$ as 
\begin{equation*}
\begin{aligned}
\CURL V\coloneqq&(\partial_{x_2}V^3-\partial_{x_3}V^2;\partial_{x_3}V^1-\partial_{x_1}V^3; \partial_{x_1}V^2-\partial_{x_2}V^1)
\end{aligned}
\end{equation*}
for any $V=(V^1;V^2;V^3)\in\reg^k(A;\R^{3})$\,,
or, equivalently, $(\CURL V)^i=\ep_{ijk}\partial_{x_j}V^k$\,, where $\ep_{ijk}$ is the Levi-Civita symbol.
For $d=3$ and $m=3\times 3$ we define $\ccurl\colon\reg^k(A;\R^{3\times 3})\to \reg^{k-1}(A;\R^{3\times 3})$ by $(\ccurl M)_{ij}\coloneqq\ep_{ipk}\partial_{x_p}M_{jk}$ for every $M\in \reg^k(A;\R^{3\times 3})$ and
we notice that $(\ccurl M)_{ij}=(\CURL M_j)^i$\,, where $M_j$ denotes the $j$-th row of $M$\,.
Moreover, we denote by $\INC\colon\reg^k(A;\R^{3\times 3})\to \reg^{k-2}(A;\R^{3\times 3})$ the operator defined by $\INC\coloneqq\ccurl\ccurl\equiv\ccurl\circ\ccurl$\,.

For $d=2$ and $m\in\{2,2\times 2\}$\,,
we define the following curl operators: $\curl\colon\reg^k(A;\R^2)\to \reg^{k-1}(A;\R)$ as $\curl v\coloneqq\partial_{x_1}V^2-\partial_{x_2}{V^1}$ for any $V=(V^1;V^2)\in\reg^k(A;\R^2) $, $\Curl\colon\reg^k(A;\R^{2\times 2})\to \reg^{k-1}(A;\R^2)$ as $\Curl M:=(\curl M_1;\curl M_2)$ for any $M\in \reg^k(A;\R^{2\times 2})$\,.

Let now $A\subset\R^2$ be open. For every $V=(V^1;V^2)\in\reg^{k}(A;\R^2)$\,, we can define $\underline{V}\in\reg^k(A\times\R;\R^3)$ as $\underline{V}(x_1;x_2;x_3)\coloneqq(V^1(x_1;x_2);V^2(x_1;x_2);0)$ and we have that
$$
\CURL\underline{V}=(0;0;\curl V)\,.
$$
Analogously, if $M\in\reg^k(A;\R^{2\times 2})$\,, then, defining $\underline{M}\colon A\times\R\to \R^{3\times 3}$ by $\underline{M}_{ij}(x_1;x_2;x_3)=M_{ij}(x_1;x_2)$ if $i,j\in\{1,2\}$ and  $\underline{M}_{ij}=0$ otherwise, we have that $\underline{M}\in \reg^k(A\times\R;\R^{3\times 3})$\,,
$$
\ccurl\underline{M}=\left[\begin{array}{ccc}
0&0&\curl M_1\\
0&0&\curl M_2\\
0&0&0
\end{array}\right]\,,\qquad \ccurl\ccurl\underline{M}=\left[\begin{array}{ccc}
0&0&0\\
0&0&0\\
0&0&\curl\Curl M
\end{array}\right]\,.
$$
In what follows, $\R^{2\times 2}_{\sym}$ is the set of the matrices $M\in\R^{2\times 2}$ with $M_{ij}=M_{ji}$ for every $i,j=1,2$\,.
Furthermore, for every $M\in\R^{2\times 2}$ we denote by $M^{\top}$ the matrix with entries $(M^{\top})_{ij}=M_{ji}$ for every  $i,j=1,2$\,.

Finally, in the whole paper, the symbol $C$ indicates a constant that may change from line to line. Whenever we want to stress the dependence of $C$ from other constants $c_1,\ldots, c_K$ or sets $\omega_1,\ldots,\omega_L$ we adopt the notation $C(\alpha_1,\ldots,\alpha_K,\omega_1,\ldots,\omega_L)$\,.


\section{The mechanical model}\label{sc:model}
\subsection{Plane strain elasticity}
Let $\Omega$ be an open bounded simply connected subset of $\R^2$ with~$C^2$ boundary. For any displacement $u\in H^1(\Omega;\R^2)$ the associated elastic strain $\epsilon\in L^2(\Omega;\R^{2\times 2}_{\sym})$ is given by $\epsilon\coloneqq\nabla^{\sym} u\coloneqq\frac{1}{2}(\nabla u+\nabla^{\top} u)$, whereas the corresponding stress $\sigma\in  L^2(\Omega;\R^{2\times 2}_{\sym})$ is defined by 
\begin{equation}\label{stressstrain}
\sigma\coloneqq\C\epsilon\coloneqq\lambda\mathrm{tr}(\epsilon)\mathbb{I}_{2\times 2}+2\mu\epsilon\,;
\end{equation}
here $\C$ is the {\it isotropic elasticity tensor} with {\it Lam\'e constants} $\lambda$ and $\mu$\,.
Notice that 
\begin{subequations}\label{lamepos}
\begin{equation}
\C\textrm{ is positive definite}
\end{equation}
if and only if 
\begin{equation}\label{lame}
\mu>0\qquad\textrm{ and }\qquad\lambda+\mu>0\,,
\end{equation}
or, equivalently,
\begin{equation}\label{lame3}
E>0\qquad\textrm{ and }\qquad-1<\nu<\frac{1}{2}\,.
\end{equation}
Here and below,  $E$ is  the {\it Young modulus}  and $\nu$ is  the {\it Poisson ratio}, in terms of which the Lam\'e constants $\lambda$ and $\mu$ are expressed by 
\end{subequations}
\begin{equation}\label{lame2}
\mu=\frac{E}{2(1+\nu)}\qquad\textrm{ and }\qquad \lambda=\frac{E\nu}{(1+\nu)(1-2\nu)}\,.
\end{equation}
We will assume \eqref{lamepos} throughout the paper.

In {plane} strain elasticity the isotropic elastic energy associated with the displacement $u$ in the body $\Omega$ is defined by 
\begin{equation}\label{def:energy}
\E(u;\Omega)\coloneqq\frac 1 2 \int_{\Omega}\sigma:\epsilon\,\ud x=\frac{1}{2}\int_{\Omega}\big(\lambda (\mathrm{tr}(\epsilon))^2+2\mu|\epsilon|^2\big)\,\ud x\,;
\end{equation}
we notice that in formula \eqref{def:energy} the energy $\E(\cdot;\Omega)$ depends only on $\epsilon$ so that in the following, with a little abuse of notation, we will denote by $\E(\cdot;\Omega)\colon L^2(\Omega;\R^{2\times 2}_\sym)\to [0,+\infty)$  the energy functional defined in \eqref{def:energy}, considered as a functional of $\epsilon$ (and not of $u$).

Notice that we can write the elastic energy also as a function of the stress $\sigma$ as
\begin{equation}\label{energysigma}
\F(\sigma;\Omega):=\frac{1}{2}\frac{1+\nu}{E}\int_{\Omega} \big(|\sigma|^2-\nu(\mathrm{tr}(\sigma))^2\big)\,\ud x=\E(\epsilon;\Omega)\,,
\end{equation}
where we have used \eqref{stressstrain}  and \eqref{lame2} to deduce that
\begin{equation}\label{strain_stress}
\epsilon_{11}=\frac{1+\nu}{E}\Big((1-\nu)\sigma_{11}-\nu\sigma_{22}\Big)\,,\quad \epsilon_{12}=\frac{1+\nu}{E}\sigma_{12}\,,\quad \epsilon_{22}=\frac{1+\nu}{E}\Big((1-\nu)\sigma_{22}-\nu\sigma_{11}\Big)\,,
\end{equation}
and 
\begin{equation}\label{intE}
\lambda \big(\mathrm{tr}(\epsilon)\big)^2+2\mu|\epsilon|^2=\frac{1+\nu}{E}\big(|\sigma|^2-\nu(\mathrm{tr}(\sigma))^2\big)\,.
\end{equation}
Finally, we reformulate the energy \eqref{energysigma} using the  Airy stress function method. This assumes the existence of a function $v\in H^2(\Omega)$ such that
\begin{equation}\label{airy}
\sigma_{11}=\partial^2_{x_2^2}v\,,\quad\sigma_{12}=-\partial^{2}_{x_1x_2}v\,,\quad\ \sigma_{22}=\partial^2_{x_1^2}v\,;
\end{equation}
more precisely, we consider the operator $\Airy\colon\reg^k(\Omega)\to\reg^{k-2}(\Omega;\R^{2\times 2}_\sym)$ such that $\sigma=\sigma[v]=\Airy(v)$ is defined by \eqref{airy}\,.
It is immediate to see that the operator $\Airy$ is not injective, since $\Airy(v)=\Airy(w)$ whenever $v$ and $w$ differ up to an affine function; its invertibility under suitable boundary conditions will be discussed in Subsection \ref{incairy} (see Corollaries \ref{prop:airyepsilonA} and \ref{prop:airyepsilon}).

Assuming that there exists $v$ such that $\sigma=\sigma[v]=\Airy(v)$\,, from \eqref{airy}, we can rewrite \eqref{energysigma} as
\begin{equation}\label{energyairy}
\F(\sigma[v];\Omega)=\frac 1 2\frac{1+\nu}{E}\int_{\Omega}\Big(|\nabla^2 v|^2-\nu|\Delta v|^2\Big)\,\ud x\eqqcolon \G(v;\Omega)\,.
\end{equation}
We notice that if the stress $\sigma$ admits an Airy potential $v$\,, i.e., $\sigma=\sigma[v]=\Airy(v)$\,, then 
\begin{equation}\label{divenulla}
\Dive\sigma[v]\equiv 0\,,
\end{equation}
that is, the equilibrium equation $\Dive\sigma= 0$ is automatically satisfied.
In fact, this is the main advantage in using the Airy stress function method.
Notice that the identity in \eqref{divenulla} is, at this stage, formal and in general holds in the distributional sense. As we will see in Subsection \ref{incairy}, in our case equation \eqref{divenulla} will hold in $H^{-1}(\Omega;\R^2)$\,.

\subsection{Kinematic incompatibility: dislocations and disclinations}\label{sc:inclab}
Let $u\in C^3(\Omega;\R^2)$ and set $\beta:=\nabla u$\,. Clearly, 
\begin{subequations}\label{compa_tutte}
\begin{equation}\label{compa}
\Curl\beta=0\qquad\textrm{ in }\Omega\,.
\end{equation}
We can decompose $\beta$ as
 $\beta=\epsilon+\beta^{\skw}$\,, where $\epsilon\coloneqq\frac{1}{2}(\beta+\beta^{\top})$ and $\beta^{\skw}\coloneqq\frac{1}{2}(\beta-\beta^{\top})$\,. By construction,
 \begin{equation*}
 \beta^{\skw}=\left(\begin{array}{ll} 0&f\\
-f&0
\end{array}\right)\,,
\end{equation*}
for some function $f\in  C^2(\Omega)$\,, and hence $\Curl\beta^{\skw}=\nabla f$\,.
 Therefore, the compatibility condition \eqref{compa} can be rewritten as
\begin{equation}\label{compa2}
\Curl\epsilon=-\nabla f\qquad\textrm{in }\Omega\,,
\end{equation}
which, applying again the $\curl$ operator, yields the {\it Saint-Venant compatibility condition}
\begin{equation}\label{compa3}
\curl\Curl\epsilon=0\qquad\textrm{in }\Omega\,.
\end{equation}
\end{subequations}

Viceversa, given $\epsilon\in C^2(\Omega;\R^{2\times2}_{\sym})$, the {\it Saint-Venant principle} \cite{SV1855} states that if \eqref{compa3} holds, then there exists $u\in C^3(\Omega;\R^2)$ such that $\epsilon=\nabla^{\sym}u$\,.  

In order to apply the direct method of the Calculus of Variations for the minimization of the elastic energy \eqref{def:energy}, the natural functional setting for the displacement $u$ is the Sobolev space $H^1(\Omega;\R^2)$\,. Therefore, a natural question that arises is whether identities \eqref{compa_tutte} make sense also when $\beta$ is just in $L^2(\Omega;\R^{2\times 2})$.  
The answer to this question is affirmative as shown by the following result proved in \cite{ciarlet05} (see also \cite{Geymonat09}).
\begin{proposition}\label{sv}
Let $\Omega\subset\R^2$ be an open, bounded, and simply connected set and let $\epsilon \in L^2(\Omega;\R^{2\times 2}_{\sym})$\,. Then,
\begin{equation}\label{compa4}
\curl\Curl\epsilon=0\qquad\textrm{in }H^{-2}(\Omega)
\end{equation}
if and only if there exists $u\in H^1(\Omega;\R^2)$ such that $\epsilon=\nabla^{\sym} u$\,. Moreover, $u$ is unique up to rigid motions.
\end{proposition}
Notice that, by the Closed Graph Theorem, we have that \eqref{compa4} holds true in $H^{-2}(\Omega)$ if and only if it holds in the sense of distributions. Therefore, the generalizations of identities \eqref{compa_tutte} when $u\in H^1(\Omega;\R^2)$ are given by
 \begin{subequations}\label{compadebole}
 \begin{eqnarray}\label{compa10}
 \Curl\beta& \!\!\!\! =&\!\!\!\! 0\qquad\qquad\!\!\textrm{ in }\Dcal'(\Omega;\R^2)\,,\\ \label{compa20}
 \Curl\epsilon& \!\!\!\!=&\!\!\!\! -\nabla f\qquad\textrm{ in }\Dcal'(\Omega;\R^2)\,, \\
 \label{compa30}
\curl\Curl\epsilon& \!\!\!\!=&\!\!\!\! 0\qquad\qquad\textrm{in }\Dcal'(\Omega)\,, 
 \end{eqnarray}
 \end{subequations}
 where  $f$ is a function in $L^2(\Omega)$ and the operator $\nabla$ should be understood in the sense of distributions. (Here and below, $\Dcal'(\Omega;\R^2)$ and $\Dcal'(\Omega)$ denote the families of $\R^2$-valued and $\R$-valued, respectively, distributions on $\Omega$\,.)
Clearly, if $\beta$ is not a gradient, then equations \eqref{compadebole}
are not satisfied anymore. 
In particular, if the right-hand side of  \eqref{compa10} is equal to some 
$\alpha\in\mathcal{D}'(\Omega;\R^2)$, then \eqref{compa20} becomes
\begin{equation}\label{compa200}
 \Curl\epsilon=\alpha-\nabla f\qquad\textrm{ in }\Dcal'(\Omega;\R^2)\,.
\end{equation}
Moreover, if the right-hand side of \eqref{compa20} is equal to $-\kappa$ where $\kappa\in H^{-1}(\Omega;\R^2)$ is not a gradient, 
then
 \eqref{compa30} becomes
\begin{equation}\label{compa300}
\curl\Curl\epsilon=-\theta\qquad\textrm{in }\Dcal'(\Omega)\,,
\end{equation}
where we have set $\theta\coloneqq\curl\kappa$\,.
Finally, when both incompatibilities are present, we have that
\begin{equation}\label{incfinal}
\curl\Curl\epsilon=\curl\alpha-\theta\qquad\textrm{in }\Dcal'(\Omega)\,.
\end{equation}
We will focus on the case when $\alpha$ and $\theta$ are finite sums of Dirac deltas.
More precisely, we will consider $\alpha\in\ED(\Omega)$ and $\theta\in\WD(\Omega)$\,, where 
\begin{equation*}
\begin{aligned}
\ED(\Omega)\coloneqq&\bigg\{\alpha=\sum_{j=1}^{J}b^j\de_{x^j}\,:\,J\in\N\,,\,b^j\in\R^2\setminus\{0\}\,,\,x^j\in\Omega\,,\,x^{j_1}\neq x^{j_2}\textrm{ for }j_1\neq j_2\bigg\}\,,\\
\WD(\Omega)\coloneqq&\bigg\{\theta=\sum_{k=1}^{K}s^k\de_{y^k}\,:\,K\in\N\,,\,s^k\in\R\setminus\{0\}\,,\,y^k\in\Omega\,,\,\,y^{k_1}\neq y^{k_2}\textrm{ for }k_1\neq k_2\bigg\}\,.
\end{aligned}
\end{equation*}
In this case \eqref{incfinal} reads
\begin{equation}\label{incfinal2}
\curl\Curl\epsilon=-\sum_{j=1}^{J}|b^j|\partial_{\frac{(b^j)^\perp}{|b^j|}}\de_{x^j}-\sum_{k=1}^{K}s^k\de_{y^k}\qquad\textrm{in }\Dcal'(\Omega)\,,
\end{equation}
where we recall that $b^\perp=(-b_2;b_1)$ for every $b=(b_1;b_2)\in\R^2$\,.
The measure $\alpha$ identifies a system of $J$ edge dislocations with Burgers vectors $b^j$\,; the measure $\theta$ identifies a system of $K$ wedge disclinations with Frank angles $s^k$\,. 

\begin{remark}\label{discreteweigths}
\rm{For the sake of simplicity we will assume that the weights $b^j$'s and $s^k$'s of the singularities of $\alpha$ and $\theta$ lie in $\R^2\setminus\{0\}$ and $\R\setminus\{0\}$\,, respectively.
Actually, in the theory of perfect edge dislocations, we have that $b^j\in\mathcal{B}\subset\R^2$\,, where $\mathcal{B}$ is the {\it slip system}, i.e., the (discrete) set of the vectors of the crystallographic lattice. Analogously, in the theory of perfect disclinations, $s^k\in\mathcal{S}$\,, where, in a regular Bravais lattice, $\mathcal{S}$ is given by the integer multiples of the minimal angle $s$ between two adjacent nearest-neighbor bonds of a given point (namely, $s=\pm\frac{\pi}{2}$ in the square lattice and $s=\pm\frac{\pi}{3}$ in the regular triangular lattice). 
Whenever $b^j$ are not vectors in $\mathcal{B}$ or $s^k$ are not angles in $\mathcal{S}$, the corresponding dislocations and disclinations are referred to as \textit{partial}, see \cite{Wit1972,N67}.
Since we will focus only on the regime of finite number of edge dislocations and wedge disclinations, the classes $\mathcal{B}$ and $\mathcal{S}$ do not play any role in our analysis. 
}
\end{remark}

Let $\alpha\in\ED(\Omega)$ and $\theta\in\WD(\Omega)$\,.
Following \cite{W68,dW1}, for every open set $A\subset\Omega$ with $\partial A\cap(\supp\alpha\cup\supp\theta)=\emptyset$ we define the Frank angle $\omega\res A$\,, the Burgers vector ${\bf b}\res A$\,, and the {\it total Burgers vector} ${\bf B}\res A$ restricted to $A$ as
$$
\omega\res A\coloneqq\theta(A)\,,\qquad
{\bf b}\res A \coloneqq \alpha(A)\,,\qquad
{\bf B}\res A\coloneqq{\bf b}\res A-\int_{A}(-x_2;x_1)\,\ud\theta\,.
$$
We notice that in \cite{W68,dW1}\,, the Frank angle is indeed a rotation vector $\bm\Omega\res A$, which in our plane elasticity setting is the vector perpendicular to the cross section given by $\bm\Omega\res A=(0;0;\omega\res A)$\,.

For the purpose of illustration, we notice that if $\supp\theta\subset \Omega\setminus A$\,, then $\omega\res A=0$ and ${\bf B}\res A={\bf b}\res A=\alpha(A)$\,. Now, if $\supp\alpha\subset \Omega\setminus A$ and 
$\theta=s\de_{y}$ for some $y\in A$\,, then
$\omega\res A=\theta(A)=s$\,, ${\bf b}\res A=0$\,, and ${\bf B}\res A=-s(-y_2;y_1)$\,. This illustrates the different contributions of dislocations and disclinations to the quantities $\omega$\,, $\bf b$\,, and $\bf B$ just introduced: dislocations only contribute to the Burgers vector but never to the Frank angle, whereas disclinations contribute both to the Frank angle and to the total Burgers vector.
   
Finally, supposing for convenience that $\supp\alpha\subset\Omega\setminus A$\,, if $\theta=s\big(\de_{y+\frac{h}{2}}-\de_{y-\frac{h}{2}}\big)$ for some $y,h\in\R^2$ with $y\pm\frac{h}{2}\in A$\,, we have that
$$
\omega\res A=0\qquad\textrm{and}\qquad\mathbf{B}\res A=-s(-h_2;h_1)\,,
$$   
which shows that a dipole of opposite disclinations does not contribute to the Frank angle but contributes to the total Burgers vector independently of its center $y$ (see Section~\ref{sub:dipole}).
   
\subsection{Disclinations in terms of the Airy stress function}\label{incairy}
In this subsection, we rewrite the incompatibility condition in \eqref{incfinal} in terms of the Airy stress function $v$ introduced in \eqref{airy}.
To this purpose, assume that $\alpha\equiv 0$\,, so that  \eqref{incfinal} coincides with \eqref{compa300}. 
Here and henceforth we use the symbols $n$ and $t$ to denote the external unit normal and tangent vectors {to the boundary of $\Omega\subset\R^2$\,}, respectively, such  that $t=n^{\perp}=(-n_2;n_1)$\,; in this way, the ordered pair $\{n,t\}$ is a right-handed orthonormal basis of~$\R^2$\,.

Consider $v\colon\Omega\to\R{}$ and let $\sigma=\sigma[v]=\Airy(v)$ (see \eqref{airy}) and $\epsilon[v]=\C^{-1}\sigma[v]$ (see \eqref{strain_stress}). 
Then, formally, 
\begin{subequations}\label{conversions}
\begin{eqnarray}
\curl\Curl \epsilon[v]& \!\!\!\!\equiv& \!\!\!\!\frac{1-\nu^2}{E}\Delta^2 v\,,
\label{airy2} \\
\C\epsilon[v]\,n& \!\!\!\!\equiv& \!\!\!\!\sigma[v]\, n\equiv(\partial^2_{x_2^2}vn_1-\partial_{x_1x_2}^2 v n_2;-\partial^2_{x_1x_2}vn_1+\partial^2_{x_1^2}vn_2)\equiv\nabla^2 v\, t\,. \label{airybdry}
\end{eqnarray}
\end{subequations}
As customary in mechanics, we refer to
the zero-stress boundary condition
  $\C\epsilon[v]\,n=0$
on $\partial\Omega$ as \textit{traction-free}.
With some abuse of notation, we also name traction-free 
the same boundary condition 
measured in terms of the tangential component of the Hessian of the  
  Airy potential, that is $ \nabla^2 v\, t=0$
on $\partial\Omega$. 

If $\epsilon$ satisfies the equilibrium equations subject to the incompatibility constraint \eqref{compa300} for some $\theta\in\WD(\Omega)$, namely
\begin{equation}\label{cauchyepA}
\begin{cases}
\curl\Curl\epsilon=-\theta&\text{in $\Omega$}\\
\Dive\C\epsilon=0&\text{in $\Omega$}\\
\C\epsilon\,n=0&\textrm{on $\partial\Omega$}\,,
\end{cases}
\end{equation} 
then, by \eqref{divenulla} and \eqref{conversions}, the Airy stress function $v$ satisfies the system
\begin{equation}\label{cauchyvA}
\begin{cases}
\displaystyle \frac{1-\nu^2}{E}\Delta^2v=-\theta&\textrm{in $\Omega$}\\[2mm]
\nabla^2v\,t=0&\textrm{on $\partial\Omega$\,.}
\end{cases}
\end{equation}
%
Recalling that \eqref{compa300} holds in the sense of distributions, the study of the regularity of the fields~$\epsilon$ and~$\sigma$ in the laboratory setting and of the Airy stress function~$v$ must be carried out carefully. The reason is the following: the measure of the elastic incompatibility $\theta\in\WD(\Omega)$ is an element of the space $H^{-2}(\Omega)$, so that it is natural to expect that $\epsilon,\sigma\in L^2(\Omega;\R^{2\times2}_{\sym})$ and that $v\in H^2(\Omega)$. At this level $\C\epsilon\,n|_{\partial\Omega}$ and $\nabla^2v\,t|_{\partial\Omega}$ make sense only as elements of $H^{-\frac{1}{2}}(\partial\Omega;\R^{2})$, so that the definition of the boundary conditions in \eqref{cauchyepA} and \eqref{cauchyvA} cannot be intended in a pointwise sense, even when the tangent and normal vectors are defined pointwise.

In Corollaries \ref{prop:airyepsilonA} and \ref{prop:airyepsilon} below, we establish the equivalence of problems \eqref{cauchyepA} and \eqref{cauchyvA} and we show that, under suitable assumptions on the regularity of $\partial\Omega$\,, the boundary conditions hold in the sense of $H^{\frac 1 2}(\partial\Omega;\R^2)$\,. 
To this purpose, we introduce the function $\bar v\in H^2_\loc(\R^2)$ defined by
\begin{equation}\label{fundamdiscl}
\bar v(x)\coloneqq\begin{cases}
\displaystyle \frac{E}{1-\nu^2}\frac{|x|^2}{16\pi}\log|x|^2 & \text{if $x\neq0$}\\[2mm]
0 & \text{if $x=0$}
\end{cases}
\end{equation}
as the fundamental solution to the equation 
\begin{equation}\label{fundbd}
\frac{1-\nu^2}E\Delta^2v=\de_0\quad\text{in $\R^2$\,.}
\end{equation}
Given $\theta=\sum_{k=1}^{K}s^k\de_{y^k}\in\WD(\Omega)$, 
for every $k=1,\ldots,K$, we let $v^{k}(\cdot)\coloneqq -s^k\bar v(\cdot -y^{k})\res\Omega$ and define
\begin{equation}\label{plastic_parts}
v^p\coloneqq \sum_{k=1}^{K} v^k\,,\quad
\sigma^p\coloneqq\sigma^p[v^p]=\Airy(v^p)=\sum_{k=1}^K\Airy(v^k),\quad
\epsilon^p\coloneqq\epsilon^p[v^p]=\C^{-1}\sigma^p[v^p]=\C^{-1}\sigma^p\,,
\end{equation}
which we are going to refer to as the \emph{plastic contributions}.
Notice that, by construction, $v^p$ is smooth in $\R^2\setminus\supp\theta$ and hence on $\partial\Omega$ and so are $\sigma^p$ and $\epsilon^p$\,.

Recalling \eqref{fundamdiscl} and \eqref{fundbd}, we see that 
\begin{equation}\label{bilaplacian_vp}
\frac{1-\nu^2}E\Delta^2v^p=-\theta\qquad \text{in $\Omega$\,,}
\end{equation} 
so that, if $v$ solves the equation in \eqref{cauchyvA} and we define the function $v^e$ through the additive decomposition
\begin{equation}\label{add_dec_v}
v\coloneqq v^p+v^e\,,
\end{equation} 
then, $v^e$ satisfies
\begin{equation}\label{bastaquesta}
\begin{cases}
\displaystyle \frac{1-\nu^2}E\Delta^2v^e=0&\text{in $\Omega$}\\[2mm]
\displaystyle \nabla^2v^e\,t=-\nabla^2v^p\,t&\text{on $\partial\Omega$\,.}
\end{cases}
\end{equation} 
Therefore, by \eqref{bilaplacian_vp}, we can find a solution $v$ to problem \eqref{cauchyvA} if and only if we find a solution to problem \eqref{bastaquesta}.
Similarly, by \eqref{airy2}, 
\begin{equation}\label{inc_ep}
\curl\Curl\epsilon^p=-\theta\qquad \text{in $\Omega$\,,}
\end{equation} 
so that if $\epsilon$ solves the equation in \eqref{cauchyepA} and we define the field $\epsilon^e$ through the additive decomposition 
\begin{equation}\label{add_dec_epsilon}
\epsilon\coloneqq \epsilon^p+\epsilon^e\,,
\end{equation} 
then we have $\curl\Curl\epsilon^e=0$ in $\Omega$ and $\C\epsilon^e\,n=-\C\epsilon^p\,n$ on $\partial\Omega$. 
Therefore, by \eqref{inc_ep}, we find a solution $\epsilon$ to problem \eqref{cauchyepA} if and only if we find a solution to problem
\begin{equation}\label{cauchy_ee}
\begin{cases}
\curl\Curl\epsilon^e=0&\text{in $\Omega$}\\
\Dive\C\epsilon^e=0&\text{in $\Omega$}\\
\C\epsilon^e\,n=-\C\epsilon^p\,n&\textrm{on $\partial\Omega$\,,}
\end{cases}
\end{equation}
where we notice that the second equation above is automatically satisfied by \eqref{divenulla}, in view of the fact that  $\Div\,\sigma^p=0$\,.

We refer to~$v^e$ and~$\epsilon ^e$ as to the \emph{elastic contributions} and we notice that they are compatible fields.
Upon noticing that the function~$\bar v$ is smooth in $\R^2\setminus\{0\}$ and by requiring that the boundary~$\partial\Omega$ be smooth enough, we will see that problems \eqref{bastaquesta} and \eqref{cauchy_ee} are ``equivalent'' and that they admit solutions which are regular enough for the boundary conditions to make sense in $H^{\frac 1 2}(\partial\Omega;\R^2)$\,.

\begin{definition}\label{defdeb}
\rm{
Let $\Omega\subset\R^2$ be a bounded, simply connected, open set.
We say that a function $\epsilon\in L^2(\Omega;\R^{2\times 2}_{\sym})$ (resp., $\epsilon^e\in L^2(\Omega;\R^{2\times 2}_{\sym})$) is a weak solution to
\eqref{cauchyepA} (resp., \eqref{cauchy_ee}) if the first equation is satisfied when tested with $H^2_0(\Omega)$ functions and the second  one is satisfied when tested with $H^1(\Omega;\R^2)$ functions.
 Analogously, we say that a function $v\in H^2(\Omega;\R^2)$ is a weak solution to 
\eqref{cauchyvA} (resp., \eqref{bastaquesta}) if the first equation holds when tested with 
$H^{2}(\Omega)$ functions satisfying the boundary condition.
}
\end{definition}

We start by proving the following result, which is one implication in the equivalence of problems \eqref{cauchy_ee} and \eqref{bastaquesta}.

\begin{proposition}\label{prop:airyepsilonA_soloel}
Let $\Omega\subset\R^2$ be a bounded, simply connected, open set with boundary of class $C^4$ and  let $\theta\in\WD(\Omega)$. Then there exists a unique weak solution (in the sense of Definition \ref{defdeb}) $\epsilon^e\in L^2(\Omega;\R^{2\times 2}_{\sym})$ to 
\eqref{cauchy_ee}. Furthermore, $\epsilon^e\in H^2(\Omega;\R^{2\times 2}_{\sym})$\,. 
Moreover, there exists a function $v^e\in H^4(\Omega)$ such that 
$\epsilon^e=\C^{-1}\Airy(v^e)$.  Finally, any function $v^e\in H^4(\Omega)$, with $\epsilon^e=\C^{-1}\Airy(v^e)$, is a weak solution to 
\eqref{bastaquesta}.
\end{proposition}
\begin{proof}
Let 
\begin{equation*}
\mathsf{E}(\Omega)\coloneqq\{\epsilon\in L^2(\Omega;\R^{2\times 2}_{\sym})\,:\,\curl\Curl\epsilon=0\textrm{ in }H^{-2}(\Omega;\R^{2\times 2})\}
\end{equation*}
and let $G\colon \mathsf{E}(\Omega)\to \R$ be the functional defined by
\begin{equation*}
G(\epsilon)\coloneqq\frac 1 2\int_{\Omega}\C\epsilon:\epsilon\,\ud x+\int_{\Omega} \sigma^p:\epsilon\,\ud x\,,
\end{equation*}
where $\sigma^p$ is defined in \eqref{plastic_parts}. By construction, $G$ is bounded from below in $L^2(\Omega;\R^{2\times 2}_{\sym})$ and $\mathsf{E}(\Omega)$ is a closed subspace of $L^2(\Omega;\R^{2\times 2}_{\sym})$\,.  Therefore, by applying the direct method of Calculus of Variations, $G$ admits a unique minimizer $\epsilon^e$ in $\mathsf{E}(\Omega)$\,.
Now we show that  \eqref{cauchy_ee} is the Euler-Lagrange equation for~$G$\,. Indeed, for any $\eta\in \mathsf{E}(\Omega)$ we have that
\begin{equation}\label{9febbr}
\int_{\Omega}(\C\epsilon^e+\sigma^p):\eta\,\ud x=0;
\end{equation}
invoking Proposition \ref{sv} we have that $\eta=\nabla^{\sym}u$ for some $u\in H^1(\Omega;\R^2)$\,; therefore, since $\C\epsilon^e+\sigma^p\in \R^{2\times 2}_{\sym}$\,, integrating by parts \eqref{9febbr}, we get
\begin{equation*}
0=\int_{\Omega}(\C\epsilon^e+\sigma^p):\nabla u\,\ud x=-\int_{\Omega} u\cdot \Div(\C\epsilon^e+\sigma^p)\,\ud x+\int_{\partial\Omega}(\C\epsilon^e+\sigma^p)n\cdot u\,\ud\Huno\,,
\end{equation*}
which, recalling that $\Div\,\sigma^p=0$\,, by the fundamental lemma of Calculus of Variations, implies that~$\epsilon^e$ satisfies \eqref{cauchy_ee}.
Moreover, by standard regularity results, we have that $\epsilon^e\in H^2(\Omega;\R^{2\times 2}_{\sym})$\,. 
Now we can apply \cite[Theorem 5.6-1(a)]{ciarlet97}, and in particular the argument in \cite[page~397]{ciarlet97}, which guarantees that a strain field $\epsilon^e\in H^m(\Omega;\R^{2\times2}_{\sym})$ admits an Airy stress function $v^e=\Airy^{-1}(\epsilon^e)\in H^{m+2}(\Omega)$, for every $m\geq0$. By applying this result with $m=2$, we obtain that $v^e\in H^4(\Omega)$\,.
Finally, by \eqref{conversions}, we have that any function $v^{e}\in H^4(\Omega)$ with  $\epsilon^e=\C^{-1}\Airy(v^e)$ is a weak solution to \eqref{bastaquesta}. 
\end{proof}
Since $\epsilon^p$ and $v^p$ are smooth in a neighborhood of the boundary of $\Omega$\,, by \eqref{plastic_parts} and by Proposition~\ref{prop:airyepsilonA_soloel}, we immediately deduce the following result.
\begin{corollary}\label{prop:airyepsilonA}
Let $\Omega\subset\R^2$ be a bounded, simply connected, open set with boundary of class~$C^4$ and  let $\theta\in\WD(\Omega)$. 
Then there exists a unique weak solution (in the sense of Definition~\ref{defdeb}) $\epsilon\in L^2(\Omega;\R^{2\times 2}_{\sym})$ to 
\eqref{cauchyepA}. Furthermore, $\C\epsilon^e\,n
\in H^{\frac 3 2}(\partial\Omega;\R^2)$\,. 
Moreover, there exists a function $v\in H^2(\Omega)$ such that 
$\epsilon=\C^{-1}\Airy(v)$.  Finally, $v$ is a weak solution to 
\eqref{bastaquesta}
and $\nabla^2v\,t
\in H^{\frac 3 2}(\partial\Omega;\R^2)$\,.
\end{corollary}

In order to prove the converse implication of Proposition~\ref{prop:airyepsilonA_soloel}, we state the following result, which is an immediate consequence of \cite[Theorem~2.20]{Gazzola09} (applied with $k=4$, $m=n=p=2$, and with $f\equiv 0$ and $h_j\in C^{\infty}$)\,.
\begin{lemma}\label{20220422_00}
Let $A\subset\R^2$ be a bounded open set with boundary of class $C^4$ and let $g$ be a $C^\infty$ function in a neighborhood of $\partial A$\,.
Then there exists a unique weak solution $w\in H^2(A)$ to 
\begin{equation}\label{20220422_10}
\begin{cases}
\displaystyle \frac{1-\nu^2}{E}\Delta^2w=0&\text{in }A\,,\\[2mm]
\displaystyle w=g&\textrm{on }\partial A\,,\\[2mm]
\displaystyle \partial_nw=\partial_ng&\textrm{on }\partial A\,.
\end{cases}
\end{equation}
Moreover, $w\in H^4(A)$\,. 
\end{lemma}
By Lemma \ref{20220422_00} and Proposition \ref{2101141730} below, we have the following result.
\begin{corollary}\label{20220422}
Let $A\subset\R^2$ be a bounded open set with boundary of class $C^4$ and let $f$ be a $C^\infty$ function in a neighborhood of $\partial A$\,.
Let $\Gamma^0,\Gamma^1,\ldots,\Gamma^L$ be the connected components of~$\partial A$\,.
Given $a^0, a^1,\ldots, a^L$ affine functions, there exists a unique weak solution $w\in H^2(A)$ to the problem
\begin{equation}\label{20230508}
\begin{cases}
\displaystyle \frac{1-\nu^2}{E}\Delta^2w=0&\text{in }A\,,\\[2mm]
\displaystyle  w=f+a^l&\textrm{on }\Gamma^l\,, \\[2mm]
\displaystyle \partial_n w=\partial_{n} (f+a^l)&\textrm{on }\Gamma^l\,;
\end{cases}
\end{equation} 
moreover, $w\in H^4(A)$ and satisfies
\begin{equation}\label{20220422_1}
\begin{cases}
\displaystyle \frac{1-\nu^2}{E}\Delta^2w=0&\text{in }A\,,\\[2mm]
\displaystyle \nabla^2 w\,t=\nabla^2f\,t&\textrm{on }\partial A\,.
\end{cases}
\end{equation}
Viceversa, 
 if $w\in H^4(A)$ is a solution to \eqref{20220422_1}, then there exist $a^0, a^1,\ldots, a^L$ affine functions such that
$w$ satisfies \eqref{20230508}. 
\end{corollary}
\begin{proof}
By Lemma \ref{20220422_00}, applied with $g\coloneqq f+a^l$ on $\Gamma^l$\,, we have that there exists a unique weak solution $w\in H^2(A)$ to \eqref{20230508}
and that $w\in H^4(A)$\,.
By the Rellich--Kondrakov Theorem, we have that $w\in C^2(\overline{A})$\,; therefore $w-f$ is of class $C^2$ in a neighborhood of $\partial A$\,. 
We can now apply Proposition \ref{2101141730} to deduce that $\nabla^2w\,t=\nabla^2 f\,t$ on $\partial A$\,, thus obtaining that $w\in H^4(A)$ is a solution to \eqref{20220422_1}.

Viceversa, if $w\in H^4(A)$ is a solution to \eqref{20220422_1}, then, 
by using Proposition~\ref{2101141730} again, we obtain that there exist $a^0,a^1,\ldots,a^L$ affine functions such that $w$ solves \eqref{20230508}.
\end{proof}
\begin{proposition}\label{prop:airyepsilon_soloel}
Let $\Omega\subset\R^2$ be a bounded, simply connected, open set with boundary of class $C^4$ and let $\theta\in\WD(\Omega)$\,. 
Then there exists a weak solution $v^e\in H^4(\Omega)$ to \eqref{bastaquesta}. Furthermore,
any weak solution  $v^e$ to \eqref{bastaquesta}
belongs to $H^4(\Omega)$ and
 the function $\epsilon^e=\C^{-1}\Airy(v^e)$ is the unique weak solution to 
 \eqref{cauchy_ee}. 
\end{proposition}
\begin{proof}
By applying Corollary \ref{20220422} with $f=-v^p$ (with $v^p$ defined in \eqref{plastic_parts}) and $A= \Omega$\,, we immediately have that there exists a weak solution $w\in H^4(\Omega)$ to \eqref{bastaquesta}.
Moreover, by \eqref{conversions}, we have that for any weak solution $v^e\in H^2(\Omega)$ to \eqref{bastaquesta}, the function
 $\epsilon^e=\C^{-1}\Airy(v^e)\in L^2(\Omega;\R^{2\times 2}_{\sym})$ is a weak solution to 
\eqref{cauchy_ee}. 
Owing to Proposition~\ref{prop:airyepsilonA_soloel}, the solution $\epsilon^e$ to \eqref{cauchy_ee} is unique and belongs to $H^4(\Omega)$\,. It follows that any weak solution $v^e$ to \eqref{bastaquesta} is actually in $H^4(\Omega)$\,.
\end{proof}
\begin{corollary}\label{prop:airyepsilon}
Let $\Omega\subset\R^2$ be a bounded, simply connected, open set with boundary of class $C^4$ and let $\theta\in\WD(\Omega)$\,. 
Then there exists a weak solution $v\in H^2(\Omega)$ to \eqref{cauchyvA} and the condition $\nabla^2v\,t=0$ on $\partial\Omega$ holds in $H^{\frac 3 2}(\partial\Omega;\R^2)$\,. 
 Furthermore,
for any weak solution  $v$ to \eqref{cauchyepA},
 the function $\epsilon=\C^{-1}\Airy(v)$ is the unique  weak solution to 
 \eqref{cauchyepA}. 
\end{corollary}
Finally, by arguing as in the proof of Proposition \ref{prop:airyepsilon_soloel} and using the additive decomposition in \eqref{add_dec_v}, one can easily prove the following result.
\begin{proposition}\label{20220421}
Let $A\subset\R^2$ be a bounded open set with boundary of class $C^4$\,. Let $\theta\in\WD(A)$ and let $v\in H^2(A)$ be such that
\begin{equation}\label{2204061255}
\frac{1-\nu^2}{E}\Delta^2v=-\theta\qquad\text{in $A$.}
\end{equation}
Then denoting by $\Gamma^0,\Gamma^1,\ldots,\Gamma^L$ the connected components of $\partial A$\,, we have that 
\begin{equation}\label{20220421_1}
\nabla^2v\,t=0\quad\textrm{on }\partial A\quad\Leftrightarrow\quad v=a^l\,, \quad\partial_n v=\partial_n a^l\quad\textrm{on }\Gamma^l\,,\quad\textrm{ for every }l=0,1,\ldots,L\,,
\end{equation}
where $a^0,a^1,\ldots,a^L$ are affine functions.
\end{proposition}
\begin{remark}
{\rm 
 We highlight that, since $\Omega$ is simply connected, the solution to \eqref{bastaquesta} is unique up to an affine function.
Indeed, given two solutions $v^e$ and $\tilde v^e$ of \eqref{bastaquesta}\,, 
the function $w\coloneqq v^e-\tilde v^e$ satisfies
\begin{equation*}
\begin{cases}
\displaystyle \frac{1-\nu^2}E\Delta^2w=0&\text{in $\Omega$}\\[2mm]
\displaystyle \nabla^2w\,t=0&\text{on $\partial\Omega$\,.}
\end{cases}
\end{equation*} 
Moreover, by Proposition \ref{20220421}, such $w$ is affine. Therefore,
 the function $v^e$ satisfying \eqref{bastaquesta} is uniquely determined up to affine functions.
 Analogously,  the function $v$ satisfying \eqref{cauchyvA} is uniquely determined up to affine functions.
}
\end{remark}


\section{Finite systems of isolated disclinations}\label{sc:isolated_disclinations}
We now study the equilibrium problem for a finite family of isolated disclinations in a body $\Omega$. 
The natural idea would be to consider the minimum problem for the elastic energy $\G$ defined in \eqref{energyairy} under the incompatibility constraint \eqref{cauchyvA}\,, associated with a measure $\theta\in\WD(\Omega)$\,; 
however, this is inconsistent, since one can easily verify that the Euler--Lagrange equation for $\G$ is $\Delta^2v=0$\,.

To overcome this inconsistency, we define a suitable functional which embeds the presence of the disclinations and whose Euler--Lagrange equation is given by \eqref{cauchyvA}. 
To this purpose, let $\Omega\subset\R^2$ be a bounded, open, and simply connected set with boundary of class $C^4$\,;
 for every $\theta\in\WD(\Omega)$ let $\I^\theta\colon H^2(\Omega)\to \R$ be the functional defined by
\begin{equation}\label{defI}
\I^{\theta}(v;\Omega)\coloneqq\G(v;\Omega)+\langle \theta,v\rangle\,,
\end{equation}
and consider the minimum problem
\begin{equation}\label{minI}
\min \big\{\I^\theta(v;\Omega) : \text{$v\in H^2(\Omega)$\,, $\nabla^2v\,t=0$ on $\partial\Omega$}\big\}\,.
\end{equation}
A simple calculation shows that the Euler--Lagrange equation for the functional \eqref{defI}, with respect to variations in $H^2_0(\Omega)$\,, is given by \eqref{cauchyvA}.
By Proposition \ref{20220421}, we deduce that the minimum problem in \eqref{minI} is equivalent, up to an affine function, to the minimum problem
\begin{equation}\label{minIaff}
\min \{\I^\theta(v;\Omega)\,:\,v\in H_0^2(\Omega)\}\,.
\end{equation}
\begin{lemma}\label{propItheta}
For every $\theta\in\WD(\Omega)$, the functional $\I^\theta(\cdot;\Omega)$ is strictly convex in $H^2(\Omega)$ and it is  bounded below and coercive in $H^2_0(\Omega)$\,.
As a consequence, the minimum problem \eqref{minIaff} has a unique solution. 
\end{lemma}
\begin{proof}
We start by proving that $\I^\theta(\cdot;\Omega)$ is bounded below and coercive in $H^2_0(\Omega)$\,. To this purpose, 
we first notice that there exists a constant $C_1=C_1(\nu,E,\Omega)>0$ such that for every $v\in H^2_0(\Omega)$
\begin{equation}\label{quadr}
 \G(v;\Omega)\ge\frac 1 2\frac{1-\nu^2}{E}\min\{1-2\nu,1\}\|\nabla^2 v\|^2_{L^2(\Omega;\R^{2\times 2})}\ge C_1\|v\|^2_{H^2(\Omega)}\,,
\end{equation}
where in the last passage we have used Friedrichs's inequality in $H^2_0(\Omega)$\,. Notice that the positivity of $C_1$ is a consequence of \eqref{lame3}.

Now, using that $H^2_0(\Omega)$ embeds into $C^0(\Omega)$\,, we have that there exists a constant $C_2=C_2(\theta,\Omega)>0$ such that for every $v\in H^2_0(\Omega)$
\begin{equation}\label{linear} 
\langle \theta,v\rangle=\sum_{k=1}^{K}s^kv(x^k)\ge -C_2\|v\|_{H^2(\Omega)}\,.
\end{equation}
By \eqref{quadr} and \eqref{linear}, we get that for every $v\in H^2_0(\Omega)$
\begin{equation}\label{pincapalla}
\I^\theta(v;\Omega)\ge  C_1\|v\|^2_{H^2(\Omega)}-C_2\|v\|_{H^2(\Omega)}\ge
 -\frac{C_2^2}{4C_1}\,,
\end{equation}
which implies boundedness below and coercivity of $\I^\theta(\cdot;\Omega)$ in $H^2_0(\Omega)$\,.

Now we show that $\G(\cdot;\Omega)$ is strictly convex in $H^2(\Omega)$\,, which, together with the linearity of the map $v\mapsto\langle\theta,v\rangle$\,, implies the strict convexity of $\I^\theta(\cdot;\Omega)$ in $H^2(\Omega)$\,. To this purpose, let $v,w\in H^2(\Omega)$ with $v\neq w$ and let $\lambda\in (0,1)$\,; then a simple computation shows that
\begin{equation}\label{strictconv}
\begin{split}
\G(\lambda v+(1-\lambda)w;\Omega)=&\,\lambda\G(v;\Omega)+(1-\lambda)\G(w;\Omega)-\lambda(1-\lambda)\G(v-w;\Omega)\\
<&\, \lambda\G(v;\Omega)+(1-\lambda)\G(w;\Omega)\,,
\end{split}
\end{equation}
which is the strict convexity condition.

By the direct method of the Calculus of Variations, problem \eqref{minIaff} has a unique solution.
\end{proof}
\begin{remark}
\rm{
We highlight that inequality \eqref{pincapalla} shows that $\I^\theta(\cdot;\Omega)$ could be negative.
In particular, being $\G$ non-negative, the sign of $\I^\theta$ is determined by the value of the linear contribution $\langle \theta,v\rangle$\,. It follows that the minimum problem \eqref{minI} and hence \eqref{minIaff} are non trivial and, as we will see later (see, e.g., \eqref{minenball}), the minimum of $\I^\theta(\cdot;\Omega)$ is indeed negative.
}
\end{remark}
\begin{remark}\label{equinorm}
\rm{
Notice that the functional $\G^{\frac 1 2}(\cdot;\Omega)$ defines a seminorm on $H^2(\Omega)$ and a norm in $H^2_0(\Omega)$\,, since $\G(v;\Omega)\equiv\langle v,v\rangle_{\G_\Omega} $ where the product $\langle \cdot,\cdot\rangle_{\G_\Omega}$\,, defined by 
\begin{equation}\label{semiprod}
\langle v,w\rangle_{\G_\Omega}\coloneqq \frac 1 2 \frac{1+\nu}{E}\int_{\Omega} \big(\nabla^2v:\nabla^2 w-\nu \Delta v\Delta w\big)\,\ud x\,,
\end{equation}
is a bilinear, symmetric, and positive semidefinite form in $H^2(\Omega)$ and positive definite in $H^2_0(\Omega)$\,.
We remark that in $H^2_0(\Omega)$ the norm $\G^{\frac 1 2}(\cdot;\Omega)$ is equivalent to the standard norm $\|\cdot\|_{H^2(\Omega)}$\,.
}
\end{remark}
In the following lemma, for any given $\xi\in\R^2$ and $R>0$, we compute the minimal value of $\I^\theta(\cdot;B_R(\xi))$ associated with a single disclination located at $\xi$, corresponding to $\theta=s\de_\xi$ for some $s\in\R\setminus\{0\}$\,. 
The explicit computation is straightforward and is omitted.
\begin{lemma}\label{sub:single}
Let $s\in\R\setminus\{0\}$\,, $\xi\in\R^2$\,, and $R>0$.
The function $v_R\colon \overline{B}_R(\xi)\to\R$ 
defined by
\begin{equation*} 
\begin{aligned}
v_R(x)\coloneqq & -s\bar v (x-\xi)-s\frac{E}{1-\nu^2}\frac{R^2-|x-\xi|^2(1+\log R^2)}{16\pi}\\
=&-sR^2\bigg(\bar v \Big(\frac {x-\xi} R\Big)+\frac{E}{1-\nu^2}\frac{1}{16\pi}\Big(1-\Big|\frac{x-\xi}{R}\Big|^2\Big)\bigg)\,,
\end{aligned}
\end{equation*}
with $\bar v$ as in \eqref{fundamdiscl}, belongs to $H^2(B_R(\xi))\cap {C^\infty(B_R(\xi)\setminus\{\xi\})}$ and solves 
\begin{equation}\label{cauchyvR}
\begin{cases}
\displaystyle \frac{1-\nu^2}{E}\Delta^2v=-s\de_\xi&\text{in $B_R(\xi)$}\\[2mm] 
\displaystyle v=\partial_nv=0&\text{on $\partial B_R(\xi)$\,.}
\end{cases}
\end{equation}
Hence $v_R$ is the only minimizer of problem \eqref{minIaff} for $\Omega=B_R(\xi)$ and $\theta=s\de_\xi$\,.
Moreover,
\begin{equation}\label{energyfinite}
\G(v_R;B_R(\xi))=\frac{E}{1-\nu^2}\frac{s^2R^2}{32\pi}\qquad\textrm{and}\qquad \langle s\de_\xi,v_R\rangle=-\frac{E}{1-\nu^2}\frac{s^2R^2}{16\pi}\,,
\end{equation}
so that
\begin{equation}\label{minenball}
\min_{v\in H^2_0(B_R(\xi))}\I^{s\de_\xi}(v; B_R(\xi))=\I^{s\de_\xi}(v_R; B_R(\xi))=-\frac{E}{1-\nu^2}\frac{s^2R^2}{32\pi}\,.
\end{equation}
\end{lemma}
In view of \eqref{energysigma} and \eqref{energyairy}, the first equality in \eqref{energyfinite} is the stored elastic energy of a single disclination located at the center of the ball $B_R(\xi)$. Observe that, according to the formulation of the mechanical equilibrium problem in the Airy variable \eqref{cauchyvR}, the contribution of the linear term in the second equality in \eqref{energyfinite} adds to the total energy functional of the system, but does not correspond to an energy of elastic nature.

\section{Dipole of disclinations}\label{sub:dipole}
In \eqref{energyfinite} we have seen that an isolated disclination in the center of a ball of radius $R$ carries an elastic energy of the order $R^2$\,.
Here we show that the situation dramatically changes when considering a dipole of disclinations with opposite signs; indeed, when the distance between the disclinations vanishes, a dipole of disclinations behaves like an edge dislocation and its elastic energy is actually of the order $\log R$\,.

\subsection{Dipole of disclinations in a ball}\label{sub:dipole1}
For every $h>0$ let 
\begin{equation}\label{poledipoles}
y^{h,\pm}\coloneqq\pm\frac h 2(1;0)
\end{equation} 
and
let $\bar v_h\colon\R^2\to \R$ be the function defined by
\begin{equation}\label{defvbarh}
\bar v_h(x)\coloneqq -s(\bar v(x-y^{h,+})-\bar v(x-y^{h,-}))\,,
\end{equation}
where $\bar v$ is given in \eqref{fundamdiscl}. 
By construction, $\bar v_h\res{B_R(0)}\in H^2(B_R(0))$ and  
\begin{equation}\label{bilaplvbarh}
\Delta^2\bar v_h=-\theta_h\qquad\textrm{in }\R^2\,,
\end{equation}
where we have set
\begin{equation}\label{thetah}
\theta_h\coloneqq s\big(\de_{y^{h,+}}-\de_{y^{h,-}}\big)\,.
\end{equation}
We start by proving that the $H^2$ norm of $\bar v_h$ in an annulus $A_{r,R}(0)\coloneqq B_R(0)\setminus \overline B_r(0)$ with fixed radii $0<r<R$ vanishes as $h\to 0$.
\begin{lemma}\label{lemma:edgetrue}
For every $0<r<R$ there exists a constant $C(r,R)$ such that
\begin{equation}\label{f:edgetrue}
\lim_{h\to 0}\frac{1}{h^2}\|\bar v_h\|^2_{H^2(A_{r,R}(0))}=C(r,R)s^2\,.
\end{equation}
\end{lemma}
\begin{proof}
{
By straightforward computations, for every $x\in A_{h,R}(0)$ we have that
\begin{equation}\label{perdislo}
\begin{split}
\bar v_h(x)=&\,-\frac{E}{1-\nu^2}\frac{s}{16\pi}\bigg(\Big(|x|^2+\frac{h^2}{4}\Big)\log\frac{\big(x_1-\frac h 2\big)^2+x_2^2}{\big(x_1+\frac h 2\big)^2+x_2^2}\\
&\,\phantom{-\frac{E}{1-\nu^2}\frac{s}{16\pi}\bigg(}\quad-hx_1\log\bigg(\Big(\Big(x_1-\frac h 2\Big)^2+x_2^2\Big)\Big(\Big(x_1+\frac h 2\Big)^2+x_2^2\Big)\bigg)\bigg)\,;\\
\partial_{x_1}{\bar v_h}(x)=&\, -\frac{E}{1-\nu^2}\frac{s}{16\pi} \bigg(2\Big(x_1-\frac h 2\Big)\log\Big(\Big(x_1-\frac h 2\Big)^2+x_2^2\Big)\\
&\,\phantom{\frac{E}{1-\nu^2}\frac{s}{16\pi} \bigg(}\quad-2\Big(x_1+\frac h 2\Big)\log\Big(\Big(x_1+\frac h 2\Big)^2+x_2^2\Big)-2h\bigg)\,;\\
\partial_{x_2}{\bar v_h}(x)=&\, -\frac{E}{1-\nu^2}\frac{s}{16\pi} 2x_2\log\frac{\Big(x_1-\frac h 2\Big)^2+x_2^2}{\Big(x_1+\frac h 2\Big)^2+x_2^2}\,;\\
\partial^2_{x_1^2}\bar v_h(x)=&\, -\frac{E}{1-\nu^2}\frac{s}{16\pi}\Bigg(2\log\frac{\big(x_1-\frac{h}{2}\big)^2+x_2^2}{\big(x_1+\frac{h}{2}\big)^2+x_2^2}\\
&\,\phantom{-\frac{E}{1-\nu^2}\frac{s}{16\pi}\Bigg(}\quad+4\bigg(\frac{\big(x_1-\frac{h}{2}\big)^2}{\big(x_1-\frac h 2\big)^2+x_2^2}-\frac{\big(x_1+\frac{h}{2}\big)^2}{\big(x_1+\frac{h}{2}\big)^2+x_2^2}\bigg)\Bigg)\,;\\
\partial^2_{x_2^2}\bar v_h(x)=&\,-\frac{E}{1-\nu^2}\frac{s}{16\pi}\Bigg(2\log\frac{\big(x_1-\frac{h}{2}\big)^2+x_2^2}{\big(x_1+\frac{h}{2}\big)^2+x_2^2}\\
&\,\phantom{-\frac{E}{1-\nu^2}\frac{s}{16\pi}\Bigg(}\quad+4x_2^2\bigg(\frac{1}{\big(x_1-\frac{h}{2})^2+x_2^2}-\frac{1}{\big(x_1+\frac{h}{2})^2+x_2^2}\bigg)\Bigg)\,;\\
\partial^2_{x_1\,x_2}\bar v_h(x)=&\,-\frac{E}{1-\nu^2}\frac{s}{16\pi}\Bigg(4x_2\bigg(\frac{x_1-\frac{h}{2}}{\big(x_1-\frac{h}{2}\big)^2+x_2^2}-\frac{x_1+\frac{h}{2}}{\big(x_1+\frac{h}{2}\big)^2+x_2^2}\bigg)\Bigg)\,.
\end{split}
\end{equation}
Moreover,
\begin{equation}\label{simplif}
\begin{aligned}
\log\frac{\big(x_1-\frac h 2\big)^2+x_2^2}{\big(x_1+\frac h 2\big)^2+x_2^2}=&\log\bigg(1-\frac{2x_1h}{\big(x_1+\frac h 2\big)^2+x_2^2}\bigg)\,,\\
\frac{\big(x_1-\frac{h}{2})^2}{\big(x_1-\frac{h}{2}\big)^2+x_2^2}-\frac{\big(x_1+\frac{h}{2}\big)^2}{\big(x_1+\frac{h}{2}\big)^2+x_2^2}=&\frac{-2x_2^2x_1h}{\Big(\big(x_1-\frac{h}{2}\big)^2+x_2^2\Big)\Big(\big(x_1+\frac{h}{2}\big)^2+x_2^2\Big)}\,,\\
x_2^2\bigg(\frac{1}{\big(x_1-\frac{h}{2}\big)^2+x_2^2}-\frac{1}{\big(x_1+\frac{h}{2}\big)^2+x_2^2}\bigg)=&\frac{2x_2^2x_1h}{\Big(\big(x_1-\frac{h}{2}\big)^2+x_2^2\big)\Big(\big(x_1+\frac{h}{2}\big)^2+x_2^2\Big)}\,,\\
\frac{x_1-\frac{h}{2}}{\big(x_1-\frac{h}{2}\big)^2+x_2^2}-\frac{x_1+\frac{h}{2}}{\big(x_1+\frac{h}{2}\big)^2+x_2^2}=&\frac{\big(|x|^2+\frac{h^2}{4}\big)h}{\Big(\big(x_1-\frac{h}{2}\big)^2+x_2^2\Big)\Big(\big(x_1+\frac{h}{2}\big)^2+x_2^2\Big)}\,.
\end{aligned}
\end{equation}
Finally, we set 
$$\bar v'(x)\coloneqq \frac{E}{1-\nu^2}\frac{s}{8\pi} (x_1\log|x|^2+x_1)\,, \quad\text{for every $x\in\R^2\setminus\{0\}$\,,}$$ 
and we observe that $\bar v'\in C^{\infty}(\R^2\setminus\{0\})$\,.
By \eqref{perdislo} and \eqref{simplif}, it is easy to check that $h^{-1}\bar v_h$ (resp., $h^{-1} \partial_{x_j}\bar v_h$\,, for $j=1,2$\,, and $h^{-1}\partial^{2}_{x_jx_k}\bar v_h$\,, for $j,k=1,2$) converges to $\bar v'$ (resp., $\partial_{x_j}\bar v'$\,, for $j=1,2$\,, and $\partial^{2}_{x_jx_k}\bar v'$\,, for $j,k=1,2$) uniformly in $A_{r,R}(0)$ as $h\to0$\,.
In particular, $h^{-2}\|\bar v_h\|^2_{H^2(A_{r,R}(0))}\to\|\bar v'\|^2_{H^2(A_{r,R}(0))}\eqqcolon C(r,R) s^2$\,, i.e.,  \eqref{f:edgetrue}.
}
\end{proof}
The next lemma is devoted to the asymptotic behavior of the elastic energy of $\bar{v}_h$ as $h\to 0$\,.
Its proof is contained in Appendix \ref{appendixprooflemma}.
\begin{lemma}\label{lemma:energyvbarh}
For every $R>0$
\begin{equation}\label{energyvbarh}
\lim_{h\to 0}\frac{1}{h^2\log \frac{R}{h}}\G(\bar v_h;B_R(0))=\frac{E}{1-\nu^2}\frac{s^2}{8\pi}\,.
\end{equation}
\end{lemma}
The next proposition shows that the same behavior in \eqref{energyvbarh} persists when replacing $\bar v_h$ with the minimizer $v_h$ of $\I^{\theta_h}(\cdot;B_R(0))$ in $H^2_0(B_R(0))$\,, for $\theta_h$ given by \eqref{thetah}.

\begin{proposition}\label{p:3.3}
For every $0<h<R$\,, let 
 $v_h$ be the minimizer of $\I^{\theta_h}(\cdot;B_R(0))$ in $H^2_0(B_R(0))$\,.
Then, 
\begin{equation}\label{plim}
\lim_{h\to 0}\frac{1}{h^2|\log h|}\G(v_h;B_R(0))=\frac{E}{1-\nu^2}\frac{s^2}{8\pi}
\end{equation}
and
\begin{equation}\label{slim}
\lim_{h\to 0}\frac{1}{h^2|\log h|}\I^{\theta_h}(v_h;B_R(0))=-\frac{E}{1-\nu^2}\frac{s^2}{8\pi}\,.
\end{equation}
\end{proposition}
\begin{proof}
We start by noticing that, for every $0<h<R$\,, the minimizer $v_h$ of $\I^{\theta_h}(\cdot;B_R(0))$ in $H^2_0(B_R(0))$ is unique by Lemma~\ref{propItheta}. 
Let $w_h\in H^2(B_R(0))$ be defined by the formula $w_h\coloneqq v_h- \bar v_h\res B_R(0)$\,,
where $\bar v_h$ is defined in \eqref{defvbarh}.
Then, by \eqref{bilaplvbarh}, we have that $w_h$ is the unique solution to
\begin{equation}\label{cauchyw}
\begin{cases}
\Delta^2 w=0&\text{in $B_R(0)$}\\
w=-\bar v_h&\text{on $\partial B_R(0)$}\\
\partial_n w=-\partial_n \bar v_h&\text{on $\partial B_R(0)$\,.}
\end{cases}
\end{equation}
By \cite[Theorem 2.16]{Gazzola09}, we have that there exists a constant $C=C(R)>0$ such that
\begin{equation}\label{estiesti}
\|w_h\|_{H^2(B_R(0))}\le C\|\bar v_h\|_{{H^{\frac 3 2}(\partial{B_R}(0))}}\le C\|\bar v_h\|_{H^2(A_{r,R}(0))}\,,
\end{equation}
where $0<r<R$ is fixed.

By \eqref{estiesti} and Lemma \ref{lemma:edgetrue} for $h$ small enough we get
\begin{equation}\label{estiwh}
\|w_h\|^2_{H^2(B_R(0))}\le C\|\bar v_h\|^2_{H^2(A_{r,R}(0))}\le {C(r,R)} s^2h^2\,,
\end{equation}
which, together with Lemma \ref{lemma:energyvbarh}, recalling the definition of $\langle \cdot;\cdot\rangle_{\G_{B_R(0)}}$ in \eqref{semiprod}, yields
\begin{equation*}
\begin{aligned}
\lim_{h\to 0}\frac{1}{h^2\log \frac{R}{h}}\G(v_h; B_R(0))=&\lim_{h\to 0}\frac{1}{h^2\log \frac{R}{h}}\G(\bar v_h;B_R(0))+\lim_{h\to 0}\frac{1}{h^2\log \frac{R}{h}}\G(w_h; B_R(0))\\
&\qquad{+}2s\lim_{h\to 0}\frac{1}{h^2\log \frac{R}{h}}\langle \bar v_h;w_h\rangle_{\G_{B_R(0)}}
=\frac{E}{1-\nu^2}\frac{s^2}{8\pi}\,,
\end{aligned}
\end{equation*}
i.e.,  \eqref{plim}. Finally, since
\begin{equation*}
\langle \theta_h,v_h\rangle=\langle \theta_h,\bar v_h\rangle+\langle \theta_h,w_h\rangle=-\frac{E}{1-\nu^2}\frac{s^2}{4\pi}h^2|\log h|+w_h(y^{h,+})-w_h(y^{h,-})\,,
\end{equation*}
{and using Morrey inequality  and \eqref{estiwh} to deduce that
$$
|w_h(y^{h,\pm})|\le C(R)\|w_h\|_{H^2(B_R(0))}\le C(r,R)sh\,,
$$
}
we get
\begin{equation*}
\lim_{h\to 0}\frac{1}{h^2\log\frac{R}{h}}\langle \theta_h,v_h\rangle=-\frac{E}{1-\nu^2}\frac{s^2}{4\pi}+\lim_{h\to 0} {\frac{sC(r,R)}{|\log h|}}=-\frac{E}{1-\nu^2}\frac{s^2}{4\pi}\,,
\end{equation*}
which, added to \eqref{plim}, yields \eqref{slim}.
\end{proof}
\subsection{Core-radius approach for a dipole of disclinations}\label{sec:hBV}
We discuss the convergence of a wedge disclination dipole to a planar edge dislocation.
We remind that the kinematic equivalence of a dipole of wedge disclinations with an edge dislocation has been first pointed out in \cite{Eshelby66} with a geometric construction in a  continuum   (see 
\cite{YL2010} for a construction on the hexagonal lattice).

Let $s>0$\,, $R>0$\,, $h\in(0,R)$, and let $\theta_h\coloneqq s\de_{(\frac h 2;0)}-s\de_{(-\frac h 2;0)}$\,.
Moreover, let $v_h\in H^2(B_R(0))$ satisfy 
\begin{equation}\label{risca}
\begin{cases}
\Delta^2 v_h=-\theta_h&\textrm{in }B_R(0)\\
v_h=\partial_nv_h=0&\textrm{on }\partial B_R(0)\,.
\end{cases}
\end{equation}
Then, since $\frac{\theta_h}h\to -s\partial_{x_1}\de_{0}$ as $h\to 0$\,, we expect that, formally, $\frac{v_h}{h}\to v$, where $v$ satisfies  
\begin{equation}\label{riscaaa}
\begin{cases}
\Delta^2 v= s\partial_{x_1}\de_{0}&\textrm{in }B_R(0)\\
v=\partial_nv=0&\textrm{on }\partial B_R(0)\,,
\end{cases}
\end{equation}
namely, $v$ is the Airy function associated with the elastic stress field of an edge dislocation centered at the origin and with Burgers vector $b=se_2$\,, see \eqref{incfinal2}.
Notice that the resulting Burgers vector is orthogonal to the direction of the disclination dipole $d$ (directed from the negative to the positive angle), more precisely we can write $\frac{b}{|b|}=\frac{d^\perp}{|d|}$ (see \cite{Eshelby66} and also \cite[formula (7.17)]{dW3} and \cite[formula (7)]{ZA2018}).

The convergence of the right-hand side of \eqref{risca} to the right-hand side of \eqref{riscaaa} represents   the kinematic equivalence between an edge dislocation and a wedge disclination dipole, obtained in the limit as the dipole distance~$h$ tends to zero. 
We now focus our attention on the investigation of  the energetic equivalence of these defects, which we pursue by analyzing rigorously the convergence of the solutions of \eqref{risca} to those of \eqref{riscaaa}.

As this analysis entails singular energies, we introduce regularized functionals parameterized by $0<\varepsilon<R$, representing the core radius. 
To this purpose, we define 
\begin{equation}\label{nuovoA1}
\begin{aligned}
\Bb_{\ce,R}\coloneqq\big\{w\in H^2_0(B_R(0)): \text{$w=a$ in $B_\ce(0)$ for some affine function~$a$}\big\}
\end{aligned}
\end{equation}
and, recalling \eqref{defI}, we introduce, for $h<\ce$\,, the functional 
\blu{$\I^{\theta_h}_{h,\ce}\colon \Bb_{\ce,R}\to\R$} defined by
\begin{equation}\label{defJheh}
\blu{\I^{\theta_h}_{h,\ce}(v)\coloneqq\G(v;B_R(0))+\frac{s}{2\pi(\ce-h)}\int_{\partial B_{\ce-h}(0)}\bigg[v\Big(x+\frac h 2e_1\Big)-v\Big(x-\frac h 2 e_1\Big)\bigg]\,\ud\Huno(x)}\,,
\end{equation}
associated with the pair $\theta_h$ of disclinations of opposite angles $\pm s$ 
placed at $\pm(\frac h 2,0)$, respectively.
We identify the relevant rescaling for the Airy stress function \blu{$v$}\,, parametrized by the dipole distance~$h$\,,
and corresponding to the energy regime of interest.
We stress that the energy scalings are dictated by the 
scaling of \blu{$v$} and not from \emph{a priori} assumptions.
Consequently, we assume $v=hw$ (with $w=\mathrm{O}(1)$) and  write
\begin{equation}\label{2206191431}
\blu{\I^{\theta_h}_{h,\ce}(hw)}=\G(hw;B_R(0))+\frac{s}{2\pi(\ce-h)}\int_{\partial B_{\ce-h}(0)}\bigg[hw\Big(x+\frac h 2e_1\Big)-hw\Big(x-\frac h 2 e_1\Big)\bigg]\,\ud\Huno(x)\,.
\end{equation}
%
%
%
%
%
%
%
It follows that the regularized energy of a disclination dipole of angles $\pm s$
is of order $\mathrm{O}(h^2)$\,. 
In order to isolate the first non-zero contribution 
in the limit as $h\to 0$, we divide \eqref{2206191431} by~$h^2$
and, \blu{in view of the homogeneity of degree $2$ of the elastic term, we get 
}
\begin{equation}\label{defJ}
\begin{aligned}
\blu{\I^{\theta_h/h}_{h,\ce}(w)}\equiv&\, \frac{1}{h^2}\I^{\theta_h/h}_{h,\ce}(hw)\\
=&\, \G(w;B_R(0))+\frac{s}{2\pi(\ce-h)}\int_{\partial B_{\ce-h}(0)}\frac{w(x+\frac h 2e_1)-w(x-\frac h 2 e_1)}{h}\,\ud\Huno(x)\,.
\end{aligned}
\end{equation}
\blu{Setting $\alpha\coloneqq se_2\de_0$\,, }
we show that the minimizers of \blu{$\I^{\theta_h/h}_{h,\ce}$} in $\Bb_{\ce,R}$ converge, as $h\to 0$\,, to the minimizers in  $\Bb_{\ce,R}$ of the functional \blu{$\I^{\alpha}_{0,\ce}\colon\Bb_{\ce,R}\to\R$} defined by
\begin{equation}\label{defJ0}
\blu{\I^{\alpha}_{0,\ce}}(w) \coloneqq\G(w;B_{R}(0))+\frac{s}{2\pi\ce}\int_{\partial B_\ce(0)}\partial_{x_1}w\,\ud\Huno\,.
\end{equation}
Notice that, by the very definition of $\Bb_{\ce,R}$ in \eqref{nuovoA1},
\begin{equation}\label{defJ0eff}
\blu{\I^{\alpha}_{0,\ce}}(w) = \G(w;A_{\ce,R}(0))+\frac{s}{2\pi\ce}\int_{\partial B_\ce(0)}\partial_{x_1}w\,\ud\Huno\,.
\end{equation}


We start by showing existence and uniqueness of the minimizers of \blu{$\I^{\theta_h/h}_{h,\ce}$ and $\I^\alpha_{0,\ce}$} in $\Bb_{\ce,R}$\,.
\begin{lemma}\label{existmin}
\blu{Let $s\in\R\setminus\{0\}$\,, $\theta_h\coloneqq s\de_{(\frac h 2;0)}-s\de_{(-\frac h 2;0)}$\,, and $\alpha\coloneqq se_2\de_0$\,.
For every $0< h <\ce<R$ there exists a unique minimizer of \blu{$\I^{\theta_h/h}_{h,\ce}$} in $\Bb_{\ce,R}$\,. Moreover, there exists a unique minimizer of  $\I^\alpha_{0,\ce}$ in $\Bb_{\ce,R}$\,.}
\end{lemma}
\begin{proof}
The proof relies on the direct method in the Calculus of Variations. We preliminarily notice that the uniqueness of the minimizers follows by the strict convexity (see \eqref{strictconv}) of {$\I^{\theta_h/h}_{h,\ce}$\,, for $h> 0$\,, and of $\I^{\alpha}_{0,\ce}$\,, for $h=0$}\,.
{For every $h> 0$ } let $\{W_{h,\ce,j}\}_{j\in\N}$ be a minimizing sequence for  $\blu{\I^{\theta_h/h}_{h,\ce}}$ in $\Bb_{\ce,R}$ \blu{and let $\{W_{0,\ce,j}\}_{j\in\N}$ be a minimizing sequence for  $\blu{\I^{\alpha}_{0,\ce}}$ in $\Bb_{\ce,R}$ }\,.
We first discuss the case $h>0$\,.
Since $W_{h,\ce,j}$ is affine in $B_\ce(0)$ for any $j\in\N$\,, for any $x\in\partial B_{\ce-h}(0)$ we have that  
\begin{equation}\label{stimasuder}
\begin{aligned}
\bigg|\frac{W_{h,\ce,j}\big(x+\frac{h}{2}e_1\big)-W_{h,\ce,j}\big(x-\frac{h}{2}e_1\big)}{h}\bigg|=&\,|\partial_{x_1}W_{h,\ce,j}(x)|\le \|\partial_{x_1}W_{h,\ce,j}\|_{L^\infty(B_\ce(0))}\\
\le&\, \frac1{\sqrt{\pi}\ce}\|W_{h,\ce,j}\|_{H^2(B_R(0))}\,.
\end{aligned}
\end{equation}
Hence, since the zero function $w= 0$ belongs to $\Bb_{\ce,R}$\,, by using Friedrich's inequality in $H^2_0(B_R(0))$\,, we get, for $j$ large enough,
\begin{equation}\label{proofi}
\begin{aligned}
0= \blu{\I^{\theta_h/h}_{h,\ce}}(0)\ge&\, \blu{\I^{\theta_h/h}_{h,\ce}}(W_{h,\ce,j}) \\
\ge&\, \frac{1}{2}\frac{1-\nu^2}{E}\min\{1-2\nu,1\}\|\nabla^2W_{h,\ce,j} \|^2_{L^2(B_R(0);\R^{2\times 2})}-\frac{s}{\sqrt{\pi}\ce}\|W_{h,\ce,j}\|_{H^2(B_R(0))}\\
\ge&\, C\|W_{h,\ce,j}\|^2_{H^2(B_R(0))}-\frac{s}{\sqrt{\pi}\ce}\|W_{h,\ce,j}\|_{H^2(B_R(0))}\,,
\end{aligned}
\end{equation}
for some constant $C>0$ depending only on $R$ (other than on $E$ and $\nu$)\,.
By \eqref{proofi}, we deduce that $\|W_{h,\ce,j}\|^2_{H^2(B_R(0))}$ is uniformly bounded. 
It follows that, up to a subsequence, $W_{h,\ce,j}\weakly W_{h,\ce}$ (as $j\to \infty$) in $H^2(B_R(0))$ for some function $W_{h,\ce}\in H^2_0(B_R(0))$ that is affine in $B_{\ce}(0)$\,. 
By the lower semicontinuity of $\blu{\I^{\theta_h/h}_{h,\ce}}$ with respect to the weak $H^2$-convergence,
we get that $W_{h,\ce}$ is a minimizer of $\blu{\I^{\theta_h/h}_{h,\ce}}$ in $\Bb_{\ce,R}$\,.
\blu{Analogously, since the last inequality in \eqref{stimasuder} holds true also for $h=0$\,, by arguing as above we have that, up to a subsequence,  $W_{0,\ce,j}\weakly W_{0,\ce}$ (as $j\to \infty$) in $H^2(B_R(0))$\,, where $W_{0,\ce}$ is the unique minimizer of $\I^{\alpha}_{0,\ce}$ in $\Bb_{\ce,R}$\,.}
\end{proof}
We are now in a position to prove the convergence of the minimizers and of the minimal values of $\blu{\I^{\theta_h/h}_{h,\ce}}$ to $\blu{\I^\alpha_{0,\ce}}$ as $h\to 0$\,.
\begin{proposition}\label{convhtozero}
Let $s\in\R\setminus\{0\}$\,.
Let $0<\ce<R$ and, for every $0<h<\ce$\,, let $\blu{W^{{\theta_h}/h}_{h,\ce}}$ be the minimizer of $\blu{\I^{\theta_h/h}_{h,\ce}}$ in $\Bb_{\ce,R}$\,. Then, as $h\to 0$\,, \blu{$W^{\theta_h/h}_{h,\ce}\to W^{\alpha}_{0,\ce}$} strongly in $H^2(B_R(0))$\,, where \blu{$W^\alpha_{0,\ce}$} is the minimizer of $\blu{\I^{\alpha}_{0,\ce}}$ in $\Bb_{\ce,R}$\,.
Moreover, $\blu{\I^{\theta_h/h}_{h,\ce}(W^{\theta_h/h}_{h,\ce})\to \I^{\alpha}_{0,\ce}(W^\alpha_{0,\ce})}$ as $h\to 0$\,.
\end{proposition}
\begin{proof}
For every $0<h<\ce$ let $a_{h,\ce}(x)\coloneqq c_{h,\ce,0}+c_{h,\ce,1}x_1+c_{h,\ce,2}x_2$ with $c_{h,\ce,0},c_{h,\ce,1}, c_{h,\ce,2}\in\R$ be such that \blu{$W^{\theta_h/h}_{h,\ce}=a_{h,\ce}$} in $B_{\ce}(0)$\,.
Then, arguing as in \eqref{proofi}, we get
\begin{equation*}
0\ge \blu{\I^{\theta_h/h}_{h,\ce}(W^{\theta_h/h}_{h,\ce})\ge C\|W^{\theta_h/h}_{h,\ce}\|^2_{H^2(B_R(0))}-\frac{s}{\sqrt{\pi}\ce}\|W^{\theta_h/h}_{h,\ce}\|_{H^2(B_R(0))}}\,.
\end{equation*}
Therefore, up to a (not relabeled) subsequence, \blu{$W^{\theta_h/h}_{h,\ce}\weakly \bar W^{\alpha}_{0,\ce}$ in $H^2(B_R(0))$ for some $\bar W^{\alpha}_{0,\ce}\in H^2_0(B_R(0))$}\,. 
Moreover, since the functions \blu{$W^{\theta_h/h}_{h,\ce}$} are affine in $\overline{B}_\ce(0)$\,, also \blu{$\bar W^\alpha_{0,\ce}$} is, and hence there exist $c_{0,\ce,0}, c_{0,\ce,1}, c_{0,\ce,2}\in\R$ such that $\bar W^\alpha_{0,\ce}(x)=c_{0,\ce,0}+ c_{0,\ce,1}x_1+ c_{0,\ce,2}x_2$ for every $x\in\overline{B}_\ce(0)$\,. It follows that $\bar W^\alpha_{0,\ce}\in\Bb_{\ce,R}$\,.

Now, since $W^{\theta_h/h}_{h,\ce}\to \bar W^\alpha_{0,\ce}$ in $H^1(B_R(0))$\,, we get that 
$c_{h,\ce,j}\to c_{0,\ce,j}$ as $h\to 0$\,, for every $j=1,2,3$\,,
which implies, in particular, that
\begin{equation}\label{termlin}
\begin{aligned}
&\, \lim_{h\to 0}\frac{1}{2\pi(\ce-h)}\int_{\partial B_{\ce-h}(0)}  \frac{W^{\theta_h/h}_{h,\ce}(x+\frac h 2 e_1)-W^{\theta_h/h}_{h,\ce}(x-\frac h 2 e_1)}{h}\,\ud\Huno(x)
=\lim_{h\to 0}c_{h,\ce,1} =c_{0,\ce,1}\\
= &\, \frac{1}{2\pi\ce}\int_{\partial B_\ce(0)}  \partial_{x_1}\bar W^s_{0,\ce}\,\ud\Huno\\
=&\, \lim_{h\to 0}\frac{1}{2\pi(\ce-h)}\int_{\partial B_{\ce-h}(0)} \!\!\!\! \frac{\bar W^\alpha_{0,\ce}(x+\frac h 2 e_1)-\bar W^\alpha_{0,\ce}(x-\frac h 2 e_1)}{h}\,\ud\Huno(x)\,.
 \end{aligned}
\end{equation}
Analogously,
\begin{equation}\label{termlin2}
\lim_{h\to 0}\frac{1}{2\pi(\ce-h)}\int_{\partial B_{\ce-h}(0)} \!\!\!\! \!\! \frac{W^\alpha_{0,\ce}(x+\frac h 2 e_1)-W^\alpha_{0,\ce}(x-\frac h 2 e_1)}{h}\,\ud\Huno(x)=\frac{1}{2\pi\ce}\int_{\partial B_\ce(0)} \!\! \partial_{x_1}W^\alpha_{0,\ce}\,\ud\Huno\,.
\end{equation}
By \eqref{termlin} and \eqref{termlin2}, using the lower semicontinuity of $\G$\,, and taking \blu[$W^\alpha_{0,\ce}$] as a competitor for $\blu{\I^{\theta_h/h}_{h,\ce}}$ in $\Bb_{\ce,R}$\,, we get
\begin{equation*}
\blu{\I^{\alpha}_{0,\ce}(W^{\alpha}_{0,\ce})\le\I^{\alpha}_{0,\ce}(\bar W^\alpha_{0,\ce})\le\liminf_{h\to 0}\I^{\theta_h/h}_{h,\ce}(W^{\theta_h/h}_{h,\ce})\le\lim_{h\to 0}\I^{\theta_h/h}_{h,\ce}(W^{\alpha}_{0,\ce})=\blu{\I^{\alpha}_{0,\ce}}(W^\alpha_{0,\ce})}\,,
\end{equation*}
so that all the inequalities above are in fact equalities. In particular, 
\begin{equation}\label{20220427}
\blu{\I^{\alpha}_{0,\ce}(W^\alpha_{0,\ce})=\lim_{h\to 0}\I^{\theta_h/h}_{h,\ce}(W^{\theta_h/h}_{h,\ce})}
\end{equation}
and consequently
 $\bar W^\alpha_{0,\ce}$ is a minimizer of $\blu{\I^{\alpha}_{0,\ce}}$ in $\Bb_{\ce,R}$\,. 
 In view of Lemma \ref{existmin}, we deduce that $\bar W^\alpha_{0,\ce}=W^\alpha_{0,\ce}$\,, which, together with \eqref{termlin} and \eqref{20220427}, implies that $\G(W^{\theta_h/h}_{h,\ce};B_R(0))\to \G(W^\alpha_{0,\ce};B_R(0))$ as $h\to 0$\,. In view of Remark \ref{equinorm}, this implies that $W^{\theta_h/h}_{h,\ce}\to W^\alpha_{0,\ce}$ strongly in $H^2(B_R(0))$ as $h\to 0$.
 Finally, by the Urysohn property, we get that the whole family $\{W^{\theta_h/h}_{h,\ce}\}_h$ converges to $W^\alpha_{0,\ce}$ as $h\to 0$\,. 
\end{proof}
We conclude this section by determining the minimizer \blu{$W^\alpha_{0,\ce}$} of $\blu{\I^{\alpha}_{0,\ce}}$ in $\Bb_{\ce,R}$, for $\alpha=se_2\de_0$\,.
\begin{lemma}\label{lm:mindis}
Let $s\in\R\setminus\{0\}$\,. For every $0<\ce<R$ the function \blu{$W^{se_2\de_0}_{0,\ce}:B_R(0)\to\R$} defined by
\begin{equation}\label{2201181926}
W_{0,\ce}^{se_2\de_0}(x)\coloneqq \begin{cases}
\displaystyle \frac{s}{16\pi}\frac{E}{1-\nu^2}\Big(\alpha_\ce+\beta_\ce\frac{1}{|x|^2}+\gamma_\ce|x|^2+2\log|x|^2\Big)x_1&\text{if $x\in A_{\ce,R}(0)$}\\[2mm]
\displaystyle \frac{s}{16\pi}\frac{E}{1-\nu^2}\Big(\alpha_\ce+\frac{\beta_\ce}{\ce^2}+\ce^2\gamma_\ce+4\log\ce\Big)x_1&\text{if $x\in B_\ce(0)$\,,}
\end{cases}
\end{equation}
with 
\begin{equation}\label{abc}
\alpha_\ce\coloneqq2\frac{R^2-\ce^2}{R^2+\ce^2}-2\log R^2\,,\quad \beta_\ce\coloneqq 2\ce^2\frac{R^2}{R^2+\ce^2}\,,\quad \gamma_\ce\coloneqq-\frac{2}{R^2+\ce^2}\,,
\end{equation}
is the unique minimizer in $\Bb_{\ce,R}$ of the functional $\blu{\I^{se_2\de_0}_{0,\ce}}$ defined in \eqref{defJ0}.  
Moreover,
\begin{equation}\label{valmin}
\blu{\I^{se_2\de_0}_{0,\ce}(W^{se_2\de_0}_{0,\ce})}=-\frac{s^2}{8\pi}\frac{E}{1-\nu^2}\Big(
\log \frac{R}{\ce}-\frac{R^2-\ce^2}{R^2+\ce^2}\Big)\,.
\end{equation}
\end{lemma}
The proof of Lemma \ref{lm:mindis} is postponed to Appendix \ref{prooflm:mindis}, where we also state Corollary~\ref{insec4}, which will be used in Section~\ref{sc:four}.

\begin{remark}\label{2202211829}
\rm{
Let $b\in\R^2\setminus\{0\}$ and let $\Pi(b)$ denote the $\frac\pi 2$ clockwise rotation of the vector $b$\,, \emph{i.e.}, 
\begin{equation}\label{pigrecobugualemenobortogonale}
\Pi(b)=-b^\perp\,.
\end{equation}
For any $0<h<\ce<R$\,, \blu{
for $\theta_h\coloneqq |b|\de_{\frac h 2\frac{\Pi(b)}{|b|}}-|b|\de_{-\frac h 2 \frac{\Pi(b)}{|b|}}$\,, 
in analogy with \eqref{defJheh} and \eqref{defJ}, we can define the functional $\I_{h,\ce}^{\theta_h/h}\colon \Bnew_{\ce,R}\to \R$ (for this choice of $\theta_h$)} as
\begin{equation*}
\I_{h,\ce}^{\theta_h/h}(w)\coloneqq\G(w;B_R(0))+\frac{|b|}{2\pi(\ce-h)}\int_{\partial B_{\ce-h}(0)}\frac{w\big(x+\frac h 2\frac{\Pi(b)}{|b|}\big)-w\big(x-\frac h 2 \frac{\Pi(b)}{|b|}\big)}{h}\,\ud\Huno(x)\,.
\end{equation*}
\blu{Notice that for $b$ directed along the positive $x$-axis, the previous formula \eqref{defJ} agree.}
By arguing verbatim as in the proof of Proposition~\ref{convhtozero}, we have that, as $h\to 0$\,, the unique minimizer of $\I_{h,\ce}^{\theta_h/h}$ in $\Bnew_{\ce,R}$ converges strongly in $H^2(B_R(0))$ to the unique minimizer in $\Bnew_{\ce,R}$ of the functional $\I_{0,\ce}^{b\de_0}$ defined by
\begin{equation}\label{dadef}
\I_{0,\ce}^{b\de_0}(w)\coloneqq\G(w;B_R(0))+\frac{1}{2\pi\ce}\int_{\partial B_{\ce}(0)}
\langle  \nabla w, \Pi(b)\rangle\,\ud\Huno\,.
\end{equation}
Notice that the minimizer of $\I_{0,\ce}^{b\de_0}$ is given by 
\begin{equation}\label{vudoppiob}
W_{0,\ce}^{b\de_0}(x)\coloneqq \blu{W_{0,\ce}^{|b|e_2\de_0}}\bigg(\Big\langle\frac{\Pi(b)}{|b|},x\Big\rangle,\Big\langle \frac{b}{|b|},x\Big\rangle\bigg)\,,
\end{equation}
where the function $W_{0,\ce}^{se_2\de_0}$ is defined in Lemma \ref{lm:mindis}.

Furthermore, one can easily check that the same proof of Proposition~\ref{convhtozero} applies also to general domains $\Omega$ as well as to a general distribution of dipoles of wedge disclinations 
\begin{equation}\label{thetahJ}
\theta_h\coloneqq\sum_{j=1}^J|b^j|\Big(\de_{x^j+\frac{h}{2}\frac{\Pi(b^j)}{|b^j|}}-\de_{x^j-\frac{h}{2}\frac{\Pi(b^j)}{|b^j|}}\Big)\in\WD(\Omega)\,,
\end{equation}
(with $b^j\in\R^2\setminus\{0\}$ and $\min_{\newatop{j_1,j_2=1,\ldots,J}{j_1\neq j_2}}|x^{j_1}-x^{j_2}|,\min_{j=1,\ldots,J}\di(x^j,\partial\Omega)>2\ce$) approximating the family of edge dislocations $\alpha\coloneqq\sum_{i=1}^Jb^j\de_{x^j}\in\ED(\Omega)$\,.
In such a case, {by arguing as in the proof of Proposition \ref{convhtozero}}, one can show that, as $h\to 0$\,, the unique minimizer $w_{h,\ce}^{\theta_h/h}$ of the functional
\begin{equation}\label{2202231647}
\mathcal{I}^{\theta_h/h}_{h,\ce}(w)\coloneqq \G(w;\Omega)+\sum_{j=1}^J\frac{|b^j|}{2\pi(\ce-h)}\int_{\partial B_{\ce-h}(x^j)}
\frac{w(x+\frac h 2\frac{\Pi(b^j)}{|b^j|})-w(x-\frac h 2 \frac{\Pi(b^j)}{|b^j|})}{h}\,\ud\Huno(x)
\end{equation}
in the set
\begin{equation}\label{2202211820}
\begin{aligned}
\Bb^{\alpha}_{\ce,\Omega}\coloneqq\{w\in H^2_0(\Omega): \text{$w=a^j$ in $B_\ce(x^j)$ for some affine functions $a^j$\,, $j=1,\dots, J$}\}\,,
\end{aligned}
\end{equation}
converges strongly in $H^2(\Omega)$ to the unique minimizer $w_{0,\ce}^\alpha$ in $\Bb^{\alpha}_{\ce,\Omega}$ of the functional
\begin{equation}\label{energyI}
\begin{aligned}
\mathcal{I}^\alpha_{0,\ce}(w)\coloneqq&\, \mathcal{G}(w;{\Omega})+\sum_{j=1}^J\frac{1}{2\pi\ce}\int_{\partial B_{\ce}(x^j)}
\langle  \nabla w, \Pi(b^j)\rangle\,\ud\Huno\\
=&\,\mathcal{G}(w;\Omega_\ce(\alpha))+\sum_{j=1}^J\frac{1}{2\pi\ce}\int_{\partial B_{\ce}(x^j)}
\langle  \nabla w, \Pi(b^j)\rangle\,\ud\Huno\,,
\end{aligned}
\end{equation}
where $\Omega_\ce(\alpha)\coloneqq\Omega\setminus\bigcup_{j=1}^{J} \overline{B}_\ce(x^j)$\,.
}
\end{remark}

\section{Limits for dislocations}\label{sc:four}
In this section, we obtain the full asymptotic expansion in $\ce$ of the singular limit functional
$\mathcal{I}^\alpha_{0,\ce}$ introduced in \eqref{energyI}. 
We first prove the convergence of the minimizers of $\mathcal{I}^\alpha_{0,\ce}$ in a suitable functional setting (see Theorem~\ref{2201181928}) and then, by showing that all terms of the expansion coincide with the corresponding terms of the renormalized energy of edge dislocations of \cite{CermelliLeoni06}, 
we finally deduce the asymptotic energetic equivalence of systems of disclination dipoles with the corresponding systems of edge dislocations.

 Let $\alpha=\sum_{j=1}^{J}b^j\de_{x^j}\in\ED(\Omega)$\,. 
 We consider the following minimum problem
\begin{equation}\label{2112192000}
\min_{w\in\Bb^{\alpha}_{\ce,\Omega}}\mathcal{I}_\ce^\alpha(w)\,,
\end{equation}
where $\mathcal{I}_\ce^\alpha(w)\coloneqq \mathcal{I}^\alpha_{0,\ce}(w)$ is  the functional defined in \eqref{energyI} and $\Bb^{\alpha}_{\ce,\Omega}$ is defined in \eqref{2202211820}.
In order to study the asymptotic behavior of the minimizers and minima of $\mathcal{I}_\ce^\alpha$ as $\ce\to 0$\,, we first introduce some notation.

 Fix  $R>0$ such that $\overline{\Omega}\subset B_R(x^j)$ for every $j=1,\ldots,J$\,, and let $\ce>0$ be such that the (closed) balls $\overline{B}_\ce(x^j)$ are pairwise disjoint and contained in $\Omega$\,, i.e.,  
 \begin{equation}\label{distaminima}
 \ce< D\coloneqq\min_{j=1,\ldots,J}\bigg\{\frac12\dist_{i\neq j}(x^i,x^j)\,, \dist(x^{j},\partial\Omega)\bigg\}\,.
 \end{equation}
 We define the function $W_{\ce}^\alpha\colon\Omega_{\ce}(\alpha)\to\R$ by
 \begin{equation}\label{20220223_1}
W_{\ce}^\alpha(x)\coloneqq\sum_{j=1}^J W_{\ce}^{j}(x)\,,\qquad\text{with}\qquad W_{\ce}^{j}(\cdot)\coloneqq W_{0,\ce}^{b^j\blu{\de_0}}(\cdot-x^j)
  \end{equation}
 (see \eqref{vudoppiob})\,.
 We highlight that the function $W^\alpha_\ce$ depends also on $R$ through the constants defined in \eqref{abc}.
Notice that any function $w\in \Bb^{\alpha}_{\ce,\Omega}$ can be decomposed as 
\begin{equation}\label{2201181904}
w=W^{\alpha}_\ce+\widetilde w
\end{equation}
 where $\widetilde w\in \widetilde{\Bb}^{\alpha}_{\ce,\Omega}$\,, with 
\begin{equation}\label{tildeBdelta}
\begin{aligned}
\widetilde{\Bb}^{\alpha}_{\ce,\Omega}\coloneqq &\, \{\widetilde{w}\in H^2_0(\Omega)-W^{\alpha}_\ce: \text{$\widetilde{w}+W^\alpha_{\ce}=a^j$ in $B_\ce(x^j)$} \\
&\, \phantom{\{\widetilde{w}\in H^2_0(\Omega)-W^{\alpha}_\ce:\,} \text{for some affine functions $a^j$\,, $j=1,\dots, J$}\}\\
\equiv&\, {\Bb}^{\alpha}_{\ce,\Omega}-W^\alpha_\ce\,.
\end{aligned}
\end{equation}

Therefore, in view of the decomposition \eqref{2201181904}, for every $w\in\Bb_{\ce,\Omega}^\alpha$ we have
\begin{equation}\label{20220223_2}
\I_{\ce}^\alpha(w)=\G(W_{\ce}^\alpha;\Omega_\ce(\alpha))+\sum_{j=1}^J\frac{1}{2\pi\ce}\int_{\partial B_\ce(x^j)}\langle\nabla W_\ce^\alpha,\Pi(b^j)\rangle\,\ud\Huno+\widetilde{\I}_{\ce}^\alpha(\widetilde{w})\,,
\end{equation}
where 
\begin{equation}\label{defItilde}
\begin{split}
\widetilde{\mathcal I}^{\alpha}_\ce(\widetilde w)\coloneqq
\mathcal{G}(\widetilde{w};\Omega_{\ce}(\alpha))
+&\,\frac{1+\nu}{E} \sum_{j=1}^J  \int_{\Omega_{\ce}(\alpha)}   
\Big(\nabla^2W^j_{\ce}:\nabla^2\widetilde{w}-\nu  \Delta W^j_{\ce}\Delta \widetilde{w}\Big)\,\mathrm{d} x\\
+&\,\sum_{j=1}^J\frac{1}{2\pi\ce}\int_{\partial B_{\ce}(x^j)}
\langle  \nabla \widetilde{w}, \Pi(b^j)\rangle\,\ud\Huno\,.
\end{split}
\end{equation}
Notice that the integration for the bulk term $\mathcal{G}$ above is performed on $\Omega_{\ce}(\alpha)$ and not on~$\Omega$\,, as the function $\widetilde{w}$ is not, in general, affine in $\bigcup_{j=1}^J B_{\ce}(x^j)$\,.

In view of \eqref{20220223_2}, as in \cite[Theorem 4.1]{CermelliLeoni06}, the minimum problem \eqref{2112192000} (for $w$) is equivalent to the following minimum problem (for $\widetilde w$)
\begin{equation}\label{2112192003}
\begin{aligned}
\min_{\widetilde w\in \widetilde{\Bb}^{\alpha}_{\ce,\Omega}}\widetilde{\mathcal I}^{\alpha}_\ce(\widetilde{w})\,.
\end{aligned}
\end{equation}

\begin{lemma}\label{20220222_1}
For every $\widetilde w\in\widetilde{\Bb}^{\alpha}_{\ce,\Omega}$ we have
\begin{equation}\label{20220222_2}
\begin{aligned}
\widetilde{\mathcal I}^{\alpha}_\ce(\widetilde w)=&\,
\mathcal{G}(\widetilde{w};\Omega_{\ce}(\alpha))
+\frac{1+\nu}{E}\sum_{j=1}^J\bigg(-(1-\nu)\int_{\partial\Omega}(\partial_{n}\Delta W_{\ce}^j)\widetilde{w} \,\ud\Huno\\
&\,+\int_{\partial\Omega}\langle \nabla^2 W_{\ce}^j   n,\nabla \widetilde{w} \rangle \,\ud\Huno- \nu\int_{\partial\Omega}\Delta W_{\ce}^j \partial_{n}\widetilde{w} \,\ud\Huno\bigg)\\
&\,+\sum_{j=1}^{J}\Bigg(\frac{1+\nu}{E}\sum_{i=1}^J\bigg((1-\nu)\int_{\partial B_\ce(x^i)}(\partial_{n}\Delta W_{\ce}^j)\widetilde{w} \,\ud\Huno
- \int_{\partial B_{\ce}(x^i)}\langle  \nabla^2 W_{\ce}^j n,\nabla \widetilde{w} \rangle \,\ud\Huno\\
&\,\phantom{- \sum_{i=1}^J}\quad+\nu\int_{\partial B_{\ce}(x^i)}\Delta W_{\ce}^j \partial_{n} \widetilde{w} \,\ud\Huno\bigg)+\frac{1}{2\pi\ce}\int_{\partial B_{\ce}(x^j)}
\langle  \nabla \widetilde{w}, \Pi(b^j)\rangle\,\ud\Huno\Bigg)\,.
\end{aligned}
\end{equation}
\end{lemma}
\begin{proof}
Let $\widetilde w\in \widetilde{\Bb}^{\alpha}_{\ce,\Omega}$ be fixed.
By the Gauss--Green Theorem, for every $j=1,\dots, J$ and for every $0<\ce<D$\,, we have
\begin{equation}\label{22011101622}
\begin{aligned}
 \int_{\Omega_{\ce}(\alpha)}
 \nabla^2W_{\ce}^j:\nabla^2\widetilde{w}\, \ud x=&\,
-\int_{\partial\Omega}(\partial_{n}\Delta W_{\ce}^j)\widetilde{w} \,\ud\Huno
 + \sum_{i=1}^J\int_{\partial B_{\ce}(x^i)}(\partial_{n}\Delta W_{\ce}^j)\widetilde{w} \,\ud\Huno\\
&+ \int_{\partial\Omega}\langle \nabla^2 W_{\ce}^j   n,\nabla \widetilde{w} \rangle \,\ud\Huno
- \sum_{i=1}^J\int_{\partial B_{\ce}(x^i)}\langle  \nabla^2 W_{\ce}^j n,\nabla \widetilde{w} \rangle \,\ud\Huno\,,
\end{aligned}
\end{equation}
and
\begin{equation}\label{22011101624}
\begin{aligned}
 \int_{\Omega_{\ce}(\alpha)}
 \Delta W_{\ce}^j\Delta \widetilde{w}\, \ud x=&\,
- \int_{\partial\Omega}(\partial_{n}\Delta W_{\ce}^j) \widetilde{w} \,\ud\Huno
+ \sum_{i=1}^J\int_{\partial B_{\ce}(x^i) }(\partial_{n}\Delta W_{\ce}^j) \widetilde{w}\,\ud\Huno\\
&+ \int_{\partial\Omega}\Delta W_{\ce}^j \partial_{n}\widetilde{w} \,\ud\Huno
- \sum_{i=1}^J\int_{\partial B_{\ce}(x^i)}\Delta W_{\ce}^j \partial_{n} \widetilde{w} \,\ud\Huno\,,
\end{aligned}
\end{equation}
where we have used that 
$ \Delta^2W_{\ce}^j\equiv 0$ in $\Omega_{\ce}(\alpha)$ for every $j=1,\ldots,J$\,.
By \eqref{22011101622} and \eqref{22011101624} it follows that
\begin{equation*}
\begin{aligned}
&\,\int_{\Omega_{\ce}(\alpha)}   
\Big( \nabla^2W^j_{\ce}:\nabla^2\widetilde{w}-\nu  \Delta W^j_{\ce}\Delta \widetilde{w}\Big)\,\ud  x\\
=&\,-(1-\nu)\int_{\partial\Omega}(\partial_{n}\Delta W_{\ce}^j)\widetilde{w} \,\ud\Huno+\int_{\partial\Omega}\langle \nabla^2 W_{\ce}^j   n,\nabla \widetilde{w} \rangle \,\ud\Huno-\nu \int_{\partial\Omega}\Delta W_{\ce}^j \partial_{n}\widetilde{w} \,\ud\Huno\\
&\,+(1-\nu)\sum_{i=1}^J\int_{\partial B_\ce(x^i)}(\partial_{n}\Delta W_{\ce}^j)\widetilde{w} \,\ud\Huno\\
&\,- \sum_{i=1}^J\int_{\partial B_{\ce}(x^i)}\langle  \nabla^2 W_{\ce}^j n,\nabla \widetilde{w} \rangle \,\ud\Huno+\nu\sum_{i=1}^J\int_{\partial B_{\ce}(x^i)}\Delta W_{\ce}^j \partial_{n} \widetilde{w} \,\ud\Huno\,,
\end{aligned}
\end{equation*}
which, in view of the very definition of $\widetilde{\mathcal I}^\alpha_{\ce}$ in \eqref{defItilde}, implies {\eqref{20220222_2}}.
\end{proof}
\begin{remark}\label{maybeuseful}
\rm{Let $\alpha=\sum_{j=1}^J b^j\de_{x^j}\in\ED(\Omega)$\,. For every $0<r<R$ and for every $j=1,\ldots,J$ we have that the plastic functions $W^j_{\ce}$ converge in $C^\infty(A_{r,R}(x^j))$, as $\ce\to0$\,, to the function~$W^j_0$ defined by
\begin{equation}\label{20220311}
\begin{aligned}
W^j_0(x)\coloneqq&\, \frac{|b^j|}{8\pi}\frac{E}{1-\nu^2}\Big((1-\log R^2)-\frac{|x|^2}{R^2}+\log|x|^2\Big)\Big\langle\frac{\Pi(b^j)}{|b^j|},x-x^j\Big\rangle\,.
\end{aligned}
\end{equation}
It follows that  $W_\ce^\alpha\to \sum_{j=1}^JW^j_0\eqqcolon W_0^\alpha$ in $C^\infty(\Omega_{r}(\alpha))$ and hence in $H^2_{\loc}\big(\Omega\setminus\bigcup_{j=1}^J\{x^j\}\big)$\,.
Therefore, in the spirit of \eqref{tildeBdelta} 
we set
\begin{equation}\label{tildeBzero}
\widetilde{\Bb}^{\alpha}_{0,\Omega}\coloneqq\{w\in H^2(\Omega)\,:\,w=-{W}^{\alpha}_0\,,\, \partial_n w=-\partial_n{W}^{\alpha}_0\textrm{ on }\partial\Omega\}\,.
\end{equation}
}
\end{remark}
Now we prove the following theorem, which is the equivalent of \cite[Theorem 4.1]{CermelliLeoni06} in terms of the Airy stress function.
\begin{theorem}\label{2201181928}
Let $\alpha=\sum_{j=1}^Jb^j\de_{x^j}\in\ED(\Omega)$ and let
$\mathcal{I}_\ce^\alpha$ be the functional in \eqref{20220223_2} 
for every $\ce>0$\,.
For $\ce>0$ small enough, 
the minimum problem \eqref{2112192000} admits a unique solution $w_\ce^\alpha$\,. 
Moreover, $w_\ce^\alpha\to w^\alpha_0$, as $\ce\to 0$\,, strongly 
 in $H^2_{\loc}(\Omega\setminus\bigcup_{j=1}^J\{x^j\})$\,, where $w^\alpha_0\in H^2_{\loc}(\Omega\setminus\bigcup_{j=1}^J\{x^j\})$ is the unique distributional solution to
%
%
%
\begin{equation}\label{limsol}
\begin{cases}
\displaystyle \frac{1-\nu^2}{E}\Delta^2 w=-\sum_{j=1}^{J}|b^j|\partial_{\frac{(b^j)^\perp}{|b^j|}}\de_{x^j} &\text{in $\Omega$}\\[2mm]
w=\partial_n w=0&\text{on $\partial\Omega$\,.}
\end{cases}
\end{equation}
\end{theorem}
Theorem \ref{2201181928} is a consequence of Propositions~\ref{2201181930} and~\ref{2201252355} below, which are the analogue of \cite[Lemma 4.2]{CermelliLeoni06} and \cite[Lemma 4.3]{CermelliLeoni06}, respectively.
%
\begin{proposition}\label{2201181930}
Let $\alpha\in\ED(\Omega)$ and let $\ce>0$ be small enough. For every $\widetilde w\in \widetilde{\Bb}^{\alpha}_{\ce,\Omega}$ we have
\begin{equation}\label{2201182011}
C_1\big(\|\widetilde w\|_{H^2(\Omega_{\ce}(\alpha))}^2-\|\widetilde w\|_{H^2(\Omega_{\ce}(\alpha))}-1\big) \le \widetilde{\mathcal{I}}^\alpha_{\ce}(\widetilde w)\le
C_2\big(\|\widetilde w\|_{H^2(\Omega_{\ce}(\alpha))}^2+\|\widetilde w\|_{H^2(\Omega_{\ce}(\alpha))}+1\big)\,,
\end{equation}
for some constants $0<C_1<C_2$ independent of $\ce$\,. 
Moreover, 
problem \eqref{2112192003} admits a unique solution $\widetilde{w}^\alpha_{\ce}\in \widetilde{\Bb}^{\alpha}_{\ce,\Omega}$ and 
 $\|\widetilde{w}_{\ce}^\alpha\|_{H^2(\Omega_{\ce}(\alpha))}$ is uniformly bounded with respect to $\ce$\,.
 Furthermore, there exists $\widetilde{w}^\alpha_0\in \widetilde{\Bb}^{\alpha}_{0,\Omega}$ such that as $\ce\to0$ and up to a (not relabeled) subsequence,
 \begin{equation}\label{2201141838} 
 \widetilde{w}_\ce^{\alpha}
 \weakly \widetilde{w}_0^\alpha\quad\text{weakly in $H^2(\Omega)$.}
\end{equation}
\end{proposition}
\begin{proposition}
\label{2201252355}
Let $\alpha=\sum_{j=1}^Jb^j\de_{x^j}\in\ED(\Omega)$ and let $\ce>0$ be small enough.
Let $\widetilde w^\alpha_{\ce}$ and $\widetilde w^\alpha_{0}$ be as in Proposition~\ref{2201181930}\,.
Then, as $\ce\to 0$\,, the whole sequence $\widetilde{w}^\alpha_\ce$ converges to $\widetilde{w}^\alpha_0$\,,  strongly in $H^2_\loc\big(\Omega\setminus\bigcup_{j=1}^J\{x^j\}\big)$ and
 $\widetilde w^\alpha_{0}$ is
the unique minimizer in $ \widetilde{\Bb}^{\alpha}_{0,\Omega}$ of the functional $\widetilde{\mathcal{I}}^\alpha_{0}$ defined by
\begin{equation*}
\begin{aligned}
\widetilde{\mathcal{I}}^\alpha_{0}(\widetilde{w})\coloneqq&\,\mathcal{G}(\widetilde{w};\Omega)
+\frac{1+\nu}{E} \sum_{j=1}^J \Big( -(1-\nu)
 \int_{\partial\Omega}(\partial_{n}\Delta  W_{0}^j)\widetilde{w} \,\ud \Huno\\
&\phantom{\mathcal{G}(\widetilde{w};\Omega)
+\frac{1+\nu}{E} \sum_{j=1}^J}+ \int_{\partial\Omega}\langle \nabla^2  W_{0}^j   n,\nabla \widetilde{w} \rangle \,\ud \Huno
-\nu \int_{\partial\Omega}\Delta  W_{0}^j \partial_{n}\widetilde{w} \,\ud\Huno \Big)\,.
\end{aligned}
\end{equation*}
Moreover,
\begin{equation}\label{20220419}
\Delta^2\widetilde{w}^\alpha_0=0 \qquad\textrm{in }\Omega
\end{equation}
and
\begin{equation}\label{20220222_7}
\widetilde{\mathcal{I}}^\alpha_{\ce}(\widetilde w^\alpha_{\ce})\to \widetilde{\mathcal{I}}^\alpha_{0}(\widetilde w^\alpha_{0})\qquad\textrm{as }\ce\to 0\,.
\end{equation}
\end{proposition}

\begin{proof}[Proof of Theorem \ref{2201181928}]
By the additive decomposition in \eqref{2201181904} and by Propositions \ref{2201181930},  we have that, for $\ce>0$ small enough, $w^\alpha_\ce=W^\alpha_\ce+\widetilde w^\ce_\alpha$\,, where $W^\alpha_\ce$ is defined in \eqref{20220223_1} and $\widetilde{w}_\ce^\alpha$ is the unique solution to the minimum problem in \eqref{2112192003}.
Therefore, by Remark \ref{maybeuseful} and by Proposition \ref{2201252355}, we have that $w_\ce^\alpha\to W_0^\alpha+\widetilde{w}_0^\alpha\eqqcolon w^\alpha_0$ in $H^2_\loc\big(\Omega\setminus\bigcup_{j=1}^J\{x^j\}\big)$ as $\ce\to 0$\,.
Notice that, by \eqref{20220419} and by the very definition of $w_0^\alpha$ (see \eqref{20220311}),
\begin{equation}\label{20220311_1}
\frac{1-\nu^2}{E}\Delta^2w_0^\alpha=\frac{1-\nu^2}{E}\Delta^2W_0^\alpha=-\sum_{j=1}^{J}|b^j|\partial_{\frac{(b^j)^\perp}{|b^j|}}\de_{x^j}\qquad\textrm{in }\Omega\,,
\end{equation}
\emph{i.e.}, the first equation in \eqref{limsol}. 
Finally, the boundary conditions are satisfied since $\widetilde{w}^\alpha_0\in\widetilde{\Bb}_{0,\Omega}^\alpha$ (see \eqref{tildeBzero}).
\end{proof}
Now we prove Proposition \ref{2201181930}.
\begin{proof}[Proof of Proposition \ref{2201181930}]
Let $\alpha=\sum_{j=1}^Jb^j\de_{x^j}\in\ED(\Omega)$ and let $\widetilde{w}\in\widetilde{\Bb}_{\ce,\Omega}^\alpha$\,.
We first prove that for every $j=1,\ldots, J$
\begin{equation}\label{20220222_3}
\begin{aligned}
\frac{1+\nu}{E}\sum_{i=1}^J\bigg((1-\nu)\int_{\partial B_\ce(x^i)}(\partial_{n}\Delta W_{\ce}^j)\widetilde{w} \,\ud\Huno
- \int_{\partial B_{\ce}(x^i)}\langle  \nabla^2 W_{\ce}^j n,\nabla \widetilde{w} \rangle \,\ud\Huno&\\
+\nu\int_{\partial B_{\ce}(x^i)}\Delta W_{\ce}^j \partial_{n} \widetilde{w} \,\ud\Huno\bigg)+\frac{1}{2\pi\ce}\int_{\partial B_{\ce}(x^j)}
\langle  \nabla \widetilde{w}, \Pi(b^j)\rangle\,\ud\Huno&=\mathrm{O}(\ce)\,.
\end{aligned}
\end{equation}
To this purpose, we recall that, for every $i=1,\ldots,J$, there exists an affine function $a^i_\ce$ such that
\begin{equation}\label{2201101612}
\widetilde w=a^i_\ce-W^{i}_{\ce}-\sum_{k\neq i}W^k_\ce\qquad\textrm{on }\partial B_\ce(x^i)\,.
\end{equation}
{Notice that $W_\ep^j$ minimizes the energy $\I_{0,\ce}^{b\de_{x^j}}$  referred to the ball $B_R(x^j)$\,: this follows by a simple translation argument keeping \eqref{nuovoA1}, \eqref{dadef}\,, and \eqref{20220223_1} into account. By the characterization of the minimality provided in \eqref{natural}, }
for every function $a$ which is affine in $B_\ce(x^j)$ we have
\begin{equation}\label{20220216_1}
 \begin{aligned}
&  \frac{1-\nu^2}{E}\int_{\partial B_\ce(x^j)}(\partial_{n}\Delta W_{\ce}^j)a \,\ud\Huno+
  \frac{1+\nu}{E}\nu\int_{\partial B_{\ce}(x^j)}  \Delta W_{\ce}^j \partial_na \,\ud\Huno\\
  &-\frac{1+\nu}{E}\int_{\partial B_{\ce}(x^j)}\langle  \nabla^2 W_{\ce}^j n,\nabla a \rangle \,\ud\Huno\\
  =&-\frac{|b^j|}{2\pi\ce}\int_{\partial B_\ce(x^j)}\partial_{\frac{(b^j)^\perp}{|b^j|}}a\,\ud\Huno
  =-\frac{1}{2\pi\ce}\int_{\partial B_\ce(x^j)}\langle\nabla a,\Pi(b^j)\rangle\,\ud\Huno\,.
 \end{aligned}
 \end{equation}
Let $j=1,\ldots,J$ be fixed. 
We first focus on the case $i=j$ in \eqref{20220222_3}\,. 
Recalling that $W_\ce^j$ is affine in $B_\ce(x^j)$\,, {by choosing $a=a_\ep^j-W_\ep^j$ in \eqref{20220216_1},} 
we get
\begin{equation}\label{2201182343}
 \begin{aligned}
 \frac{1-\nu^2}{E}\int_{\partial B_\ce(x^j)}(\partial_{n}\Delta W_{\ce}^j)(a^j_\ce-W^j_\ce) \,\ud\Huno+
  \frac{1+\nu}{E}\nu\int_{\partial B_{\ce}(x^j)}  \Delta W_{\ce}^j \partial_n(a^j_\ce-W^j_\ce) \,\ud\Huno&\\
  -\frac{1+\nu}{E}\int_{\partial B_{\ce}(x^j)}\langle  \nabla^2 W_{\ce}^j n,\nabla(a^j_\ce-W^j_\ce) \rangle \,\ud\Huno
 +\frac{1}{2\pi\ce}\int_{\partial B_\ce(x^j)}\langle\nabla(a^j_\ce-W_\ce^j),\Pi(b^j)\rangle\,\ud\Huno&=0\,.
 \end{aligned}
 \end{equation}
 Furthermore, recalling that {$W_\ce^k$} is smooth in $B_\ce(x^j)$ for every $k\neq j$, by Taylor expansion
 we have that for every $x\in B_\ce(x^j)$
 \begin{equation*}
W^k_\ce(x)=W^k_\ce(x^j)+\langle\nabla W^k_\ce(x^j),x-x^j\rangle+\mathrm{O}(\ce^2) {\quad\textrm{and}\quad
\nabla W^k_\ce(x)=\nabla W^k_\ce(x^j)+\mathrm{O}(\ep)}\,,
\end{equation*}
 whence, using \eqref{20220216_1} with $a(\cdot):=W^k_\ce(x^j)+\langle\nabla W^k_\ce(x^j),\cdot-x^j\rangle$\,,  {recalling that $|\partial_n\nabla W_\ep^j|\sim|x-x^j|^{-1}$
and $|\nabla^2W_\ep^j(x)|\sim |x-x^j|^{-1}$\,, summing over $k\neq j$\,,} we deduce that
  \begin{equation}\label{dopo}
 \begin{aligned}
&\frac{1-\nu^2}{E}\int_{\partial B_\ce(x^j)} \!\!\!\! (\partial_{n}\Delta W_{\ce}^j)\Big(-\sum_{k\neq j}W^k_\ce\Big) \,\ud\Huno+
  \frac{1+\nu}{E}\nu\int_{\partial B_{\ce}(x^j)} \!\!\!\! \Delta W_{\ce}^j \partial_n\Big(-\sum_{k\neq j}W^k_\ce\Big) \,\ud\Huno\\
&-\frac{1+\nu}{E}\int_{\partial B_{\ce}(x^j)}\Big\langle  \nabla^2 W_{\ce}^j n,\nabla\Big(-\sum_{k\neq j}W^k_\ce\Big) \Big\rangle \,\ud\Huno\\
&+ \frac{1}{2\pi\ce}\int_{\partial B_\ce(x^j)}\Big\langle\nabla\Big(-\sum_{k\neq j}W^k_\ce\Big),\Pi(b^j)\Big\rangle\,\ud\Huno=\mathrm{O}(\ce)\,.
 \end{aligned}
 \end{equation}
By adding \eqref{2201182343} and \eqref{dopo}, in view of \eqref{2201101612}, we get 
\begin{equation}\label{20220214_5}
\begin{aligned}
\frac{1-\nu^2}{E}\int_{\partial B_\ce(x^j)}(\partial_{n}\Delta W_{\ce}^j)\widetilde{w} \,\ud\Huno+
  \frac{1+\nu}{E}\nu\int_{\partial B_{\ce}(x^j)}  \Delta W_{\ce}^j \partial_n\widetilde{w} \,\ud\Huno&\\
-\frac{1+\nu}{E}\int_{\partial B_{\ce}(x^j)}\langle  \nabla^2 W_{\ce}^j n,\nabla\widetilde{w}\rangle \,\ud\Huno
+\frac{1}{2\pi\ce}\int_{\partial B_\ce(x^j)}\big\langle \nabla \widetilde{w},\Pi(b^j)\big\rangle\,\ud\Huno &= \mathrm{O}(\ce)\,.
\end{aligned}
\end{equation}

Now we focus on the case $i\neq j$ in \eqref{20220222_3}\,.
We first notice that, by the Gauss--Green Theorem, for any affine function $a$ there holds
\begin{equation}\label{220210_1}
\begin{aligned}
0=& \int_{B_{\ce}(x^i)}
 \Delta W_{\ce}^j\Delta(-W_{\ce}^i+a)\,\ud x=
 \int_{B_{\ce}(x^i)}{\Delta^2W_{\ce}^j} (-W_{\ce}^i+a)\, \ud x\\
&- \int_{\partial B_{\ce}(x^i)}(\partial_{(-n)}\Delta W_{\ce}^j) (-W_{\ce}^i+a) \,\ud\Huno
+ \int_{\partial B_{\ce}(x^i)}\Delta W_{\ce}^j \partial_{(-n)}(-W_{\ce}^i+a)\,\ud\Huno\\
=&\int_{\partial B_{\ce}(x^i)}(\partial_{n}\Delta W_{\ce}^j) (-W_{\ce}^i+a) \,\ud\Huno- \int_{\partial B_{\ce}(x^i)}\Delta W_{\ce}^j \partial_{n}(-W_{\ce}^i+a)\,\ud\Huno\,,
\end{aligned}
\end{equation}
where the first equality follows from the fact that  $W_\ce^i$ is affine in $B_\ce(x^i)$ whereas the last one is a consequence of $\Delta^2W^j_\ce=0$ in $A_{\ce,R}(x^j)$\,.
Similarly, we have
\begin{equation}\label{220210_2}
\begin{aligned}
0=&
 \int_{B_{\ce}(x^i)}
 \nabla^2W_{\ce}^j:\nabla^2 (-W_{\ce}^i+a) \,\ud x=
 \int_{B_{\ce}(x^i)}\Delta^2W_{\ce}^j(-W_{\ce}^i+a)\, \ud x\\
&- \int_{\partial B_{\ce}(x^i)}(\partial_{(-n)}\Delta W_{\ce}^j) (-W_{\ce}^i+a) \,\ud\Huno
+ \int_{\partial B_{\ce}(x^i)}\langle \nabla^2 W_{\ce}^j   (-n),\nabla (-W_{\ce}^i+a)  \rangle\,\ud\Huno\\
=&\int_{\partial B_{\ce}(x^i)}(\partial_{n}\Delta W_{\ce}^j) (-W_{\ce}^i+a) \,\ud\Huno
- \int_{\partial B_{\ce}(x^i)}\langle \nabla^2 W_{\ce}^j   n,\nabla (-W_{\ce}^i+a)  \rangle\,\ud\Huno\,.
\end{aligned}
\end{equation}
Furthermore, we have 
\begin{equation}\label{220210_3}
\begin{aligned}
&\int_{\partial B_{\ce}(x^i)}(\partial_{n}\Delta W_{\ce}^j)   \Big(  -\sum_{k\neq i} W^k_{\ce} \Big)\,\ud\Huno{=\mathrm{O}(\ep)}\\
&\int_{\partial B_{\ce}(x^i)}\Delta W_{\ce}^j \partial_{n}  \Big(   -\sum_{k\neq i} W^k_{\ce} \Big)\,\ud\Huno {=\mathrm{O}(\ep)}\\
&\int_{\partial B_{\ce}(x^i)}\Big\langle  \nabla^2 W_{\ce}^j n,\nabla  \Big( -\sum_{k\neq i} W^k_{\ce} \Big) \Big\rangle\,\ud\Huno {=\mathrm{O}(\ep)}\,,
\end{aligned}
\end{equation}
since all the integrands are uniformly bounded in $\ce$ and the domain of integration is vanishing.
Therefore, in view of \eqref{2201101612}, by \eqref{220210_1}, \eqref{220210_2}, \eqref{220210_3},
for any function $\widetilde{w}\in\widetilde{\Bb}_{\ce,\Omega}^\alpha$ we have that
\begin{equation*}
\begin{aligned}
-\nu \sum_{i\neq j}\int_{\partial B_{\ce}(x^i) }(\partial_{n}\Delta W_{\ce}^j) \widetilde{w}\,\ud\Huno
+\nu  \sum_{i\neq j}\int_{\partial B_{\ce}(x^i)}\Delta W_{\ce}^j \partial_{n} \widetilde{w}\,\ud\Huno& \\
+ \sum_{i\neq j}\int_{\partial B_{\ce}(x^i)}(\partial_{n}\Delta W_{\ce}^j)\widetilde{w}\,\ud\Huno
- \sum_{i\neq j}\int_{\partial B_{\ce}(x^i)}\langle  \nabla^2 W_{\ce}^j n,\nabla \widetilde{w} \rangle\,\ud\Huno&=\mathrm{O}(\ce)\,,
\end{aligned}
\end{equation*}
which, together with \eqref{20220214_5}, implies \eqref{20220222_3}.

Since the functions $W^j_\ce$ (for every $j=1,\ldots,J$) are uniformly bounded with respect to $\ce$ on $\partial\Omega$, by the standard trace theorem we get 
\begin{equation}\label{20220211_3}
\begin{aligned}
&\,\Bigg|-(1-\nu)\int_{\partial\Omega} (\partial_n\Delta W_\ce^j)\widetilde{w}\,\ud\Huno+\int_{\partial\Omega}\langle\nabla^2W_\ce^j,\nabla\widetilde w\rangle\,\ud\Huno
-\nu\int_{\partial\Omega}\Delta W_\ce^j\partial_n\widetilde{w}\,\ud\Huno\bigg|\\
\le&\, C\|\widetilde{w}\|_{H^1(\partial\Omega)}\le C\|\widetilde{w}\|_{H^2(\Omega_\ce(\alpha))}\,,
\end{aligned}
\end{equation}
where $C>0$ is a constant that does not depend on $\ce$\,.

In view of Lemma \ref{20220222_1}, by \eqref{20220222_3} and \eqref{20220211_3} (summing over $j=1,\ldots,J$), for $\ce$ small enough, we get
\begin{equation}\label{20220211_4}
\begin{aligned}
&\,\bigg|\frac{1+\nu}{E} \sum_{j=1}^J  \int_{\Omega_{\ce}(\alpha)}   
\big( \nabla^2W^j_{\ce}:\nabla^2\widetilde{w}-\nu  \Delta W^j_{\ce}\Delta \widetilde{w}\big)\,\ud x
+\sum_{j=1}^J\frac{1}{2\pi\ce}\int_{\partial B_{\ce}(x^j)}
\langle  \nabla \widetilde{w}, \Pi(b^j)\rangle\,\ud\Huno\bigg|\\
\le&\, C\big(\|\widetilde{w}\|_{H^2(\Omega_\ce(\alpha))}+1\big)\,,
\end{aligned}
\end{equation}
for some constant $C>0$ that does not depend on $\ce$\,.

Now, by applying Proposition \ref{202112230052} with $f=W_\ce^\alpha$ and by the very definition of $\G$ in \eqref{energyairy}, we deduce the existence of two constants $0<C_1<C_2$ independent of $\ce$ (but depending on $\alpha$ and $\Omega$) such that
\begin{equation}\label{20220214_1}
C_1\big(\|\widetilde{w}\|^2_{H^2(\Omega_\ce(\alpha))}-{\|W_\ce^\alpha\|^2_{L^\infty(\partial\Omega)}-\|\nabla W_\ce^\alpha\|^2_{L^\infty(\partial\Omega)}\big)}\le \G(\widetilde{w};\Omega_\ce(\alpha))\le C_2\|\widetilde{w}\|^2_{H^2(\Omega_\ce(\alpha))}\,,
\end{equation}
for every $\widetilde w\in \widetilde{\Bb}_{\ce,\Omega}^\alpha$\,. Therefore, by \eqref{20220211_4} and 
\eqref{20220214_1}, we deduce \eqref{2201182011}.
 By \eqref{2201182011}, existence and uniqueness of the solution $\widetilde{w}_\ce^\alpha$ to the minimization problem \eqref{2112192003} for $\ce>0$ small enough follows by the direct method in the Calculus of Variations. 
Furthermore, by \eqref{2201182011} and by Proposition \ref{2201202252} applied with $f=W_\ce^\alpha$ and $f^j=\sum_{i\neq j}W_\ce^i$\,, we have that
\begin{equation}\label{20220214_2}
C'\|\widetilde{w}_{\ce}^\alpha\|^2_{H^2(\Omega)}\le\widetilde{\I}_\ce^\alpha(\widetilde{w}_{\ce}^\alpha)+C''\,,
\end{equation}
for some constants $C',C''>0$ independent of $\ce$ (but depending on $\alpha$ and $\Omega$).
Hence, in order to conclude the proof it is enough to construct (for $\ce$ small enough) a competitor function $\widehat{w}^\alpha_\ce\in\widetilde{\Bb}_{\ce,\Omega}^\alpha$ such that
\begin{equation}\label{20220214_3}
\widetilde{\I}_\ce^\alpha(\widehat{w}_\ce^\alpha)\le C
\end{equation}
for some constant $C>0$ independent of $\ce$\,.

We construct $\widehat{w}_\ep^\alpha$ as follows. 
{Recalling the definition of $D$ in \eqref{distaminima},}
for every $j=1,\dots,J$\,, we consider $\varphi^j\in C^{\infty}(\Omega)$ be such that
$\varphi^j\equiv 0$ on $\overline{B}_{\frac D 4}(x^j)$\,,
$\varphi^j\equiv 1$ on $\Omega_{\frac D 2}(\alpha)$\,,
and $|\nabla\ffi(x)|\le \frac{C}{|x-x^j|}$ for every $x\in A_{\frac{D}{4},\frac{D}{2}}(x^j)$\,;
for every $\ce$ small enough\,, we define $\widehat{w}^\alpha_\ce\colon \Omega\to\R$ as
$$
\widehat{w}^\alpha_\ce\coloneqq-\sum_{i=1}^J \varphi^j W^j_{\ce}.
$$
By construction,
$$
\widehat{w}^\alpha_\ce+W^\alpha_\ce=\sum_{j=1}^J(1-\ffi^j)W^j_{\ce}\in\widetilde{\Bb}_{\ce,\Omega}^\alpha
$$
and
\begin{equation}\label{20220214_4}
\|\widehat{w}^\alpha_\ce\|_{H^2(\Omega_{\ce}(\alpha))}\le
\|\widehat{w}^\alpha_\ce\|_{H^2(\Omega)}\le\sum_{j=1}^J\|\ffi^jW_\ce^j\|_{H^2\big({A}_{\frac D 4,R}(x^j)\big)}\le C\,,
\end{equation}
for some constant $C>0$ independent of $\ce$ (but possibly depending on $\alpha$ and on $R$). By \eqref{2201182011} and \eqref{20220214_4} we obtain \eqref{20220214_3} and this concludes the proof.
%
\end{proof}

\begin{proof}[Proof of Proposition \ref{2201252355}]
%
We preliminarily notice that, since $\G$ is lower semicontinuous with respect to the weak $H^2$ convergence, {\eqref{2201141838}} yields
\begin{equation}\label{20220221_1}
\G(\widetilde{w}_0^\alpha;\Omega)\le\liminf_{\ce\to 0}\G(\widetilde{w}_\ce^\alpha;\Omega_\ce(\alpha))\,,
\end{equation}
and hence
\begin{equation}\label{2201192250}
\widetilde{\I}_{0}^\alpha(\widetilde{w}_0^\alpha)\le\liminf_{\ce\to 0}\widetilde{\I}_{\ce}^\alpha(\widetilde{w}_\ce^\alpha)\,.
\end{equation}
Here we have used that the boundary integrals on $\partial B_\ce(x^j)$ vanish as $\ce\to 0$ in view of \eqref{20220222_3}, and that, 
by compactness of the trace operator \cite[Theorem~6.2, page~103]{Necas11} (see also Remark~\ref{maybeuseful}), as $\ce\to 0$\,,
\begin{equation}\label{numerata1606}
\begin{aligned}
\int_{\partial\Omega}(\partial_{n}\Delta W_{\ce}^j)\widetilde{w}_\ce^\alpha \,\ud\Huno\to& \int_{\partial\Omega}(\partial_{n}\Delta W_{0}^j)\widetilde{w}_0^\alpha \,\ud\Huno\\
\int_{\partial\Omega}\langle \nabla^2 W_{\ce}^j   n,\nabla \widetilde{w}_\ce^\alpha \rangle \,\ud\Huno\to& \int_{\partial\Omega}\langle \nabla^2 W_{0}^j   n,\nabla \widetilde{w}_0^\alpha \rangle \,\ud\Huno\\
\int_{\partial\Omega}\Delta W_{\ce}^j \partial_{n}\widetilde{w}_\ce^\alpha \,\ud\Huno\to& \int_{\partial\Omega}\Delta W_{0}^j \partial_{n}\widetilde{w}_0^\alpha \,\ud\Huno\,.
\end{aligned}
\end{equation}


Moreover, by Proposition \ref{prop:approx}
for every $\widehat{w}_0\in\widetilde{\Bb}_{0,\Omega}^\alpha$
 there exists a sequence $\{\widehat{w}_\ce\}_\ce\subset H^2(\Omega)$ with $\widehat{w}_\ce\in\widetilde{\Bb}_{\ce,\Omega}^\alpha$ (for every $\ce>0$) such that $\widehat{w}_\ce\to \widehat{w}_0$ strongly in $H^2(\Omega)$\,. It follows that
 \begin{equation}\label{20220222_6}
\widetilde{\mathcal{I}}_{0}^\alpha(\widehat w_0)=\lim_{\ce\to 0}\widetilde{\mathcal{I}}_{\ce}^\alpha(\widehat w_\ce)\,,
\end{equation}
which, by the minimality of $\widetilde{w}_{\ce}^\alpha$ and in view of \eqref{2201192250}, gives
\begin{equation}\label{20220222_6}
\widetilde{\I}_{0}^\alpha(\widehat{w}_0)=\lim_{\ce\to 0}\widetilde{\mathcal{I}}_{\ce}^\alpha(\widehat w_\ce)\ge\limsup_{\ce\to 0}\widetilde{\I}_{\ce}^\alpha(\widetilde{w}_\ce^\alpha)\ge \widetilde{\I}_{0}^\alpha(\widetilde{w}_0^\alpha)\,.
\end{equation}
It follows that $\widetilde{w}_0^\alpha$ is a minimizer of $\widetilde{\I}_0^\alpha$ in $\widetilde{\Bb}_{0,\Omega}^\alpha$\,. 
By convexity (see \eqref{strictconv}), such a minimizer is unique and, by computing the first variation of $\widetilde{\I}_{0}^\alpha$ in $\widetilde{w}_0^\alpha$\,, we have that it satisfies \eqref{20220419}.
Furthermore, by applying \eqref{20220222_6} with $\widehat{w}_0=\widetilde{w}_0^\alpha$ we get \eqref{20220222_7}.


%
Finally, we discuss the strong convergence of $\widetilde{w}_{\ce}^\alpha$ in the compact subsets of of $\Omega\setminus\bigcup_{j=1}^J\{x^j\}$\,.
To this purpose, we preliminarily notice that, from \eqref{20220222_7}, \eqref{20220222_3}, and \eqref{numerata1606}, we have that
\begin{equation*}
\lim_{\ce\to 0}\G(\widetilde{w}^\alpha_\ce;\Omega_\ce(\alpha))=\G(\widetilde{w}^\alpha_0;\Omega)\,.
\end{equation*}
We now want to show that for every (fixed) $r>0$
\begin{equation}\label{440}
\int_{\Omega_{r}(\alpha)}
|\nabla^2\widetilde{w}^\alpha_{\ce}- \nabla^2 \widetilde{w}^\alpha_0 |^2
\,\ud x\to 0\qquad\text{as $\ce\to 0$\,.}
\end{equation}
To this purpose, we will use the weak convergence \eqref{2201141838} and Remark~\ref{equinorm}; we start by observing that
\begin{equation*}
\begin{aligned}
&\int_{\Omega_{r}(\alpha)}
|\nabla^2\widetilde{w}^\alpha_{\ce}- \nabla^2 \widetilde{w}_0^\alpha |^2\,\ud x
-\nu\int_{\Omega_{r}(\alpha)}
|\Delta\widetilde{w}_{\ce}^\alpha- \Delta \widetilde{w}^\alpha_0 |^2
\,\ud x
\\
=&\int_{\Omega_{r}(\alpha)} \!\!\! \big(
|\nabla^2\widetilde{w}_{\ce}^\alpha |^2+| \nabla^2 \widetilde{w}_0^\alpha |^2-2
\nabla^2 \widetilde{w}_0^\alpha:\nabla^2\widetilde{w}^\alpha_{\ce}\big)\,\ud x 
-\nu\int_{\Omega_{r}(\alpha)} \!\!\! \big(
|\Delta\widetilde{w}_{\ce}^\alpha |^2+| \Delta \widetilde{w}^\alpha_0 |^2-2
\Delta \widetilde{w}_0^\alpha \Delta\widetilde{w}_{\ce}^\alpha\big)\,\ud x\,,
\end{aligned}
\end{equation*}
whence, thanks to the convergence
\eqref{2201141838},
we  deduce
\begin{equation}\label{20220225_1}
\int_{\Omega_{r}(\alpha)}
|\nabla^2\widetilde{w}_{\ce}^\alpha- \nabla^2 \widetilde{w}^\alpha_0 |^2
 \,\ud x
-\nu\int_{\Omega_{r}(\alpha)}
|\Delta\widetilde{w}^\alpha_{\ce}- \Delta \widetilde{w}_0^\alpha |^2
 \,\ud x \to 0\,.
 \end{equation}
Since (see the first inequality in \eqref{quadr})
\begin{equation*}
c(\nu)
\int_{\Omega_{r}(\alpha)}
|\nabla^2\widetilde{w}^\alpha_{\ce}- \nabla^2 \widetilde{w}_0^\alpha |^2
\,\ud x
\le
\int_{\Omega_{r}(\alpha)}
|\nabla^2\widetilde{w}^\alpha_{\ce}- \nabla^2 \widetilde{w}_0^\alpha |^2
 \,\ud x
-\nu\int_{\Omega_{r}(\alpha)}
|\Delta\widetilde{w}^\alpha_{\ce}- \Delta\widetilde{w}_0^\alpha |^2
\,\ud x
\end{equation*}
for some constant $c(\nu)>0$ depending only on $\nu$\,, by \eqref{20220225_1},
we get \eqref{440}.
Finally, by \eqref{2201141838}, we get that $\widetilde{w}_\ce^\alpha$ converges strongly in $H^1(\Omega)$, as $\ce\to0$, to $\widetilde{w}_0^\alpha$\,,
which together with 
\eqref{440}, implies that
\begin{equation}\label{20220225_3}
\widetilde{w}_\ce^\alpha\to \widetilde{w}_0^\alpha\qquad\textrm{ strongly in }H^{2}(\Omega_r(\alpha))\,.
\end{equation}   
In conclusion, for any compact set $K\subset\Omega\setminus\bigcup_{j=1}^J\{x^j\}$\,, there exists $r>0$ such that $K\subset\Omega_{r}(\alpha)$\,, which, in view of  \eqref{20220225_3}, implies the claim and concludes the proof of the  proposition.
\end{proof}

We are in a position to discuss the asymptotic expansion of energies and to classify each term of the expansion. 
\begin{theorem}\label{CLequiv}
For every $\ce>0$ small enough, let $w_\ce^\alpha$ be the minimizer of $\I_\ce^\alpha$ in $\Bb_{\ce,\Omega}^\alpha$\,. 
Then we have
\begin{equation}\label{20220222_8}
\I_\ce^\alpha(w_\ce^\alpha)=-\frac{E}{1-\nu^2} \sum_{j=1}^J 
\frac{|b_j|^2}{8\pi}|\log\ce|
+F(\alpha)+f(D,R;\alpha)+\omega_\ce\,,
\end{equation}
where $\omega_\ce\to 0$ as $\ce\to 0$\,,
\begin{equation}\label{20220224_1}
\begin{aligned}
f(D,R;\alpha)=\sum_{j=1}^J\frac{|b^j|^2}{8\pi}\frac{E}{1-\nu^2}\bigg(&\,2+\frac{D^2}{R^2}\Big(\frac{D^2}{R^2}-2\Big)-2\log R\\
&\,+\frac{1}{4(1-\nu)}\frac{D^2}{R^2}\Big(\frac{R^2}{D^2}-1\Big)\Big(\frac{D^2}{R^2}\Big(\frac{R^2}{D^2}+1\Big)-2\Big)\bigg)\,,
\end{aligned}
\end{equation}
(recall \eqref{distaminima} for the definition of $D$) and 
\begin{equation}\label{ren_en_dislo}
F(\alpha)=
F^{\mathrm{self}}(\alpha)+
F^{\mathrm{int}}(\alpha)+
F^{\mathrm{elastic}}(\alpha)
\end{equation}
is the renormalized energy defined by
\begin{equation}\label{2201201755}
F^{\mathrm{self}}(\alpha)\coloneqq\sum_{j=1}^J \mathcal{G}( W^j_{0};\Omega_{D}(\alpha))+
\frac{E}{1-\nu^2} \sum_{j=1}^J 
\frac{{|b_j|^2}}{8\pi}\log D\,,
\end{equation}
\begin{equation}\label{2201260016}
\begin{aligned}
F^{\mathrm{int}}(\alpha)\coloneqq&\frac{1+\nu}{E}\sum_{j=1}^J\sum_{k\neq j} 
\Big(  -(1-\nu)\int_{\partial\Omega}(\partial_{n}\Delta{ W_{0}^j}){W^k_0} \,\ud\Huno
\\
&\phantom{\frac{1+\nu}{E}\sum_{j=1}^J\sum_{k\neq j} }+ \int_{\partial\Omega}\langle \nabla^2 {W_{0}^j }  n,\nabla {W^k_0}\rangle\,\ud\Huno
-\nu \int_{\partial\Omega}\Delta {W_{0}^j} \partial_{n}{W^k_0} \,\ud\Huno\Big)\,,
\end{aligned}
\end{equation}
\begin{equation}\label{2201221615}
F^{\mathrm{elastic}}(\alpha)\coloneqq\widetilde{\I}_0^\alpha(\widetilde{w}_0^\alpha)\,.
\end{equation}
\end{theorem}
 \begin{remark}
\rm{Notice that $F^{\mathrm{self}}(\alpha)$ is independent of $D$\, 
 as it can be verified by a simple computation.
 }
 \end{remark}

\begin{proof}
By \eqref{2201181904} and \eqref{20220223_2},  we have that
$w_\ce^\alpha=W_\ce^\alpha+\widetilde{w}_\ce^\alpha$\,, where $W_\ce^\alpha$ is defined in \eqref{20220223_1} and $\widetilde{w}_\ce^\alpha$ is the unique minimizer of $\widetilde{\I}_\ce^\alpha$ in $\widetilde{\Bb}_{\ce,\Omega}^\alpha$\,\ provided by Proposition~\ref{2201252355}. 
Notice that
\begin{equation}\label{20220223_3}
\begin{aligned}
&\,\G(W_\ce^\alpha;\Omega_\ce(\alpha))+\sum_{j=1}^J\frac{1}{2\pi\ce}\int_{\partial B_\ce(x^j)}\langle\nabla W_\ce^\alpha,\Pi(b^j)\rangle\,\ud\Huno\\
=&\,\sum_{j=1}^J\Big(\G(W_\ce^j;\Omega_\ce(\alpha))+\frac1{2\pi\ce}\int_{\partial B_\ce(x^j)}\langle\nabla W_\ce^j,\Pi(b^j)\rangle\,\ud\Huno\Big)\\
&\,+\sum_{j=1}^J\sum_{k\neq j}\bigg(\frac{1+\nu}{E} \!\! \int_{\Omega_\ce(\alpha)} \!\!\! \Big(\nabla^2 W_\ce^j:\nabla^2W_\ce^k-\nu\Delta W_\ce^j\Delta W_\ce^k\Big)\,\ud x\\
&\,+\frac1{2\pi\ce}\int_{\partial B_\ce(x^j)} \!\!\!\! \langle\nabla W_\ce^k,\Pi(b^j)\rangle\,\ud\Huno\bigg)\\
\eqqcolon&\, F^{\mathrm{self}}_\ce(\alpha)+F^{\mathrm{int}}_\ce(\alpha)\,.
\end{aligned}
\end{equation}
We notice that, for every $j=1,\ldots,J$ and for every $0<\ce<r\le D$ with $\ce<1$
\begin{equation}\label{20220223_5}
\begin{aligned}
&\,\G(W^j_\ce;\Omega_\ce(\alpha))+\frac{1}{2\pi\ce}\int_{\partial B_\ce(x^j)}\langle\nabla W_\ce^j,\Pi(b^j)\rangle\,\ud\Huno\\
=&\,\G(W^j_\ce;\Omega_r(\alpha))+\G(W^j_\ce;A_{\ce,r}(x^j))+\frac{1}{2\pi\ce}\int_{\partial B_\ce(x^j)}\langle\nabla W_\ce^j,\Pi(b^j)\rangle\,\ud\Huno\,.
\end{aligned}
\end{equation}
Furthermore, by Corollary~\ref{insec4}, we have that
\begin{multline}\label{20220224_2}
\G(W_\ce^j; A_{\ce,r}(x^j))+\frac{1}{2\pi\ce}\int_{\partial B_\ce(x^j)}\langle\nabla W_\ce^j,\Pi(b^j)\rangle\,\ud\Huno
=\,-\frac{|b^j|^2}{8\pi}\frac{E}{1-\nu^2}\log\frac{1}{\ce}\\
+\frac{|b^j|^2}{8\pi}\frac{E}{1-\nu^2}\log r+f_\ce(r,R;|b^j|)\,,
\end{multline}
where $f_\ce(r,R;|b^j|)$ is defined in \eqref{vanerr}.

%
Notice moreover that $f_{\ce}(r,R;|b^j|)\to f(r,R;|b^j|)$ (as $\ce\to 0$) with $f(r,R;|b^j|)$ defined by 
\begin{equation}\label{20220224_3}
\begin{aligned}
f(r,R;|b^j|)\coloneqq&\frac{|b^j|^2}{8\pi}\frac{E}{1-\nu^2}\Big(2+\frac{r^2}{R^2}\Big(\frac{r^2}{R^2}-2\Big)-2\log R\Big)\\
&+\frac{|b^j|^2}{32\pi}\frac{E}{(1-\nu)^2(1+\nu)}\frac{r^2}{R^2}\Big(\frac{R^2}{r^2}-1\Big)\Big(\frac{r^2}{R^2}\Big(\frac{R^2}{r^2}+1\Big)-2\Big)\,.
\end{aligned}
\end{equation}
By Remark \ref{maybeuseful}, summing over $j=1,\ldots,J$ formulas \eqref{20220223_5}, \eqref{20220224_2} and \eqref{20220224_3}, for $r= D$ we obtain
\begin{equation}\label{20220224_4}
F^{\mathrm{self}}_\ce(\alpha)=-\sum_{j=1}^J\frac{|b^j|^2}{4\pi}\frac{E}{1-\nu^2}|\log\ce|+F^{\mathrm{self}}(\alpha)+f(D,R;\alpha)+\omega_\ce\,,
\end{equation}
where $\omega_\ce\to 0$ as $\ce\to 0$ and $f(D,R;\alpha):=\sum_{j=1}^Jf(D,R;|b^j|)$\,.

We now focus on $F_\ce^{\mathrm{int}}(\alpha)$\,.
By arguing as in the proof of Lemma \ref{20220222_1}, for every $j,k=1,\ldots,J$ with $k\neq j$\,, we have that
\begin{equation*}
\begin{aligned}
&\int_{\Omega_\ce(\alpha)}\Big(\nabla^2 W_\ce^j:\nabla^2W_\ce^k-\nu\Delta W_\ce^j\Delta W_\ce^k\Big)\,\ud x\\
=&-(1-\nu)\int_{\partial\Omega}(\partial_{n}\Delta W_{\ce}^j){W}_\ce^k \,\ud\Huno+\int_{\partial\Omega}\langle \nabla^2 W_{\ce}^j   n,\nabla W_\ce^k \rangle \,\ud\Huno-\nu \int_{\partial\Omega}\Delta W_{\ce}^j \partial_{n}W_\ce^k \,\ud\Huno\\
&+(1-\nu)\sum_{i=1}^J\int_{\partial B_\ce(x^i)}(\partial_{n}\Delta W_{\ce}^j)W_\ce^k \,\ud\Huno\\
&- \sum_{i=1}^J\int_{\partial B_{\ce}(x^i)}\langle  \nabla^2 W_{\ce}^j n,\nabla W_\ce^k \rangle \,\ud\Huno+\nu\sum_{i=1}^J\int_{\partial B_{\ce}(x^i)}\Delta W_{\ce}^j \partial_{n} W_\ce^k \,\ud\Huno\,,
\end{aligned}
\end{equation*}
which, in view of \eqref{20220216_1}, \eqref{220210_1}, and \eqref{220210_3}, and using Remark~\ref{maybeuseful}, implies
\begin{equation}\label{20220223_4}
F^{\mathrm{int}}_\ce(\alpha)=F^{\mathrm{int}}(\alpha)+\omega_\ce\,,
\end{equation}
where $\omega_\ce\to 0$ as $\ce\to 0$\,.

Finally, by \eqref{20220223_3}, \eqref{20220223_5}, \eqref{20220224_4}, and \eqref{20220223_4}, we get
\begin{equation*}
\begin{aligned}
&\G(W_\ce^\alpha;\Omega_\ce(\alpha))+\sum_{j=1}^J\frac{1}{2\pi\ce}\int_{\partial B_\ce(x^j)}\langle\nabla W_\ce^\alpha,\Pi(b^j)\rangle\,\ud\Huno\\
=&-\sum_{j=1}^J\frac{|b^j|^2}{4\pi}\frac{E}{1-\nu^2}|\log\ce|+F^{\mathrm{self}}(\alpha)+f(D,R;\alpha)+F^{\mathrm{int}}(\alpha)+\omega_\ce\,,
\end{aligned}
\end{equation*}
which, by \eqref{20220223_2} together with Propositions \ref{2201181930} and \ref{2201252355}, allows us to conclude the proof.
\end{proof}


We conclude by showing, via a diagonal argument, that the asymptotic behavior in Theorem~\ref{CLequiv} remains valid also for systems of disclination dipoles,
that is, when the finite system 
$\alpha\in\ED(\Omega)$ of edge dislocations is replaced with the approximating system of disclination dipoles.
\begin{theorem}\label{diago}
Let $J\in\N$, let $b^1,\ldots,b^J\in\R^2\setminus\{0\}$, and let $x^1,\ldots,x^J$ be distinct points in~$\Omega$\,.
For every $h>0$\,, let $\theta_h\in\WD(\Omega)$ be the measure defined in \eqref{thetahJ}.
Then, 
\begin{equation}\label{eovvia}
\theta_h\weakstar\alpha\coloneqq\sum_{j=1}^Jb^j\de_{x^j}\in\ED(\Omega)\qquad \text{as $h\to 0$\,.}
\end{equation}
Let $D>0$ be as in \eqref{distaminima}; for every $0<h<\ep<D$ let $w^{\theta_h}_{h,\ep}$ be the unique minimizer in $\Bb_{\ep,\Omega}^\alpha$ of the functional $\mathcal{I}_{h,\ep}^{\theta_h}$ defined in \eqref{2202231647}. 
Then there exists a function $\ep\colon\R^+\to\R^+$  with $\ep(h)>h$ and $\ep(h)\to 0$ as $h\to 0$ such that $w_{h,\ep(h)}^{\theta_h}\to w_0^\alpha$ in $H^2_\loc\big(\Omega\setminus\bigcup_{j=1}^J\{x^j\}\big)$ as $h\to 0$\,, where $w_0^\alpha$ is the function provided by Theorem \ref{2201181928}.
Moreover,
\begin{equation}\label{enerinormafine}
\mathcal{I}_{h,\ep(h)}^{\theta_h}\big(w_{h,\ep(h)}^{\theta_h}\big)=-\frac{E}{1-\nu^2} \sum_{j=1}^J 
\frac{|b_j|^2}{8\pi}|\log\ce(h)|
+F(\alpha)+f(D,R;\alpha)+\omega_{h}\,,
\end{equation}
where $F(\alpha)$ and $f(D,R;\alpha)$ are defined in \eqref{ren_en_dislo} and \eqref{20220224_1}, respectively, and $\omega_h\to 0$ as $h\to 0$\,.
\end{theorem}
%
%
%
\begin{proof}
Convergence \eqref{eovvia} is obvious.
Let now $0<\ce<D$ be fixed.
{By Remark~\ref{2202211829}}, there exists $\bar h<\ce$ such that, for every $h<\bar h$, 
\begin{equation}\label{fattaprima}
\lVert w^{\theta_h}_{h,\ce}-w^\alpha_{0,\ce}\rVert_{H^2(\Omega)}<\ce\,,
\end{equation}
where $w_{0,\ce}^\alpha$ 
is the unique minimizer of \eqref{energyI} in $\Bb_{\ep,\Omega}^\alpha$\,.
Choose such an $h$, call it $h(\ce)$, and notice that this choice can be made in a strictly monotone fashion\,.
Let now $0<r<D$\,; 
by \eqref{fattaprima} and Theorem \ref{2201181928}\,, we get
$$
\big\lVert w^{\theta_{h(\ce)}}_{h(\ce),\ce}-w^\alpha_{0}\big\rVert_{H^2(\Omega_r(\alpha))}\leq \big\lVert w^{\theta_{h(\ce)}}_{h(\ce),\ce}-w^\alpha_{0,\ce}\big\rVert_{H^2(\Omega_r(\alpha))}+\lVert w^\alpha_{0,\ce}-w^\alpha_{0}\rVert_{H^2(\Omega_r(\alpha))}<\ce+\mathrm{o}_\ce\,,
$$
where $\mathrm{o}_\ce\to 0$ as $\ep\to 0$\,.
By the arbitrariness of $r$ we get that $w^{\theta_{h(\ce)}}_{h(\ce),\ce}\to w^\alpha_{0}$ in $H^2_\loc(\Omega\setminus\bigcup_{j=1}^{J}\{x^j\})$\,, and hence, by the strict monotonicity of the map $\ce\mapsto h(\ce)$\, the first part of the claim follows.
Finally, \eqref{enerinormafine} is an immediate consequence of Theorem \ref{CLequiv}.
\end{proof}



\appendix

\section{Equivalence of boundary conditions}\label{geom_lemmas}
Here we show that if $A$ is a domain of class~$C^2$ and $v\in C^2(\overline{A})$\,,  then the boundary condition $\nabla^2v\,t=0$ on $\partial A$  is equivalent to requiring that~$v|_{\Gamma}$ is the trace of an affine function on every connected component $\Gamma$ of $\partial A$\,.
To this end, we first state and prove the following geometric lemma.
\begin{lemma}\label{geolemma}
Let $A\subset \R^2$ be a bounded, open, simply connected set with $C^2$ boundary and set $\ell\coloneqq|\partial A|$. Let $\gamma\in C^2([0,\ell];\R^{2})$ be the arc-length parametrization of $\partial A$ and let $\vartheta\in C^1([0,\ell])$ such that $\gamma'(\xi)=(-\sin\vartheta(\xi);\cos\vartheta(\xi))$\,.
Set $\vk(\xi)\coloneqq\vartheta'(\xi)$ for every $\xi\in [0,\ell]$\,. 
Let $v\in C^2(\overline{A})$ and let $g_D,g_N\colon[0,\ell]\to\R$ be the functions defined by
 $g_D\coloneqq v\circ\gamma$ and $g_N\coloneqq\partial_n v\circ\gamma$\,.
Then
\begin{equation}\label{eqdiff}
\begin{cases}
g_D''(\xi)=\langle \nabla^2 v(\gamma(\xi))\, \gamma'(\xi),\gamma'(\xi)\rangle-\vk(\xi)g_N(\xi) \\[1mm]
g_N'(\xi)=\langle\nabla^2 v(\gamma(\xi))\,\gamma'(\xi),-(\gamma'(\xi))^{\perp}\rangle+\vk(\xi)g_D'(\xi)
\end{cases}\quad \text{for every $\xi\in[0,\ell]$\,.}
\end{equation}
\end{lemma}
\begin{proof}
By definition, the unit tangent vector is $t(\gamma(\xi))=\gamma'(\xi)=(-\sin\vartheta(\xi);\cos\vartheta(\xi))$ and the outer unit normal vector is $n(\gamma(\xi))=(-\gamma'(\xi))^{\perp}=(\cos\vartheta(\xi);\sin\vartheta(\xi))$\,, so that 
\begin{equation*}
\begin{aligned}
\frac{\ud}{\ud\xi}t(\gamma(\xi))=&\gamma''(\xi)=-\vk(\xi)n(\gamma(\xi))=-\vk(\xi)(-\gamma'(\xi))^{\perp}\\
\frac{\ud}{\ud\xi}n(\gamma(\xi))=&(-\gamma''(\xi))^{\perp}=\vk(\xi)t(\gamma(\xi))=\vk(\xi)\gamma'(\xi)\,,
\end{aligned}
\end{equation*}
and hence
\begin{eqnarray*}
g_D'(\xi)&=&\frac{\ud}{\ud\xi}v(\gamma(\xi))=\langle\nabla v(\gamma(\xi)),\gamma'(\xi)\rangle\,,\\
g_N'(\xi)&=&\frac{\ud}{\ud\xi}\langle\nabla v(\gamma(\xi)),n(\gamma(\xi))\rangle=\langle\nabla^2 v(\gamma(\xi))\gamma'(\xi),(-\gamma'(\xi))^{\perp})\rangle+\langle\nabla v(\gamma(\xi)),(-\gamma''(\xi))^{\perp}\rangle\\
&=&\langle\nabla^2 v(\gamma(\xi))\gamma'(\xi),(-\gamma'(\xi))^{\perp})\rangle+\vk(\xi)\langle\nabla v(\gamma(\xi)),\gamma'(\xi)\rangle\\
&=&\langle\nabla^2 v(\gamma(\xi))\gamma'(\xi),(-\gamma'(\xi))^{\perp})\rangle+\vk(\xi)g_D'(\xi)\\
g_D''(\xi)&=&\frac{\ud}{\ud\xi}\langle\nabla v(\gamma(\xi)),\gamma'(\xi)\rangle=\langle\nabla^2 v(\gamma(\xi))\gamma'(\xi),\gamma'(\xi)\rangle+\langle\nabla v(\gamma(\xi)),\gamma''(\xi)\rangle\\
&=&\langle\nabla^2 v(\gamma(\xi))\gamma'(\xi),\gamma'(\xi)\rangle-\vk(\xi)\langle\nabla v(\gamma(\xi)),n(\gamma(\xi))\rangle\\
&=&\langle\nabla^2 v(\gamma(\xi))\gamma'(\xi),\gamma'(\xi)\rangle-\vk(\xi)g_N(\xi)\,,
\end{eqnarray*}
that is \eqref{eqdiff}.
\end{proof}

We are now in a position to prove the main result of this section on the equivalence of the boundary conditions.
\begin{proposition}\label{2101141730}
Let $A\subset\R^2$ be an open and bounded set with boundary of class $C^2$\,.
Let $v\in C^2(\overline{A})$\,. 
Then, for every connected component $\Gamma$ of $\partial A$ we have that 
\begin{equation}\label{cambiocondA}
\nabla^2v\,t=0\textrm{ on }\Gamma\qquad\Leftrightarrow\qquad v=a\,,\quad\partial_nv=\partial_na\textrm{ on }\Gamma\,,
\end{equation}
for some affine function $a$\,.
\end{proposition}

\begin{proof}
Let $\Gamma$ be a connected component of $\partial A$\,, set $\ell\coloneqq|\Gamma|$ and let $\gamma\colon[0,\ell]\to\R^{2}$ be the arc-length parametrization of $\Gamma$\,,
so that the unit tangent vector is $t(\gamma(\xi))=\gamma'(\xi)$ and the outer unit normal vector is $n(\gamma(\xi))=\big(-\gamma'(\xi)\big)^{\perp}$\,.
Moreover, let $\vartheta\in C^1([0,\ell])$ be such that 
\begin{equation}\label{gammatan}
\gamma'(\xi)=(-\sin\vartheta(\xi);\cos\vartheta(\xi))
\end{equation} 
and set $\vk(\xi)\coloneqq\vartheta'(\xi)$ for every $\xi\in [0,\ell]$\,.
Recalling that $\{n(\gamma(\xi)),t(\gamma(\xi))\}$ is an orthonormal basis of $\R^{2}$ for every $\xi\in[0,\ell]$\,, we have
$\nabla^2v\, t=0$ on $\Gamma$ if and only if 
\begin{equation}\label{HttHtn}
\langle \nabla^2 v(\gamma(\xi))\, \gamma'(\xi),\gamma'(\xi)\rangle=\langle\nabla^2 v(\gamma(\xi))\,\gamma'(\xi),-(\gamma'(\xi))^{\perp}\rangle=0 \textrm{ for every }\xi\in[0,\ell] \,.
\end{equation}
Furthermore, letting $g_D,g_N\colon[0,\ell]\to\R$ be the functions defined by
 $g_D\coloneqq v\circ\gamma$ and $g_N\coloneqq\partial_n v\circ\gamma$\,, Lemma \ref{geolemma} and \eqref{HttHtn}, imply that \eqref{cambiocondA} is equivalent to 
\begin{equation}\label{eqdiff1}
\begin{cases}
g_D''(\xi)=-\vk(\xi)g_N(\xi) \\
g_N'(\xi)=\vk(\xi)g_D'(\xi)
\end{cases}\quad \text{for every $\xi\in[0,\ell]$}\quad\textrm{if and only if}\quad \begin{cases}
v=a& \text{on $\Gamma$} \\
\partial_nv=\partial_n a & \text{on $\Gamma$}\,,
\end{cases}
\end{equation}
for some affine function $a\colon\R^2\to\R^2$\,, namely, for a function $a$ of the form 
\begin{equation}\label{affine}
a(x)=c_0+c_1x_1+c_2x_2\,,
\end{equation}
 with $c_0,c_1,c_2\in\R$\,. 
If $v=a$ and $\partial_nv=\partial_n a$ on $\Gamma$\,, by straightforward computations for every $\xi\in[0,\ell]$ we get
\begin{equation*}
\begin{aligned}
g_N(\xi)=&c_1\cos\vartheta(\xi)+c_2\sin\vartheta(\xi)\,,\\
g_D'(\xi)=&-c_1\sin\vartheta(\xi)+c_2\cos\vartheta(\xi)\,,\\
g_N'(\xi)=&\vartheta'(\xi)\big(-c_1\sin\vartheta(\xi)+c_2\cos\vartheta(\xi)\big)=\vk(\xi)g'_D(\xi)\,,\\
g_D''(\xi)=&-\vartheta'(\xi)\big(c_1\cos\vartheta(\xi)+c_2\sin\vartheta(\xi)\big)=-\vk(\xi)g_N(\xi)\,,
\end{aligned}
\end{equation*}
which proves one implication in \eqref{eqdiff1}.
To prove the opposite implication, we study the ODE system {on} the left-hand side of \eqref{eqdiff1}, that, setting $z^1\coloneqq g_D'$ and $z^2\coloneqq g_N$\,, can be conveniently rewritten in the form
\begin{equation}\label{2012311409}
\begin{cases}
(z^1)'=-\vk z^2 \\
(z^2)'=\vk z^1\,.
\end{cases}
\end{equation}
By the classical theory of ordinary differential equations, since $\vk$ is a continuous function, for any given initial datum $z_0=(z^1_0;z^2_0)\in\R^2$\,, the Cauchy problem associated with the system \eqref{2012311409} with initial condition $(z^1(0);z^2(0))=z(0)=z_0$ admits a unique solution $z\in C^1([0,\ell];\R^2)$\,.
Furthermore, letting $\bar\vartheta$ denote a primitive of $\vk$\,, we observe that the functions $\xi\mapsto\bar z(\xi)\coloneqq (-\sin\bar\vartheta(\xi);\cos\bar\vartheta(\xi))$ and $\xi\mapsto\hat z(\xi)\coloneqq (\cos\bar\vartheta(\xi);\sin\bar\vartheta(\xi))$ provide a basis of solutions to \eqref{2012311409}. Therefore, since $\bar\vartheta$ and $\vartheta$ differ by a constant, any solution to  \eqref{2012311409} is of the form
\begin{equation*}
\begin{aligned}
(z^1;z^2)=(-c_1\sin\vartheta+c_2\cos\vartheta;c_1\cos\vartheta+c_2\sin\vartheta)\,,
\end{aligned}
\end{equation*}
so that, recalling the definitions of $z^1$ and $z^2$ and using \eqref{gammatan}, we get
\begin{equation}\label{evvai}
\begin{aligned}
g_N(\xi)=&c_1\cos\vartheta(\xi)+c_2\sin\vartheta(\xi)  =\langle (c_1;c_2),n(\gamma(\xi))\rangle\\
g_D(\xi)=&g_D(0)+\int_{0}^\xi\big(-c_1\sin\vartheta(\zeta)+c_2\cos\vartheta(\zeta)\big)\,\ud\zeta\eqqcolon c_0+\langle (c_1;c_2),\gamma(\xi)\rangle\,,
\end{aligned}
\end{equation}
where we have set $c_0\coloneqq g_D(0)-c_1\gamma^1(0)-c^2\gamma^2(0)$\,.
By \eqref{evvai} and the definitions of $g_D$ and $g_N$\,,
we get that $v=a$ and $\partial_n v=\partial_n a$ on $\Gamma$\,, for a certain function $a$ as in \eqref{affine}. 
This concludes the proof of the converse inequality and hence of the whole proposition.
\end{proof}

\section{Proof of Lemma \ref{lemma:energyvbarh}}\label{appendixprooflemma}

This section in devoted to the proof of Lemma \ref{lemma:energyvbarh}.
\begin{proof}
{By \eqref{perdislo} and \eqref{simplif}, straightforward computations show that} 
 \begin{equation}\label{hessbar}
 \begin{aligned}
 |\nabla^2\bar v_h(x)|^2=\frac{E^2}{(1-\nu^2)^2}\frac{s^2}{256\pi^2}
 \Bigg(&8\log^2\frac{\big(x_1-\frac{h}{2}\big)^2+x_2^2}{\big(x_1+\frac{h}{2}\big)^2+x_2^2}\\
 &+128\, \frac{h^2\,x_2^4\,x_1^2}{\bigg(\Big(\big(x_1-\frac{h}{2}\big)^2+x_2^2\Big)\Big(\big(x_1+\frac{h}{2}\big)^2+x_2^2\Big)\bigg)^2}\\
 &+32 x_2^2\bigg(\frac{x_1-\frac{h}{2}}{\big(x_1-{\frac{h}{2}})^2+x_2^2}-\frac{x_1+\frac{h}{2}}{\big(x_1+\frac{h}{2}\big)^2+x_2^2}\bigg)^2\Bigg)\,,
 \end{aligned}
 \end{equation}
and
 \begin{equation}\label{laplbar}
 |\Delta\bar v_h(x)|^2=\frac{E^2}{(1-\nu^2)^2}\frac{s^2}{16\pi^2}\log^2\frac{\big(x_1-\frac{h}{2}\big)^2+x_2^2}{\big(x_1+\frac{h}{2}\big)^2+x_2^2}\,.
 \end{equation}
For every open set $A\subset\R^2$ we set:
 \begin{eqnarray}\label{defF1}
  \F^1_h(A)&\!\!\!\!\coloneqq&\!\!\!\! \int_{A} \log^2\frac{\big(x_1-\frac{h}{2}\big)^2+x_2^2}{\big(x_1+\frac{h}{2}\big)^2+x_2^2}\,\ud x\,,\\ \label{defF2}
\F^2_h(A)&\!\!\!\!\coloneqq&\!\!\!\! h^2\int_{A} \frac{x_2^4\,x_1^2}{\bigg(\Big(\big(x_1-\frac{h}{2}\big)^2+x_2^2\Big)\Big(\big(x_1+\frac{h}{2}\big)^2+x_2^2\Big)\bigg)^2}\,\ud x\,,\\ \label{defF3}
\F^3_h(A)&\!\!\!\!\coloneqq&\!\!\!\! \int_{A}x_2^2\bigg(\frac{x_1-\frac{h}{2}}{(x_1-\frac{h}{2})^2+x_2^2}-\frac{x_1+\frac{h}{2}}{\big(x_1+\frac{h}{2}\big)^2+x_2^2}\bigg)^2\,\ud x\\ \nonumber
&\!\!\!\!=&\!\!\!\! h^2\int_{A}\frac{x_2^2\big(\frac{h^2}{4}+x_2^2-x_1^2\big)^2}{\bigg(\Big(\big(x_1-\frac{h}{2}\big)^2+x_2^2\Big)\Big(\big(x_1+\frac{h}{2}\big)^2+x_2^2\Big)\bigg)^2}\,\ud x\,,
\end{eqnarray}
so that, in view of \eqref{hessbar} and \eqref{laplbar}, it holds
\begin{equation}\label{enough}
\begin{aligned}
\int_{A}|\nabla^2\bar v_h|^2\,\ud x=&\,\frac{E^2}{(1-\nu^2)^2}\frac{s^2}{32\pi^2}\Big(\F^1_h(A)+16\F^2_h(A)+4\F^3_h(A)\Big)\,,\\
\int_{A} |\Delta\bar v_h|^2\,\ud x=&\,\frac{E^2}{(1-\nu^2)^2}\frac{s^2}{16\pi^2}\F^1_h(A)\,.
\end{aligned}
\end{equation}
We start by proving that
\begin{equation}\label{anello}
\lim_{h\to 0}\frac{1}{h^2\log\frac{R}{h}}\G(\bar v_h; A_{h,R}(0))=\frac{E}{1-\nu^2}\frac{s^2}{8\pi}\,.
\end{equation}
 To this end, by the very definition of $\G$ in \eqref{energyairy} and in view of \eqref{enough}, it is enough to show that
 \begin{eqnarray}\label{F1}
 \lim_{h\to 0}\frac{1}{h^2\log \frac{R}{h}} \F^1_h(A_{h,R}(0))&\!\!\!\!=&\!\!\!\! 4\pi\,,\\
 \label{F2}
  \lim_{h\to 0}\frac{1}{h^2\log \frac{R}{h}} \F^2_h(A_{h,R}(0))&\!\!\!\!=&\!\!\!\! \frac{\pi}{8}\,,\\
   \label{F3}
  \lim_{h\to 0}\frac{1}{h^2\log\frac{R}{h}} \F^3_h(A_{h,R}(0))&\!\!\!\!=&\!\!\!\! \frac{\pi}{2}\,.
 \end{eqnarray}
To this purpose, for every $0<h<R$\,, we set $N_h\coloneqq \big\lceil\frac{\log\frac R h}{\log 2}\big\rceil$\,, so that $2^{N_h-1}h\le R\le 2^{N_h}h$\,. 
We start by proving \eqref{F1}. 
By using the change of variable $x=2^{n-1}h y$ for every $n=1,\ldots,N_h$\,, we have
\begin{equation}\label{contopereffeuno}
\begin{aligned}
h^2\sum_{n=1}^{N_h-1}\frac{2^{2n}}{4}\int_{A_{1,2}(0)}\log^2\frac{(y_1-2^{-n})^2+y_2^2}{(y_1+2^{-n})^2+y_2^2}\,\ud y
\le \F^1_h(A_{h,R}(0))\\
\le h^2\sum_{n=1}^{N_h}\frac{2^{2n}}{4}\int_{A_{1,2}(0)}\log^2\frac{(y_1-2^{-n})^2+y_2^2}{(y_1+2^{-n})^2+y_2^2}\,\ud y\,.
\end{aligned}
\end{equation}
We start by discussing the limit of the left-hand side integral in \eqref{contopereffeuno}. 
Let $K< N_h$ and let $h$ be sufficiently small; 
{then, for every $n=\lfloor\frac{N_h}{K}\rfloor,\ldots,N_h-1$\,, and for every $y\in A_{1,2}(0)$\,, by the Taylor expansion of the functions $\log(1+t)$ and $((y_1+t)^2+y_2^2)^{-1}$ around the pole $t=0$\,, we have that
\begin{equation}\label{taylore}
\log^2\frac{(y_1-2^{-n})^2+y_2^2}{(y_1+2^{-n})^2+y_2^2}= \log^2\Big(1-2^{2-n}\frac{y_1}{(y_1+2^{-n})^2+y_2^2}\Big)
\ge 2^{4-2n}\frac{y_1^2}{|y|^4}-C2^{-3n}\,,
\end{equation}
for some universal constant $C>0$\,.
 }
Therefore, by \eqref{taylore}, we deduce that
\begin{equation*}
\begin{aligned}
&\sum_{n=1}^{N_h-1}\frac{2^{2n}}{4}\int_{A_{1,2}(0)}\log^2\frac{(y_1-2^{-n})^2+y_2^2}{(y_1+2^{-n})^2+y_2^2}\,\ud y
\ge \sum_{n=\lfloor\frac{N_h}{K}\rfloor}^{N_h-1}\frac{2^{2n}}{4}\int_{A_{1,2}(0)}\log^2\frac{(y_1-2^{-n})^2+y_2^2}{(y_1+2^{-n})^2+y_2^2}\,\ud y\\
\ge&  \sum_{n=\lfloor\frac{N_h}{K}\rfloor}^{N_h-1} \int_{A_{1,2}(0)}\frac{4y_1^2}{|y|^4}\,\ud y-C {\sum_{n=\lfloor\frac{N_h}{K}\rfloor}^{N_h-1}2^{-n}}
\ge N_h\Big(1-\frac{1}{K}\Big)4\pi\log 2-C\,.
\end{aligned}
\end{equation*}
By the very definition of $N_h$ and by \eqref{contopereffeuno}, we thus have that
\begin{equation}\label{quasiuno}
\frac{1}{h^2\log \frac{R}{h}}\F^1_h(A_{h,R}(0))\ge \Big(1-\frac{1}{K}\Big)4\pi+\omega(h)\,,
\end{equation}
where $\omega(h)\to 0$ as $h\to 0$\,.
By sending first $h\to 0$ and then $K\to +\infty$ in \eqref{quasiuno} we get the inequality ``$\ge$'' in \eqref{F1}.
As for the inequality ``$\le$'' in \eqref{F1}, 
{we preliminarily observe that, by arguing as in \eqref{taylore} for $h$ sufficiently small, for every $K<N_h$\,, for every $n=\lfloor\frac{N_h}{K}\rfloor,\ldots,N_h-1$\,, and for every $y\in A_{1,2}(0)$\,, it holds
$$
\log^2\frac{(y_1-2^{-n})^2+y_2^2}{(y_1+2^{-n})^2+y_2^2}= \log^2\Big(1-2^{2-n}\frac{y_1}{(y_1+2^{-n})^2+y_2^2}\Big)
\le 2^{4-2n}\frac{y_1^2}{|y|^4}+C2^{-3n}\,,
$$
for some universal constant $C>0$\,;
}
therefore, {for $h$ sufficiently small and for every $K<N_h$\,, we get}
\begin{equation*}
\begin{aligned}
 &\sum_{n=1}^{N_h}\frac{2^{2n}}{4}\int_{A_{1,2}(0)}\log^2\frac{(y_1-2^{-n})^2+y_2^2}{(y_1+2^{-n})^2+y_2^2}\,\ud y\\
 \le&\, C_0+ C_1\Big\lceil\frac{N_h}{K}\Big\rceil+\sum_{n=\lceil\frac{N_h}{K}\rceil}^{N_h}\frac{2^{2n}}{4}\int_{A_{1,2}(0)}\log^2\frac{(y_1-2^{-n})^2+y_2^2}{(y_1+2^{-n})^2+y_2^2}\,\ud y\\
 \le&\, C_0+ C_1\Big\lceil\frac{N_h}{K}\Big\rceil+ \sum_{n=\lceil\frac{N_h}{K}\rceil}^{N_h} \int_{A_{1,2}(0)}\frac{4y_1^2}{|y|^4}\,\ud y+C_2
 \le C_0+ C_1\Big\lceil\frac{N_h}{K}\Big\rceil+N_h\Big(1-\frac{1}{K}\Big)4\pi\log 2+C_2\,,
\end{aligned}
\end{equation*}
{for some universal constants $C_0, C_1, C_2>0$\,.} 
Therefore, by the very definition of $N_h$ and by \eqref{contopereffeuno}, we get
\begin{equation*}
\frac{1}{h^2\log\frac{R}{h}}\F^1_h(A_{h,R}(0))\le \frac{C_1}{K}+ \Big(1-\frac{1}{K}\Big)4\pi+\omega(h)\,,
\end{equation*}
where $\omega(h)\to 0$ as $h\to 0$\,.
Now, sending first $h\to 0$ and then $K\to +\infty$\,, this implies also the inequality ``$\le$'' in \eqref{F1}.
In order to prove \eqref{F2}, we notice that 
\begin{equation*}
\begin{aligned}
\F^2_h(A_{h,R}(0))
=&h^2\int_{0}^{2\pi}\ud \vartheta \sin^4\vartheta \cos^2\vartheta\int_{h}^{R}\frac{\rho^7}{\Big(\big(\rho^2-h\rho\cos\vartheta+\frac{h^2}{4}\big)\big(\rho^2+h\rho\cos\vartheta+\frac{h^2}{4}\big)\Big)^2}\,\ud\rho\\
=& h^2\int_{0}^{2\pi}\ud \vartheta \sin^4\vartheta \cos^2\vartheta\int_{h}^{R}\frac{\rho^7}{\Big(\rho^4-\rho^2\frac{h^2}{2}\cos(2\vartheta)+\frac{h^4}{16}\Big)^2}\,\ud\rho\,,
\end{aligned}
\end{equation*}
so that
\begin{equation}\label{auno}
\begin{split}
&\, \frac{1}{\log\frac{R}{h}}\int_{0}^{2\pi}\ud \vartheta \sin^4\vartheta \cos^2\vartheta\int_{h}^{R}\frac{\rho^7}{\big(\rho^2+\frac{h^2}{4}\big)^4}\,\ud\rho \le \frac{1}{h^2\log\frac{R}{h}}\F^2_h(A_{h,R}(0))\\
\le&\,  \frac{1}{\log\frac{R}{h}}\int_{0}^{2\pi}\ud \vartheta \sin^4\vartheta \cos^2\vartheta\int_{h}^{R}\frac{\rho^7}{\big(\rho^2-\frac{h^2}{4}\big)^4}\,\ud\rho\,.
\end{split}
\end{equation}
By the change of variable $t=\frac{\rho}{h}$, we have that
\begin{equation*}
\begin{aligned}
\int_{h}^R\frac{\rho^7}{\big(\rho^2\mp\frac{h^2}{4}\big)^4}\,\ud\rho=\int_{1}^{\frac{R}{h}}\frac{t^7}{\big(t^2\mp \frac{1}{4}\big)^4}\,\ud t\,,
\end{aligned}
\end{equation*}
and, by de l'H$\hat{\mathrm{o}}$pital's rule, we get
\begin{equation}\label{adue}
\lim_{h\to 0}\frac{1}{\log\frac{R}{h}} \int_{1}^{\frac{R}{h}}\frac{t^7}{\big(t^2\mp \frac{1}{4}\big)^4}\,\ud t=\lim_{N\to +\infty}\frac{N^8}{\big(N^2\mp \frac{1}{4}\big)^4}=1\,. 
\end{equation}
Now, since
\begin{equation}\label{atre}
\int_{0}^{2\pi} \sin^4\vartheta \cos^2\vartheta\,\ud \vartheta=\frac{\pi}{8}\,,
\end{equation} 
in view of \eqref{auno} and \eqref{adue}, we obtain
\begin{equation}\label{afin}
\lim_{h\to 0}\frac{1}{h^2\log\frac{R}{h}}  \F^2_h(A_{h,R}(0))=\frac{\pi}{8}\,,
\end{equation}
i.e., \eqref{F2}.

Finally, we prove that also \eqref{F3} holds true. To this purpose, by using the change of variable $x=2^{n-1}h y$ for every $n=1,\ldots,N_h$\,, we have
\begin{equation*}
\begin{aligned}
&h^2\sum_{n=1}^{N_h-1}\int_{A_{1,2}(0)}\frac{y_2^2\big(y_2^2-y_1^2+{2^{-2n}}\big)^2}{\bigg(\Big(\big(y_1-2^{-n}\big)^2+y_2^2\Big)\Big(\big(y_1 + 2^{-n}\big)^2+y_2^2\Big)\bigg)^2}\,\ud y\le \F_h^3(A_{h,R}(0))\\
\le &h^2\sum_{n=1}^{N_h}\int_{A_{1,2}(0)}\frac{y_2^2\big(y_2^2-y_1^2+{2^{-2n}}\big)^2}{\bigg(\Big(\big(y_1-2^{-n}\big)^2+y_2^2\Big)\Big(\big(y_1 + 2^{-n}\big)^2+y_2^2\Big)\bigg)^2}\,\ud y\,.
\end{aligned}
\end{equation*}
Therefore, by arguing as in the proof of \eqref{F1} we have that there exist two functions $\omega_1,\omega_2$ with $\omega_j(h)\to 0$ as $h\to 0$ and a constant $C>0$ such that
\begin{equation*}
\begin{split}
&\, \frac{1}{\log\frac{R}{h}}\Big(1-\frac{1}{K}\Big)N_h\frac{\pi}{2}\log 2+\omega_1(h) \le \frac{1}{h^2\log\frac{R}{h}}\F^3_h(A_{h,R}(0)) \\
\le&\, \frac{C}{\log\frac{R}{h}}\frac{N_h}{K}+ \frac{1}{\log\frac{R}{h}}\Big(1-\frac{1}{K}\Big)N_h\frac{\pi}{2}\log 2+\omega_2(h)\,,
\end{split}
\end{equation*}
whence \eqref{F3} follows by sending first $h\to 0$ and then $K\to +\infty$.
\vskip5pt
\noindent
Now we show that
\begin{equation}\label{palletta}
\lim_{h\to 0}\frac{1}{h^2\log\frac{R}{h}}\G(\bar v_h;B_h(0))=0\,.
\end{equation} 
By the very definition of $\G$ in \eqref{energyairy} and in view of \eqref{enough}, it is enough to prove that
\begin{equation}\label{claim_core}
\lim_{h\to 0}\frac{1}{h^2\log\frac{R}{h}}\F_h^k(B_h(0))=0\qquad\textrm{for every }k=1,2,3\,.
\end{equation}
Notice that
\begin{equation*}
0\le \F_h^1({B}_{h}(0))\le 2\pi\int_{0}^h\rho\log^2\frac{\big(\rho+\frac{h}{2}\big)^2}{\big(\rho-\frac{h}{2}\big)^2}\,\ud\rho\,,
\end{equation*}
so that using the change of variable $t=\frac{\rho}{h}$, we get
\begin{equation*}
0\le\frac{1}{h^2\log\frac{R}{h}} \F_h^1({B}_{h}(0))\le 2\pi\frac{1}{\log\frac{R}{h}}\int_{0}^1 t\log^2\frac{\big(t+\frac{1}{2}\big)^2}{\big(t-\frac{1}{2}\big)^2}\,\ud t\,,
\end{equation*}
whence the claim \eqref{claim_core} for $k=1$ follows since
$$
\int_{0}^1t\log^2\frac{\big(t+\frac{1}{2}\big)^2}{\big(t-\frac{1}{2}\big)^2}\,\ud t<+\infty\,.
$$
Now we show that 
\begin{equation}\label{claim_f2_core}
\lim_{h\to 0}\frac{1}{h^2}\F^2_h(B_h(0))=0\,.
\end{equation}
To this purpose, we notice that, by the very definition of $\F^2_h$ in \eqref{defF2}, by passing to polar coordinates $(\rho,\vartheta)$ and by using the change of variable $t=\frac{\rho}{h}$, we can write
\begin{equation*}
\frac{1}{h^2}\F^2_h(B_h(0))=\int_{0}^1\ud t\int_{0}^{2\pi}\frac{t^7\sin^4\vartheta\cos^2\vartheta}{\Big(\big(t^2-\frac{1}{4}\big)^2+{t^2}\sin^2\vartheta\Big)^2}\,\ud\vartheta
\le\int_{0}^1\ud t\int_{0}^{2\pi}\frac{t^7\sin^4\vartheta}{\Big(\big(t^2-\frac{1}{4}\big)^2+t^2\sin^2\vartheta\Big)^2}\,\ud\vartheta\,;
\end{equation*}
since the integrand above is $\pi$-periodic and bounded when $\vartheta$ is far away from $0$\,, $\pi$\,, and $2\pi$, in order to obtain \eqref{claim_f2_core} it is enough to show that for $\ep>0$ small enough we have
\begin{equation}\label{claim_f2_core_1}
\int_{0}^1\ud t\int_{0}^{\ep}\frac{t^7\sin^4\vartheta}{\Big(\big(t^2-\frac{1}{4}\big)^2+{t^2}\sin^2\vartheta\Big)^2}\,\ud\vartheta<+\infty\,.
\end{equation}
By a first-order Taylor approximation, the integral above is equivalent to 
\begin{equation*}
\begin{aligned}
&\int_{0}^1\ud t\int_{0}^{\ep}\frac{t^7\vartheta^4}{\Big(\big(t^2-\frac{1}{4}\big)^2+t^2\vartheta^2\Big)^2}\,\ud\vartheta\\
=&\frac 1 2\int_{0}^1\bigg(\ep t^3\Big(\frac{\big(t^2-\frac 1 4\big)^2}{\big(t^2-\frac 1 4)^2+t^2\ep^2}+2\Big)-3t^2\Big(t^2-\frac 1 4\Big)\arctan\frac{\ep t}{t^2-\frac 1 4}\bigg)\,\ud t<+\infty\,,
\end{aligned}
\end{equation*}
which proves \eqref{claim_f2_core_1} and hence \eqref{claim_f2_core}.
Analogously, by the very definition of $\F^3_h$ in \eqref{defF3},  by passing to polar coordinates $(\rho,\vartheta)$ and by using the change of variable $t=\frac{\rho}{h}$\,, we have that
\begin{equation*}
\frac{1}{h^2}\F^3_h(B_h(0))=\int_{0}^1\ud t\int_{0}^{2\pi}\frac{t^3\sin^2\vartheta\big(\frac 1 4-t^2\cos2\vartheta\big)^2}{\Big(\big(t^2-\frac{1}{4}\big)^2+{t^2}\sin^2\vartheta\Big)^2}\,\ud\vartheta <+\infty\,,\\
\end{equation*}
where the boundedness can be proved by arguing as in the proof of \eqref{claim_f2_core_1}. 
Indeed, by a first-order Taylor approximation, the integral above close to $\vartheta=0,\pi,2\pi$ is equivalent to 
$$\int_{0}^1\ud t\int_{0}^{\ep}\frac{t^3\vartheta^2\big(\frac14-t^2+2t^2\vartheta^2\big)^2}{\Big(\big(t^2-\frac{1}{4}\big)^2+t^2\vartheta^2\Big)^2}\,\ud\vartheta\,,$$
which can be proved to be finite by a straightforward computation.
%
This proves 
\eqref{claim_core} also for $k=3$\,, so that, by \eqref{anello} and \eqref{palletta}, the proof is concluded.
\end{proof}
\section{Proof of Lemma \ref{lm:mindis}}\label{prooflm:mindis}
This section is devoted to the proof of Lemma \ref{lm:mindis}.
\begin{proof}[Proof of Lemma \ref{lm:mindis}]
\blu{With some abuse of notation we set $W^s_{0,\ce}\coloneqq W^{se_2\de_0}_{0,\ce}$\,, where $W^{se_2\de_0}_{0,\ce}$ is  defined in \eqref{2201181926}.}
We preliminarily show that $W^s_{0,\ce}\in \Bb_{\ce,R}$\,.
To this purpose, we first notice that
\begin{equation*}
\alpha_\ce+\frac{\beta_\ce}{R^2}+\gamma_\ce R^2+2\log R^2=2\frac{R^2-\ce^2}{R^2+\ce^2}-2\log R^2+2\frac{\ce^2}{R^2+\ce^2}-2\frac{R^2}{R^2+\ce^2}+2\log R^2=0\,,
\end{equation*}
and hence
\begin{equation}\label{=0bordoR}
W^s_{0,\ce}=0\qquad\textrm{on }\partial B_R(0)\,.
\end{equation}
We define the function $\widehat W_{0,\ce}\colon A_{\ce,R}(0)\to\R$ as 
\begin{equation}\label{2206060944}
\widehat W_{0,\ce}(x)\coloneqq \bigg(\alpha_\ce +\beta_\ce\frac{1}{|x|^2}+\gamma_\ce|x|^2+2\log|x|^2\bigg) x_1 
\,,
\end{equation}
and we notice that 
\begin{equation}\label{2206060945}
W_{0,\ce}^s\equiv \frac{s}{16\pi}\frac{E}{1-\nu^2} \widehat W_{0,\ce}\qquad\text{in $A_{\ce,R}(0)$\,.}
\end{equation}
For every $x\in  A_{\ep,R}(0)$
\begin{equation*}
\nabla \widehat W_{0,\ce}(x)=\begin{pmatrix}
\displaystyle \alpha_\ce+\beta_\ce\frac{x_2^2-x_1^2}{|x|^4}+\gamma_\ce(|x|^2+2x_1^2)+2\log|x|^2+4\frac{x_1^2}{|x|^2}\\[3mm]
\displaystyle -2\beta_\ce\frac{x_1x_2}{|x|^4}+2\gamma_\ce x_1x_2+4\frac{x_1x_2}{|x|^2}
\end{pmatrix}
\end{equation*}
whence, for $x\in\partial B_R(0)$, we deduce that
\begin{equation*}
\begin{aligned}
\partial_n \widehat W_{0,\ce}(x)=&\,\frac{x_1}{R}\Big(\alpha_\ce+\beta_\ce\frac{x_2^2-x_1^2}{|x|^4}+\gamma_\ce(|x|^2+2x_1^2)+2\log|x|^2+4\frac{x_1^2}{|x|^2}\Big)\\
&\,+\frac{x_2}{R}\Big(-2\beta_\ce\frac{x_1x_2}{|x|^4}+2\gamma_\ce x_1x_2+4\frac{x_1x_2}{|x|^2}\Big)\\
=&\,\frac{x_1}{R}\Big(\alpha_\ce-\frac{\beta_\ce}{|x|^2}+3\gamma_\ce|x|^2+2\log|x|^2+4\Big)\\
=&\,\frac{x_1}{R}\Big(\alpha_\ce-\frac{\beta_\ce}{R^2}+3\gamma_\ce R^2+2\log R^2+4\Big)\\
=&\,\frac{x_1}{R}\Big(2\frac{R^2-\ce^2}{R^2+\ce^2}-2\frac{\ce^2}{R^2+\ce^2}-6\frac{R^2}{R^2+\ce^2}+4\Big)=0
\end{aligned}
\end{equation*}
and, consequently,
\begin{equation}\label{norm=0bordoR}
\partial_n W^s_{0,\ce}=0\qquad\textrm{on }\partial B_R(0);
\end{equation}
Moreover, it is immediate to see that $W^s_{0,\ce}\in C^0(B_R(0))$\,. 
Furthermore, the inner and outer traces of $\partial_t   W^s_{0,\ce}$ at $\partial B_{\ep}(0)$ are continuous so that $W^s_{0,\ce}\in H^1(B_R(0))$. 
Therefore, in order to check that  $W^s_{0,\ce}\in H^2(B_R(0))$ it is enough to show that
\begin{equation}\label{2206031751}
\partial_t \nabla {\widehat W_{0,\ep}}=0\qquad \textrm{on } \partial B_{\ep}(0),
\end{equation}
{where $t$ is the tangent vector to $\partial B_\ce(0)$.}
To this end, we observe
\begin{equation*}
\begin{split}
&\, \nabla^2 \widehat W_{0,\ce}(x) \\
=&\,
\begin{pmatrix}
\displaystyle -2\beta_\ce\frac{x_1(3x_2^2-x_1^2)}{|x|^6}+6 \gamma_\ce x_1+4x_1\frac{x_1^2+3x_2^2}{|x|^4}&\displaystyle 2\beta_\ce\frac{x_2(3x_1^2-x_2^2)}{|x|^6}+2\gamma_\ce x_2+4x_2\frac{x_2^2-x_1^2}{|x|^4}\\[3mm]
\displaystyle 2\beta_\ce\frac{x_2(3x_1^2-x_2^2)}{|x|^6}+2\gamma_\ce x_2+4x_2\frac{x_2^2-x_1^2}{|x|^4}
&\displaystyle -2\beta_\ce\frac{x_1(x_1^2-3x_2^2)}{|x|^6}+2\gamma_\ce x_1+4x_1\frac{x_1^2-x_2^2}{|x|^4}
\end{pmatrix}
\end{split}
\end{equation*}
so that, for $x\in\partial B_\ce(0)$\,, we have
\begin{equation}\label{hesstang01}
\begin{aligned}
\partial_{x_1}\nabla \widehat W_{0,\ce}\cdot t=&\,-\frac{x_2}{\ce}\Big(-2\beta_\ce\frac{x_1(3x_2^2-x_1^2)}{|x|^6}+6 \gamma_\ce x_1+4x_1\frac{|x|^2+2x_2^2}{|x|^4}\Big)\\
&\,+\frac{x_1}{\ce} \Big(2\beta_\ce\frac{x_2(3x_1^2-x_2^2)}{|x|^6}+2\gamma_\ce x_2 +4x_2\frac{x_2^2-x_1^2}{|x|^4}
\Big)\\
=&\,\frac{1}{\ce}x_1x_2\Big(4\frac{\beta_\ce}{|x|^4}-4\gamma_\ce-\frac{8}{|x|^2}\Big)
=\frac{1}{\ce}x_1x_2\Big(4\frac{\beta_\ce}{\ce^4}-4\gamma_\ce-\frac{8}{\ce^2}\Big)\\
=&\,\frac{8}{\ce}x_1x_2\Big(\frac{R^2}{\ce^2(R^2+\ce^2)}+\frac{1}{R^2+\ce^2}-\frac{1}{\ce^2}\Big)=0
\end{aligned}
\end{equation}
and, analogously,
\begin{equation}\label{hesstang02}
\begin{aligned}
\partial_{x_2}\nabla \widehat W_{0,\ce}\cdot t=&\frac{2}{\ce}(x_2^2-x_1^2)\Big(\frac{\beta_\ce}{\ce^4}-\gamma_\ce-\frac{2}{\ce^2}\Big)=0\,.
\end{aligned}
\end{equation}
Finally, by \eqref{=0bordoR}, \eqref{norm=0bordoR}, \eqref{hesstang01}, \eqref{hesstang02}, and using that $W^s_{0,\ce}\in C^4(A_{\ce,R}(0))$\,, we deduce that $W_{0,\ce}^s\in \Bb_{\ce,R}$\,.

Now we prove that  $W_{0,\ce}^s$ is the minimizer of \blu{$\I_{0,\ce}^{se_2\de_0}$} in $\Bb_{\ce,R}$\,.
In view of \eqref{defJ0eff}, for every $\phi\in\Bb_{\ce,R}$\,,   $W_{0,\ce}^s$ must satisfy
\begin{equation}\label{el}
\begin{aligned}
0=&\,\frac{\ud}{\ud t}{\Big|_{t=0}}\blu{\I^{se_2\de_0}_{0,\ce}}(W_{0,\ce}^s+t\phi)\\ 
=&\,\frac{1+\nu}{E}\Big(\int_{A_{\ce,R}(0)}\nabla^2W_{0,\ce}^s:\nabla^2\phi\,\ud x-\nu\int_{A_{\ce,R}(0)}\Delta W_{0,\ce}^s\,\Delta\phi\,\ud x\Big)+\frac{s}{2\pi\ce}\int_{\partial B_\ce(0)}\partial_{x_1}\phi\,\ud\Huno\\
=&\,\frac{1-\nu^2}{E}\int_{A_{\ce,R}(0)}\Delta^2W_{0,\ce}^s\,\phi\,\ud x\\
&\,+\frac{1-\nu^2}{E}\int_{\partial B_\ce(0)}\partial_{n}\Delta W^s_{0,\ce}\phi\,\ud\Huno+\frac{1+\nu}{E}\int_{\partial B_\ce(0)}\big(\nu\Delta W^s_{0,\ce}-(\nabla^2W^s_{0,\ce})_{nn}\big)\partial_{n}\phi\,\ud\Huno\\ \nonumber
&\,
+\int_{\partial B_\ce(0)}\frac{s}{2\pi\ce}\partial_{x_1}\phi\,\ud\Huno\,,
\end{aligned}
\end{equation}
where we have used that
$\phi=\partial_n\phi=0$ on $\partial B_R(0)$ and integration by parts to get
\begin{eqnarray*}
\int_{A_{\ce,R}(0)}\nabla^2W^s_{0,\ce}:\nabla^2\phi\,\ud x&\!\!\!\!=&\!\!\!\! \int_{A_{\ce,R}(0)}\Delta^2W^s_{0,\ce}\,\phi\,\ud x+
\int_{\partial A_{\ce,R}(0)}\langle\nabla^2 W^s_{0,\ce} n,\nabla\phi\rangle\,\ud\Huno\\
&&\!\!\!\! -\int_{\partial A_{\ce,R}(0)}\partial_n(\Delta W^s_{0,\ce})\phi\,\ud\Huno\\
&\!\!\!\!=&\!\!\!\! \int_{A_{\ce,R}(0)}\Delta^2W^s_{0,\ce}\,\phi\,\ud x+\int_{\partial B_\ce(0)}\big(\phi\,\partial_n(\Delta W^s_{0,\ce})-\langle\nabla^2 W_{0,\ce}^s n,\nabla\phi\rangle\big)\,\ud\Huno\,,
\\
\int_{A_{\ce,R}(0)}\Delta W^s_{0,\ce}\,\Delta\phi\,\ud x&=&\int_{A_{\ce,R}(0)}\Delta^2W^s_{0,\ce}\,\phi\,\ud x+\int_{\partial A_{\ce,R}(0)}\Delta W^s_{0,\ce}\,\partial_n\phi\,\ud\Huno\\
&&\!\!\!\! -\int_{\partial A_{\ce,R}(0)}\phi\,\partial_n(\Delta W^s_{0,\ce})\,\ud\Huno\\
&\!\!\!\!=&\!\!\!\! \int_{A_{\ce,R}(0)}\Delta^2W^s_{0,\ce}\,\phi\,\ud x+\int_{\partial B_\ce(0)} \big(\phi\,\partial_n(\Delta W^s_{0,\ce})-\Delta W^s_{0,\ce}\,\partial_n\phi\big)\,\ud\Huno\,.
\end{eqnarray*}
 Therefore, proving the minimality of $W_{0,\ce}^s$ is equivalent to showing that
\begin{eqnarray}\label{bihar}
&&\Delta^2 W^s_{0,\ce}=0\qquad\textrm{in }A_{\ce,R}(0)\,,\\ \label{natural}
&&\frac{1-\nu^2}{E}\int_{\partial B_\ce(0)}\partial_{n}\Delta W^s_{0,\ce}\phi\,\ud\Huno+\frac{1+\nu}{E}\int_{\partial B_\ce(0)}\big(\nu\Delta W^s_{0,\ce}-(\nabla^2W^s_{0,\ce})_{nn}\big)\partial_{n}\phi\,\ud\Huno\\ \nonumber
&&\phantom{\frac{1-\nu^2}{E}}+\int_{\partial B_\ce(0)}\frac{s}{2\pi\ce}\partial_{x_1}\phi\,\ud\Huno=0\qquad\textrm{ for every }\phi\in \Bb_{\ce,R}\,.
\end{eqnarray}
By \eqref{2206060945} and the very definition of $\widehat W_{0,\ce}$ in \eqref{2206060944}, the biharmonicity in \eqref{bihar} follows by a direct computation,
so that we are left with proving \eqref{natural}.
To this purpose, we notice that
\begin{eqnarray*}
\Delta W^s_{0,\ce}&\!\!\!\!=&\!\!\!\! \frac{s}{16\pi}\frac{E}{1-\nu^2}\Delta \widehat W_{0,\ce}=\frac{s}{2\pi}\frac{E}{1-\nu^2}x_1\Big(\gamma_\ce+\frac{1}{|x|^2}\Big)\,,\\
\partial_{n}\Delta W^s_{0,\ce}&\!\!\!\! =&\!\!\!\! \frac{s}{2\pi}\frac{E}{1-\nu^2}\frac{x_1}{|x|}\Big(\gamma_\ce-\frac{1}{|x|^2}\Big)\,,\\
(\nabla^2W^s_{0,\ce})_{nn}&\!\!\!\!=&\!\!\!\! \frac{s}{8\pi}\frac{E}{1-\nu^2} x_1\Big(\frac{\beta_\ce}{|x|^4}+3\gamma_\ce +\frac{2}{|x|^2}\Big)\,,
\end{eqnarray*}
whence we deduce, recalling \eqref{abc},
\begin{eqnarray}\label{lapla}
\Delta W^s_{0,\ce}\Big|_{\partial B_\ce(0)}&\!\!\!\!=&\!\!\!\! \frac{s}{2\pi}\frac{E}{1-\nu^2}x_1\Big(\gamma_\ce+\frac{1}{\ce^2}\Big)= \frac{s}{2\pi}\frac{E}{1-\nu^2}\frac{R^2-\ce^2}{\ce^2(R^2+\ce^2)}x_1\,,\\ \label{dernormlapla}
\partial_{n}\Delta W^s_{0,\ce}\Big|_{\partial B_\ce(0)}&\!\!\!\!=&\!\!\!\! \frac{s}{2\pi}\frac{E}{1-\nu^2}\frac{x_1}{\ce}\Big(\gamma_\ce-\frac{1}{\ce^2}\Big)=-\frac{s}{2\pi}\frac{E}{1-\nu^2}\frac{R^2+3\ce^2}{\ce^2(R^2+\ce^2)}\frac{x_1}{\ce}
\,,\\ \label{hessnormnorm}
(\nabla^2W^s_{0,\ce})_{nn}\Big|_{\partial B_\ce(0)}&\!\!\!\!=&\!\!\!\! \frac{s}{8\pi}\frac{E}{1-\nu^2} x_1\Big(\frac{\beta_\ce}{\ce^4}+3\gamma_\ce +\frac{2}{\ce^2}\Big)=\frac{s}{2\pi}\frac{E}{1-\nu^2}\frac{R^2-\ce^2}{\ce^2(R^2+\ce^2)}x_1\,.
\end{eqnarray}
Furthermore, every $\phi\in {\Bb}_{\ce,R}$ 
satisfies $\phi(x)=a_\phi+b_\phi x_1+c_\phi x_2$ for every $x\in\partial B_\ce(0)$, for some $a_\phi, b_\phi, c_\phi\in\R$\,, so that the equation in \eqref{natural} can be rewritten as
\begin{equation}\label{newnat}
\begin{cases}
\displaystyle \frac{1-\nu^2}{E}\int_{\partial B_\ce(0)}\partial_{n}\Delta W^s_{0,\ce}\,\ud\mathcal H^1=0\\[3mm]
\displaystyle \int_{\partial B_\ce(0)}x_1\bigg(\frac{1-\nu^2}{E}\partial_{n}\Delta W^s_{0,\ce}+\frac{1+\nu}{E}\frac{1}{\ce}\Big(\nu\Delta W^s_{0,\ce}-\big(\nabla^2W^s_{0,\ce}\big)_{nn}\Big)\bigg)\,\ud\mathcal{H}^1=-s\\[3mm]
\displaystyle \int_{\partial B_\ce(0)}x_2 \bigg(\frac{1-\nu^2}{E}\partial_{n}\Delta W^s_{0,\ce}+\frac{1+\nu}{E}\frac{1}{\ce}\Big(\nu\Delta W^s_{0,\ce}-\big(\nabla^2W^s_{0,\ce}\big)_{nn}\Big)\bigg)\,\ud\mathcal{H}^1=0\,,
\end{cases}
\end{equation}
which follow by straightforward computations from
\eqref{lapla}, \eqref{dernormlapla}, and \eqref{hessnormnorm}.

Now we compute $\blu{\I^{\alpha}_{0,\ce}}(W^s_{0,\ce})$\,. As for the second summand in the right-hand side of \eqref{defJ0}, we have
\begin{equation}\label{calcolocarico}
\frac{s}{2\pi\ce}\int_{\partial B_\ce(0)}\partial_{x_1}W^s_{0,\ce}\,\ud\mathcal{H}^1
=
-\frac{s^2}{4\pi}\frac{E}{1-\nu^2}\bigg(
\log\frac{R}{\ce}-\frac{R^2-\ce^2}{R^2+\ce^2}\bigg)\,.
\end{equation}
Moreover,
\begin{equation*}
\begin{aligned}
\int_{A_{\ce,R}(0)} \!\!\! |\Delta W^s_{0,\ce}|^2\,\ud x=&\,\frac{s^2}{4\pi^2}\frac{E^2}{(1-\nu^2)^2}\int_{A_{\ce,R}(0)}\!\!\!x_1^2\Big(\gamma_\ce+\frac{1}{|x|^2}\Big)^2\,\ud x
=\frac{s^2}{4\pi}\frac{E^2}{(1-\nu^2)^2}\int_{\ce}^R \! \rho^3\bigg(\gamma_\ce+\frac{1}{\rho^2}\Big)^2\,\ud\rho\\
=&\frac{s^2}{4\pi}\frac{E^2}{(1-\nu^2)^2}\Big(\gamma_\ce^2\frac{R^4-\ce^4}{4}+\gamma_\ce(R^2-\ce^2)+\log\frac{R}{\ce}
\bigg)\\
=&\frac{s^2}{4\pi}\frac{E^2}{(1-\nu^2)^2}\bigg(
\log\frac{R}{\ce}-\frac{R^2-\ce^2}{R^2+\ce^2}\bigg)
\end{aligned}
\end{equation*}
and
\begin{equation*}
\begin{aligned}
|\nabla^2\widehat W_{0,\ce}|^2=&\, |\Delta \widehat W_{0,\ce}|^2-2\partial^2_{x_1^2}\widehat W_{0,\ce}\partial^2_{x_2^2}\widehat W_{0,\ce}+2|\partial^2_{x_1x_2}\widehat W_{0,\ce}|^2\\
=&\, 64x_1^2\Big(\gamma_\ce+\frac{1}{|x|^2}\Big)^2+\frac{8}{|x|^6}\beta_\ce^2
-16\gamma_\ce^2x_1^2-\frac{32}{|x|^6}\beta_\ce x_2^2-\frac{32}{|x|^2}\gamma_\ce x_1^2\\
&\, +8(x_1^2-x_2^2)\bigg(-\gamma_\ce^2+2\frac{\beta_\ce\gamma_\ce}{|x|^4}-\frac{4}{|x|^4}-\frac{4}{|x|^2}\gamma_\ce\bigg)\,,\\
\end{aligned}
\end{equation*}
so that, recalling \eqref{abc}, 
\begin{equation*}
\begin{aligned}
\int_{A_{\ce,R}(0)}|\nabla^2W^s_{0,\ce}|^2\,\ud x=&\, 
\frac{s^2}{4\pi}\frac{E^2}{(1-\nu^2)^2}\Big(\log\frac{R}{\ce}-\frac{R^2-\ce^2}{R^2+\ce^2}\Big)\,. 
\end{aligned}
\end{equation*}
%
It follows that
\begin{equation}\label{2201201658}
\begin{aligned}
\mathcal{G}(W^s_{0,\ce};A_{\ce,R}(0))=&\, 
\frac{s^2}{8\pi}\frac{E}{1-\nu^2}\Big(\log\frac{R}{\ce}-\frac{R^2-\ce^2}{R^2+\ce^2}\Big)
\,,
\end{aligned}
\end{equation}
and hence, in view of \eqref{calcolocarico},
$$
\blu{\I_{0,\ce}^{se_2\de_0}}(W^s_{0,\ce})=-\frac{s^2}{8\pi}\frac{E}{1-\nu^2}\Big(\log\frac{R}{\ce}-\frac{R^2-\ce^2}{R^2+\ce^2}\Big)
$$
i.e., \eqref{valmin}.
\end{proof}
The next result follows from the proof of Lemma~\ref{lm:mindis} by straightforward computations.
\begin{corollary}\label{insec4}
Let $s\in\R\setminus\{0\}$\,, $0<\ce<R$ and let {$W_{0,\ce}^{se_2\de_0}$} be the function defined in \eqref{2201181926}.
Then, for every $\ce<r\le R$\,,
\begin{equation*}
\begin{aligned}
\G(W_{0,\ce}^{se_2\de_0}; A_{\ce,r}(0))=&\frac{s^2}{8\pi}\frac{E}{1-\nu^2}\log\frac r \ce+\frac{s^2}{8\pi}\frac{E}{1-\nu^2}\frac{r^2-\ce^2}{R^2+\ce^2}\Big(\frac{r^2+\ce^2}{R^2+\ce^2}-2\Big)\\
&+\frac{s^2}{32\pi}\frac{E}{(1-\nu)^2(1+\nu)}\frac{r^2-\ce^2}{R^2+\ce^2}\Big(\frac{R^2}{r^2}-1\Big)\Big(\frac{r^2+\ce^2}{R^2+\ce^2}\Big(\frac{R^2}{r^2}+1\Big)-2\Big)\,,
\end{aligned}
\end{equation*}
and hence
\begin{equation*}
\G(W_{0,\ce}^{se_2\de_0}; A_{\ce,r}(0))+\frac{s}{2\pi\ce}\int_{\partial B_\ce(0)}
\partial_{x_1}W^{se_2\de_0}_{0,\ce}\,\ud\mathcal{H}^1
=-\frac{s^2}{8\pi}\frac{E}{1-\nu^2}\log\frac{1}{\ce}+\frac{s^2}{8\pi}\frac{E}{1-\nu^2}\log r+f_\ce(r,R;s)\,,
\end{equation*}
where 
\begin{equation}\label{vanerr}
\begin{aligned}
f_\ce(r,R;s)\coloneqq&\,\frac{s^2}{8\pi}\frac{E}{1-\nu^2}\Big(2\frac{R^2-\ce^2}{R^2+\ce^2}+\frac{r^2-\ce^2}{R^2+\ce^2}\Big(\frac{r^2+\ce^2}{R^2+\ce^2}-2\Big)-2\log R\Big)\\
&\,+\frac{s^2}{32\pi}\frac{E}{(1-\nu)^2(1+\nu)}\frac{r^2-\ce^2}{R^2+\ce^2}\Big(\frac{R^2}{r^2}-1\Big)\Big(\frac{r^2+\ce^2}{R^2+\ce^2}\Big(\frac{R^2}{r^2}+1\Big)-2\Big)\,.
\end{aligned}
\end{equation}
\end{corollary}

\section{$\ce$-independent  integral inequalities for $H^2$ functions}
Let $\alpha=\sum_{j=1}^Jb^j\de_{x^j}\in\ED(\Omega)$\,.
Here we prove a Poincar\'{e}-type inequality for functions in $H^2(\Omega_\ce(\alpha))$ where the Poincar\'{e} constant is shown to be independent of $\ce$\,.
The proof is obtained by combining the ``$\varepsilon$-independent Poincar\'{e} inequality'' contained in \cite[Proposition~A.1]{CermelliLeoni06} and the generalized Poincar\'{e} inequality contained in \cite[Theorem~6.1-8(b)]{ciarlet88}, which we recall here.
\begin{proposition}[{\cite[Theorem~6.1-8(b)]{ciarlet88}}]\label{20220527}
Let $\Omega'\subset\Omega$ {be a connected and open set} 
and let $\Gamma_0\subseteq\partial\Omega'$ be a portion of the boundary with $\mathcal{H}^1(\Gamma_0)>0$.
Then there exists a constant $C(\Omega')>0$ depending only on $\Omega'$ such that for every function $u\in H^1(\Omega)$ it holds
\begin{equation}\label{202112230042}
\int_{\Omega'} |u(x)|^2\,\ud x\leq C(\Omega')\bigg(\int_{\Omega'} |\nabla u(x)|^2\,\ud x+\bigg|\int_{\Gamma_0} u(x)\,\ud \Huno(x)\bigg|^2\bigg).
\end{equation}
\end{proposition}
Take now $\Omega'\subset\Omega$ such that $\partial\Omega'\supset\partial\Omega$ and $\Gamma_0=\partial\Omega$ in Proposition \ref{20220527}. 
Moreover, let $f$ be a function which is smooth in a neighborhood of $\partial\Omega$. 
Then for every $u\in H^2(\Omega)$ with $u=f$ and $\partial_nu=\partial_n f$ on $\partial\Omega$\,, thanks to Jensen's inequality, formula \eqref{202112230042} reads
\begin{equation}\label{202112230047}
\int_{\Omega'} |u(x)|^2\,\ud x\leq C(\Omega')\int_{\Omega'} |\nabla u(x)|^2\,\ud x+C(\Omega',\partial\Omega)\int_{\partial\Omega} |f(x)|^2\,\ud \Huno(x)\,;
\end{equation}
analogously, noticing that $\nabla (u-f)=0$ on $\partial\Omega$\,, by applying \eqref{202112230042} to $\partial_{x_1} u$ and $\partial_{x_2} u$\,, we obtain  
\begin{equation}\label{202112230050}
\int_{\Omega'} |\nabla u(x)|^2\,\ud x\leq 2C(\Omega')\int_{\Omega'} |\nabla^2 u(x)|^2\,\ud x+C(\Omega',\partial\Omega)\int_{\partial\Omega} |\nabla f(x)|^2\,\ud \Huno(x)\,.
\end{equation}

%
\begin{proposition}[$\ce$-independent Poincar\'{e} inequality]\label{202112230052}
Let $\alpha\in \ED(\Omega)$ and let $0<\ep<\frac D 2$ with $D$ defined in \eqref{distaminima}\,. 
Then there exists a constant $C_1(\Omega,\alpha)>0$ depending only on $\Omega$ and on $\supp\alpha$\,, and  independent of $\ce$\,, such that the following holds true.
For every function $f$ which is smooth in a neighborhood of $\partial\Omega$ and for every $u\in H^2(\Omega)$ with $u=f$ and $\partial_nu=\partial_n f$ on $\partial\Omega$
\begin{multline}\label{202112230053}
\int_{\Omega_\ce(\alpha)} |u(x)|^2\,\ud x+\int_{\Omega_\ce(\alpha)} |\nabla u(x)|^2\,\ud x
\le C_1(\Omega,\alpha) \bigg(\int_{\Omega_\ce(\alpha)} |\nabla^2 u(x)|^2\,\ud x\\
{+\|f\|^2_{L^\infty(\partial\Omega)}+\|\nabla f\|^2_{L^\infty(\partial\Omega)}}
\bigg)\,.
\end{multline}
\end{proposition}
\begin{proof}
The proof follows \cite[Proposition~A.1]{CermelliLeoni06}. 
We recall  here the main lines of the proof for the reader's convenience. 

Let $j=1,\ldots,J$ be fixed and let  the pair $(r;\vartheta)$ denote the polar coordinates centered at $x^j$\,. 
Let $\ce\leq s \leq D/2 \leq\rho<D$ (with $D$ defined in \eqref{distaminima}) and let $\vartheta\in[0,2\pi]$. By the Fundamental Theorem of Calculus, we can write
$$
u(s,\vartheta)=u(\rho,\vartheta)-\int_s^\rho \frac{\partial u}{\partial r}(r,\vartheta)\,\ud r\,,
$$
so that (by recalling that $(a-b)^2\leq 2a^2+2b^2$)
$$
|u(s,\vartheta)|^2\leq2|u(\rho,\vartheta)|^2+2\bigg|\int_s^\rho \frac{\partial u}{\partial r}(r,\vartheta)\,\ud r\bigg|^2
$$
and in turn, by Jensen's inequality, 
$$|u(s,\vartheta)|^2\leq
2|u(\rho,\vartheta)|^2+2(\rho-s)\int_s^\rho \bigg|\frac{\partial u}{\partial r}(r,\vartheta)\bigg|^2\,\ud r\leq
2|u(\rho,\vartheta)|^2+2D\int_s^D \bigg|\frac{\partial u}{\partial r}(r,\vartheta)\bigg|^2\,\ud r\,.
$$
We now multiply by $s$ and integrate with respect to $\vartheta$ to obtain
\begin{equation}\label{202112271025}
\begin{split}
\int_0^{2\pi} s|u(s,\vartheta)|^2\,\ud\vartheta\leq &\, 2s\int_0^{2\pi}|u(\rho,\vartheta)|^2\,\ud\vartheta+2D\int_0^{2\pi}\int_s^D \bigg|\frac{\partial u}{\partial r}(r,\vartheta)\bigg|^2\,s\,\ud r\,\ud\vartheta \\
\leq&\, 2 \int_0^{2\pi}|u(\rho,\vartheta)|^2\,\rho\,\ud\vartheta+2D\int_0^{2\pi}\int_\ce^D |\nabla u(r,\vartheta)|^2\,r\,\ud r\,\ud \vartheta \\
=&\, 2 \int_0^{2\pi}|u(\rho,\vartheta)|^2\,\rho\,\ud\vartheta+2D\int_{A_{\ce,D}(x^j)} |\nabla u(x)|^2\,\ud x\,.
\end{split}
\end{equation}
We now integrate with respect to $s$ in $\big[\ce,\frac D 2\big]$ (notice that the right-hand side does not depend on~$s$) to get 
$$\int_\ce^{\frac D 2}\int_0^{2\pi} |u(s,\vartheta)|^2\,s\,\ud s\,\ud\vartheta\leq 2\Big(\frac{D}2-\ce\Big)\bigg(\int_0^{2\pi}|u(\rho,\vartheta)|^2\,\rho\ud\vartheta+D\int_{A_{\ce,D}(x^j)} |\nabla u(x)|^2\,\ud x\bigg)\,,$$
whence
$$\int_{A_{\ce,\frac D 2}(x^j)}|u(x)|^2\,\ud x\leq D\int_0^{2\pi}|u(\rho,\vartheta)|^2\,\rho\,\ud\vartheta+D^2\int_{A_{\ce,D}(x^j)} |\nabla u(x)|^2\,\ud x\,;$$
an integration with respect to $\rho$ in $\big[\frac D 2,D\big]$ now yields
\[\begin{split}
\int_{A_{\ce,\frac D 2}(x^j)}|u(x)|^2\,\ud x\leq&\, 2\int_{A_{\frac D 2,D}(x^j)}|u(x)|^2\,\ud x+D^2\int_{A_{\ce,D}(x^j)} |\nabla u(x)|^2\,\mathrm{d}x \\
\leq&\, 2\int_{\Omega_{\frac D 2}(\alpha)}|u(x)|^2\,\ud x+D^2\int_{A_{\ce,D}(x^j)} |\nabla u(x)|^2\,\ud x \\
\leq&\, 2C\big(\Omega_{\frac D 2}(\alpha)\big)\int_{\Omega_{\frac D 2}(\alpha)}|\nabla u(x)|^2\,\ud x+2C\big(\Omega_{\frac D 2}(\alpha)\big)\bigg|\int_{\partial\Omega} f(x)\,\ud \Huno(x)\bigg|^2\\
&+D^2\int_{A_{\ce,D}(x^j)} |\nabla u(x)|^2\,\ud x\,,
\end{split}\]
where we have used \eqref{202112230047} in the last inequality.
Therefore, by using \eqref{202112230047} again, we have
\begin{equation}\label{202112271051}
\begin{split}
\int_{\Omega_\ce(\alpha)} |u(x)|^2\,\ud x=&\, \sum_{j=1}^J \int_{A_{\ce,\frac D 2}(x^j)} |u(x)|^2\,\ud x+\int_{\Omega_{\frac D 2}(\alpha)} |u(x)|^2\,\ud x \\
\leq&\, (2J+1)C\big(\Omega_{\frac D 2}(\alpha)\big) \int_{\Omega_{\frac D 2}(\alpha)} |\nabla u(x)|^2\,\ud x+D^2 \sum_{j=1}^J \int_{A_{\ce,D}(x^j)}  |\nabla u(x)|^2\,\ud x \\
&\, +2(J+1)C\big(\Omega_{\frac D 2}(\alpha),\partial\Omega\big) \int_{\partial\Omega} |f(x)|^2\,\ud \Huno(x)\\
\leq&\, (2J+1)C_{\Omega_{\frac D 2}(\alpha)} \int_{\Omega_\ce(\alpha)} |\nabla u(x)|^2\,\ud x+JD^2 \int_{\Omega_\ce(\alpha)}  |\nabla u(x)|^2\,\ud x \\
&\, +2(J+1)C\big(\Omega_{\frac D 2}(\alpha),\partial\Omega\big) \int_{\partial\Omega} |f(x)|^2\,\ud \Huno(x)\\
\leq&\, \max\big\{(2J+1)C\big(\Omega_{\frac D 2}(\alpha)\big),JD^2\big\} \int_{\Omega_\ce(\alpha)} |\nabla u(x)|^2\ud x\\
&\,+2(J+1)C\big(\Omega_{\frac D 2}(\alpha),\partial\Omega\big) \int_{\partial\Omega} |f(x)|^2\,\ud \Huno(x)\\
\le&\, \widetilde C_1(\Omega,\alpha) \bigg(\int_{\Omega_\ce(\alpha)} |\nabla u(x)|^2\,\ud x+
\|f\|^2_{L^\infty(\partial\Omega)}\bigg)\,,
\end{split}
\end{equation}
where we have set $\widetilde C_1(\Omega,\alpha)\coloneqq\max\big\{(2J+1)C\big(\Omega_{\frac D 2}(\alpha)\big),JD^2\big\}+2(J+1)C\big(\Omega_{\frac D 2}(\alpha)\big)$\,.
By repeating the same reasoning for $\partial_{x_1} u$ and $\partial_{x_2} u$ and by using \eqref{202112230050} in place of \eqref{202112230047}, we obtain
\begin{equation*}
\begin{aligned}
\int_{\Omega_\ce(\alpha)} |\nabla u(x)|^2\,\ud x\le& 2 
\widetilde C^2_1(\Omega,\alpha) \bigg(\int_{\Omega_\ce(\alpha)} |\nabla^2 u(x)|^2\,\ud x+
\|\nabla f\|^2_{L^\infty(\partial\Omega)}
\bigg)\,,
\end{aligned}
\end{equation*}
and the proposition is proved with $C_1(\Omega,\alpha)\coloneqq 3\widetilde C_1^2(\Omega,\alpha)$\,.  
\end{proof}

\begin{proposition}[$\ce$-independent trace inequality]\label{PropA6CL}
Let $\alpha\in \ED(\Omega)$ and let $\ep>0$ satisfy \eqref{distaminima}\,. Then, there exists a constant $C_2(\Omega,\alpha)
>0$ depending only on $\Omega$ and on $\supp\alpha$\,, and independent of $\ce$\,, such that,
for every function $f$ which is smooth in a neighborhood of $\partial\Omega$ and for every $u\in H^2(\Omega)$ with $u=f$ and $\partial_nu=\partial_n f$ on $\partial\Omega$\,,
the following fact holds true:
\begin{equation*}
\int_{\partial\Omega_\ce(\alpha)} |u(x)|^2\,\ud\Huno(x)+\int_{\partial\Omega_\ce(\alpha)} |\nabla u(x)|^2\,\ud\Huno(x) \leq C_2(\Omega,\alpha) \Big(\int_{\Omega_\ce(\alpha)} |\nabla^2 u(x)|^2\,\ud x+\|f\|^2_{C^\infty(\partial\Omega)}\Big)\,.
\end{equation*}
\end{proposition}
\begin{proof}
By \cite[Proposition~A.6]{CermelliLeoni06}, there exists a constant $C(\Omega,\alpha)$ depending only on $\Omega$ and on $\supp\alpha$ such that, for any function $v\in H^1(\Omega)$\,, there holds
\begin{equation}\label{propA6diCL}
\int_{\partial\Omega_\ce(\alpha)}|v|^2\,\ud x\le C(\Omega,\alpha)\Big(\int_{\Omega_\ce(\alpha)}|v|^2\,\ud x+\int_{\Omega_\ce(\alpha)}|\nabla v|^2\,\ud x\Big)\,. 
\end{equation}
We conclude by applying \eqref{propA6diCL} with $v=u$, $v=\partial_{x_1}u$, and $v=\partial_{x_2}u$ and using \eqref{202112230053}. 
\end{proof}
\begin{proposition}\label{2201202252}
Let $\alpha=\sum_{j=1}^Jb^j\de_{x^j}\in \ED(\Omega)$ and let
$\ce>0$ satisfy \eqref{distaminima}.
For every $j=1,\ldots,J$ let $f^j$ and $a^j_\ce$ be two functions with $f^j\in C^\infty(B_{\frac D 2}(x^j))$ and $a^j_\ce$ affine.
Moreover, let $f$ be a function which is smooth in a neighborhood of $\partial\Omega$ and $u\in H^2(\Omega_{\ce}(\alpha))$ be such that $u=f$ and $\partial_n u=\partial_nf$ on $\partial\Omega$ and $u=a_\ce^j+f^j$ and $\partial_nu=\partial_na_\ce^j+\partial_n f^j$ on $\partial B_\ce(x^j)$ for every $j=1,\ldots,J$\,.
 Then the function $\widehat u\colon\Omega\to\R$ defined by
$$
\widehat u(x)\coloneqq\begin{cases}
u(x) & \text{if $x\in\Omega_\ce(\alpha)$}\\
a^j_\ce(x)+f^j & \text{if $x\in B_{\ce}(x^j)$}
\end{cases}$$
is in $H^2(\Omega)$ and satisfies
 \begin{equation*}
\|\widehat{u}\|_{H^2(\Omega)}
\le C \Big(\|\nabla^2 u\|_{L^2(\Omega_\ce(\alpha);\R^{2\times 2})}+\|f\|_{C^\infty(\partial\Omega)}+\sum_{j=1}^J\|f^j\|_{C^\infty(B_{\frac {D}{2}}(x^j))}\Big)\,,
\end{equation*}
for some constant $C$ independent of $u$ and of $\ce$\,.
\end{proposition}
\begin{proof}
By assumption and by Proposition \ref{PropA6CL}, we have
\begin{equation*}
\sum_{j=1}^J\|a^j_\ce+f^j\|^2_{H^1(\partial B_\ce(x^j))}\le JC_2(\Omega,\alpha)\Big(\|\nabla^2 u\|^2_{L^2(\Omega_\ce(\alpha);\R^{2\times 2})}+\|f\|^2_{C^\infty(\partial\Omega)} \Big)\,,
\end{equation*}
which implies, in particular, 
\begin{equation*}
\sum_{j=1}^J\|a^j_\ce\|^2_{H^1(\partial B_\ce(x^j))}\le JC_2(\Omega,\alpha)\Big(\|\nabla^2 u\|^2_{L^2(\Omega_\ce(\alpha);\R^{2\times 2})}+\|f\|^2_{C^\infty(\partial\Omega)}+\sum_{j=1}^J\|f^j\|^2_{H^1(\partial B_\ce(x^j))}\Big)\,.
\end{equation*}
Since $a_\ce^j$ is affine, this implies that, for every $j=1,\ldots,J$\,,
\begin{equation*}
\|a^j_\ep\|^2_{H^1(B_\ep(x^j))}\le JC_2(\Omega,\alpha)\ep\Big(\|\nabla^2 u\|^2_{L^2(\Omega_\ce(\alpha);\R^{2\times 2})}+\|f\|^2_{C^\infty(\partial\Omega)}+\sum_{j=1}^J\|f^j\|^2_{H^1(\partial B_\ce(x^j))}\Big)\,,
\end{equation*}
which immediately provides the claim.
\end{proof}


\section{A density result for traction-free $H^2$ functions}\label{sec:density}
In this appendix we prove that, given $\alpha=\sum_{j=1}^Jb^j\de_{x^j}\in\ED(\Omega)$\,, any function $w\in \widetilde{\Bb}_{0,\Omega}^\alpha$ (see \eqref{tildeBzero}) can be approximated in the strong $H^2$ norm by a sequence of functions $w_\ce\in\widetilde{\Bb}_{\ce,\Omega}^\alpha$ (see \eqref{tildeBdelta}). 
The rough idea is (up to modifying the boundary datum) to replace $w+W_0^\alpha$ (see Remark \ref{maybeuseful})\,, with its first-order Taylor expansion in $B_{\ce}(x^j)$ ($j=1,\ldots,J$)\,.
We highlight that $W_0^\alpha$ is not even in $H^2(\Omega)$ but, in view of Remark \ref{maybeuseful}, it is the strong $H^2_\loc$ limit of $W_\ce^\alpha:=\sum_{j=1}^JW^j_\ep$\,, where $W^j_\ce$ is affine in $B_\ep(x^j)$ and smooth in $B_\ce(x^i)$ with $i\neq j$\,. This allows us to approximate $W_0^\alpha$ 
in the desired manner.
Then we approximate $w$ by a sequence $\{v_k\}_k$ of smooth functions and we apply Taylor's formula with Lagrange remainder to further approximate each $v_k$ by a sequence $\{v_{k,\ep}\}_\ce$ that is affine in $\bigcup_{j=1}^JB_\ce(x^j)$\,.
Finally, the claim is obtained by summing $v_{k,\ep}$ 
to the contribution approximating $W_0^\alpha$\,, and by using a diagonal argument.


\begin{proposition}\label{prop:approx}
Let $\alpha\in\ED(\Omega)$\,. For every $w\in\widetilde{\Bb}_{0,\Omega}^\alpha$ 
 there exists a sequence $\{w_\ce\}_\ce\subset H^2(\Omega)$ with $w_\ce\in\widetilde{\Bb}_{\ce,\Omega}^\alpha$ 
for $\ce>0$ small enough, such that $w_\ce\to w$ strongly in $H^2(\Omega)$ as $\ce\to 0$\,.
\end{proposition}
\begin{proof}
By standard density arguments, there exists a sequence $\{v_k\}_{k\in\N}\subset C^\infty(\Omega)$ such that $v_k\to w$ strongly in $H^2(\Omega)$ as $k \to \infty$\,. Furthermore, we can assume that there exists a sequence $\{\delta_k\}_{k\in\N}$ with $\delta_k\to 0$ as $k\to \infty$\,, such that
 $v_k\equiv -W_0^\alpha$ in a $\delta_k$-neighborhood of $\partial \Omega$\,.
 
Let $\supp\alpha=\{x^1,\ldots, x^J\}$\,.
For every $j=1,\ldots,J$ and for every $k\in\N$\,, we set
$$
\hat v^j_{k}(x)\coloneqq v_{k}(x^j)+\langle\nabla v_{k}(x^j),x-x^j\rangle\qquad\textrm{for every }x\in\R^2\,;
$$
moreover, we consider a $C^2$ function $\gamma\colon [1,2]\to [0,1]$ with $\gamma(1)=0$\,, $\gamma(2)=1$\,, ${\gamma}'\ge 0$ in $(1,2)$\,, $\gamma'_+(1)=0=\gamma'_-(2)$ and $\gamma''_+(1)=0=\gamma''_-(2)$\,.
For every $0<\ce<\frac{1}{2}\min\{D,\delta_k\}$ (with $D$ defined in \eqref{distaminima}) we define the function
 $v_{k,\ce}\colon\Omega\to \R$ as
\begin{equation*}
v_{k,\ce}(x)\coloneqq
\begin{cases}
\hat v^j_{k}(x)&\textrm{if }x\in B_\ce(x^j)\\[2mm]
\displaystyle\Big(1-\gamma\Big(\frac{|x-x^j|}{\ce}\Big)\Big)\hat v^j_{k}(x)+\gamma\Big(\frac{|x-x^j|}{\ce}\Big)v_{k}(x)&\textrm{in }A_{\ce,2\ce}(x^j)\\[2mm]
v_{k}(x)&\textrm{if }x\in\Omega_{2\ce}(\alpha)\,.
\end{cases}
\end{equation*} 
Notice that, since $\ce<\frac{\delta_k}{2}$\,, we have that  $v_{k,\ce}$
coincide with $-W^\alpha_0$ in a $\frac{\delta_k}{2}$-neighborhood of $\partial\Omega$\,.
We claim that,
for every $k\in\N$, 
\begin{equation}\label{confin}
\|v_{k,\ce}-v_{k}\|_{H^2(\Omega)}=\sum_{j=1}^J\|v_{k,\ce}-v_{k}\|_{H^2(B_{2\ce}(x^j))}\to 0\qquad\textrm{as }\ce\to 0\,.
\end{equation}
To this end, we prove that for every $j=1,\ldots,J$
\begin{equation}\label{confin1}
\|v_{k,\ce}-v_{k}\|_{H^2(B_{2\ce}(x^j))}\le C\ce\|\nabla^2v_{k}\|_{L^\infty(\Omega;\R^{2\times 2})}\,,
\end{equation}
for some universal constant $C$ independent of $k$ and $\ce$\,.
Indeed, fix $j=1,\ldots,J$\,.
By the Taylor expansion formula with Lagrange remainder, 
we have that
\begin{eqnarray}\label{confin1a}
\|\hat v^j_{k}-v_{k}\|^2_{L^2(B_{2\ce}(x^j))}&=&\frac 1 4\int_{B_{2\ce}(x^j)}|\langle\nabla^2 v_{k}(\xi_x^j)(x-x^j), x-x^j\rangle|^2\,\ud x\\\nonumber
&\le& C\ce^6\|\nabla^2 v_{k}\|^2_{L^\infty(\Omega;\R^{2\times 2})}\,,\\ \label{confin1b}
\|\nabla\hat v^j_{k}-\nabla v_{k}\|^2_{L^2(B_{2\ce}(x^j);\R^{2})}&=&\int_{B_{2\ce}(x^j)}|\nabla v_{k}(x^j)-\nabla v_{k}(x)|^2\,\ud x\\ \nonumber
&\le& C\ce^4\|\nabla^2 v_{k}\|^2_{L^\infty(\Omega;\R^{2\times 2})}\,,\\  \label{confin1c}
\|\nabla^2\hat v^j_{k}-\nabla^2 v_{k}\|^2_{L^2(B_{2\ce}(x^j);\R^{2\times 2})}&=&\|\nabla^2 v_{k}\|^2_{L^2(B_{2\ce}(x^j);\R^{2\times 2})}\\ \nonumber
&\le& C\ce^2\|\nabla^2 v_{k}\|^2_{L^\infty(\Omega;\R^{2\times 2})}\,,
\end{eqnarray}
where in \eqref{confin1a} $\xi_x^j$ is a point in the segment joining $x^j$ and $x$\,.
Furthermore, since 
$$
\Big\|\nabla\gamma\Big(\frac{|\cdot|}{\ce}\Big)\Big\|_{L^\infty(A_{\ce,2\ce}(0);\R^2)}\le\frac{C}{\ce}\,,\qquad \Big\|\nabla^2\gamma\Big(\frac{|\cdot|}{\ce}\Big)\Big\|_{L^\infty(A_{\ce,2\ce}(0);\R^{2\times 2})}\le\frac{C}{\ce^2}\,,
$$
by \eqref{confin1a},  \eqref{confin1b}, and \eqref{confin1c}, we deduce that
\begin{equation*}
\begin{aligned}
\|v_{k,\ce}-v_{k}\|^2_{L^2(A_{\ce,2\ce}(x^j))}\le&\,\|\hat v^j_{k}-v_{k}\|^2_{L^2(A_{\ce,2\ce}(x^j))}\\
\le&\, C\ce^6\|\nabla^2 v_{k}\|^2_{L^\infty(\Omega;\R^{2\times 2})}\,,\\
\|\nabla v_{k,\ce}-\nabla v_{k}\|^2_{L^2(A_{\ce,2\ce}(x^j);\R^2)}\le&\,\frac{C}{\ce^2}\|\hat v^j_{k}-v_{k}\|^2_{L^2(B_{2\ce}(x^j))}\\
&\,+C\|\nabla\hat v^j_{k}-\nabla v_{k}\|^2_{L^2(B_{2\ce}(x^j);\R^2)}\\
\le&\, C\ce^4\|\nabla^2 v_{k}\|_{L^\infty(\Omega;\R^{2\times 2})}\,,\\\
\|\nabla^2 v_{k,\ce}-\nabla^2v_{k}\|^2_{L^2(A_{\ce,2\ce}(x^j);\R^{2\times 2})}\le&\,\frac{C}{\ce^4}\|\hat v^j_{k}-v_{k}\|^2_{L^2(B_{2\ce}(x^j))}\\
&\,+\frac{C}{\ce^2}\|\nabla\hat v^j_{k}-\nabla v_{k}\|^2_{L^2(B_{2\ce}(x^j);\R^2)}\\
&\,+C\|\nabla^2\hat v^j_{k}-\nabla^2 v_{k}\|^2_{L^2(B_{2\ce}(x^j);\R^{2\times 2})}\\ \nonumber
\le&\, C\ce^2\|\nabla^2v_{k}\|^2_{L^\infty(\Omega;\R^{2\times 2})}\,,
\end{aligned}
\end{equation*}
whence we get that
$$
\|v_{k,\ce}-v_{k}\|_{H^2(A_{\ce,2\ce}(x^j))}\le C\ce\|\nabla^2v_{k}\|_{L^\infty(\Omega;\R^{2\times 2})}\,;
$$
{this fact, together with \eqref{confin1a}, \eqref{confin1b} and \eqref{confin1c}, implies \eqref{confin1} and hence \eqref{confin}}.
Moreover, up to using a cut-off function, in view of Remark \ref{maybeuseful}, we can assume that $v_{k,\ce}\equiv -W_{\ce}^\alpha$ in an $\ce$-neighborhood of $\partial\Omega$\,, so that the boundary condition in the definition of $\widetilde{\Bb}_{\ce,\Omega}^\alpha$ in \eqref{tildeBdelta} is satisfied.

To recover the traction-free condition on each $\partial B_{\ce}(x^j)$, we notice that each function $W_{\ce}^j$ (defined in \eqref{20220223_1}) is affine in $\overline{B}_{\ce}(x^j)$, whereas it is smooth in $\bigcup_{i\neq j}B_{\ce}(x^i)$\,.
Therefore, for every $j=1,\ldots,J$ we define the function
 $\widehat W_\ce^{\neq j}\colon B_D(x^j)\to \R$ as the affine contribution of all of the $W_{\ce}^i$ for $i\neq j$, \emph{i.e.}, 
$$
\widehat W_\ce^{\neq j}(x)\coloneqq \sum_{i\neq j}\Big(W_\ce^i(x^j)+\langle\nabla W_\ce^i(x^j),x-x^j\rangle\Big)\,.
$$
Now, we define the function $\overline{W}^\alpha_{\ce}\colon\Omega\to \R$ as
\begin{equation*}
\overline{W}^\alpha_{\ce}(x)\coloneqq
\begin{cases}
\widehat W_\ce^{\neq j}(x)+W_\ce^j(x)-W_\ce^\alpha(x)&\textrm{if }x\in B_\ce(x^j) \\[2mm]
\displaystyle \Big(1-\gamma\Big(\frac{|x-x^j|}{\ce}\Big)\Big)\big(\widehat W^{\neq j}_{\ce}(x)+W_\ce^j(x)-W_\ce^\alpha(x)\big)&\textrm{if }x\in A_{\ce,2\ce}(x^j)\\[2mm]
0&\textrm{if }x\in\Omega_{2\ce}(\alpha)\,.
\end{cases}
\end{equation*}
By the very definition of $W_\ce^\alpha$ (see \eqref{20220223_1} again) it is easy to check that
\begin{equation*}
\|\overline{W}_\ce^\alpha\|_{H^2(\Omega)}\to 0\qquad\textrm{as }\ce\to 0\,.
\end{equation*}
For every $k$ and $\ce$ as above, we define $w_{k,\ce}\colon\Omega\to\R$ as
$w_{k,\ce}\coloneqq v_{k,\ce}+\overline{W}_\ce^\alpha$\,, and we notice that it belongs to $\widetilde\Bb_{\ce,\Omega}^\alpha$ by construction.
Therefore, by a standard diagonal argument, there exists a sequence $\{w_\ce\}_{\ce}$ with $w_{\ce}=w_{k(\ce),\ce}$ satisfying the desired properties.  
\end{proof}

\bibliographystyle{plain} 
\bibliography{refsPatrick.bib}

\end{document}